\documentclass[11pt]{amsart}
\usepackage[utf8]{inputenc}

\usepackage[style=alphabetic,maxbibnames=9]{biblatex}
 \pdfoutput=1

\usepackage{amsmath}%
\usepackage{amsfonts}%
\usepackage{amssymb}
\usepackage{amsthm}
\usepackage{bm}
\usepackage{graphicx}
\usepackage{stmaryrd}
\usepackage{centernot}
\usepackage{wrapfig}
\usepackage{mathtools}

\usepackage{tikz}
\usepackage{tikz-cd}
\usetikzlibrary{arrows}

\usepackage{url}
\usepackage{geometry}
\usepackage{enumerate}

\usepackage{latexsym}

\DeclareFontFamily{U}{mathb}{\hyphenchar\font45}
\DeclareFontShape{U}{mathb}{m}{n}{
      <5> <6> <7> <8> <9> <10> gen * mathb
      <10.95> mathb10 <12> <14.4> <17.28> <20.74> <24.88> mathb12
      }{}
\DeclareSymbolFont{mathb}{U}{mathb}{m}{n}
\DeclareFontSubstitution{U}{mathb}{m}{n}

\DeclareMathSymbol{\sqsubsetneq}{3}{mathb}{'210}

\usepackage{xspace}
\usepackage{xcolor}
\usepackage{ifthen}
\newcommand{\showcomments}{false}

\newcommand{\aris}[1]%
{\ifthenelse{\equal{\showcomments}{true}}%
{{\color{red}{\small \textbf{Aristotelis says:} #1}}}{\xspace}}%

\newcommand{\shaun}[1]%
{\ifthenelse{\equal{\showcomments}{true}}%
{{\color{blue}{\small \textbf{Shaun:} #1}}}{\xspace}}%

\theoremstyle{plain}
\newtheorem{theorem}{Theorem}[section]
\newtheorem{corollary}[theorem]{Corollary}
\newtheorem{lemma}[theorem]{Lemma}

\newtheorem{claim}[theorem]{Claim}
\newtheorem{proposition}[theorem]{Proposition}
\newtheorem{problem}{Problem}

\theoremstyle{definition}
\newtheorem{definition}[theorem]{Definition}
\newtheorem{example}[theorem]{Example}
\newtheorem{remark}[theorem]{Remark}

\newcommand{\mcB}{\mathcal{B}}
\newcommand{\mcC}{\mathcal{C}}
\newcommand{\mcL}{\mathcal{L}}
\newcommand{\mcM}{\mathcal{M}}
\newcommand{\mcN}{\mathcal{N}}
\newcommand{\mcO}{\mathcal{O}}

\newcommand{\BPi}{\mathbf{\Pi}}
\newcommand{\BSigma}{\mathbf{\Sigma}}

\newcommand{\smf}{\smallfrown}

\newcommand{\PG}{\mathcal{PG}}

\DeclareMathOperator{\Aut}{\text{Aut}}
\DeclareMathOperator{\rk}{\mathrm{rk}}
\DeclareMathOperator{\Drk}{\mathrm{Drk}}

\DeclareMathOperator{\CBrk}{\mathrm{CBrk}}
\DeclareMathOperator{\Stab}{\mathrm{Stab}}

\newcommand{\Grk}{\textrm{Grk}}

\newcommand{\OPEN}{\mathrm{OPEN}}

\title[The class of {\large $\alpha$}-balanced Polish groups and their dynamics]{The class and dynamics of {\LARGE$\alpha$}-balanced Polish groups}

\author{Shaun Allison}
\address{Department of Mathematics, University of Toronto, Toronto, ON, M5S 2E4}
\email{shaunpallison@gmail.com}

\author{Aristotelis Panagiotopoulos}
\address{Kurt G\"odel Research Center, Faculty of Mathematics,   Universit\"at Wien, Kolingasse 14-16, 1090 Vienna, Austria}
\email{aristotelis.panagiotopoulos@gmail.com}

\thanks{
This research was supported by the
NSF Grant DMS-2154258: ``Dynamics Beyond Turbulence and Obstructions to Classification''
and Israel Science Foundation (ISF) grant no.  1832/19
.}

\begin{document}

\begin{abstract}
For each  ordinal $\alpha<\omega_1$, we introduce the class of $\alpha$-balanced Polish groups. These classes form a hierarchy which  stratifies the space between
the class of Polish groups admitting a two-side-invariant  metric (TSI) and the class of Polish groups admitting a complete left-invariant metric (CLI). We  establish various closure properties, provide connections to model theory, and  we develop a boundedness principle for CLI groups by showing that $\alpha$-balancedness is an initial segment of a regular coanalytic rank.  

In the spirit of Hjorth's turbulence theory we also  introduce ``generic $\alpha$-unbalancedness": a new dynamical condition for Polish $G$-spaces which serves as an obstruction to classification by actions of $\alpha$-balanced Polish groups. We use this to provide, for  each $\alpha<\omega_1$,  an action of an $\alpha$-balanced Polish group whose orbit equivalence relation is strongly generically ergodic against actions of any $\beta$-balanced Polish group with $\beta<\alpha$.
\end{abstract}

\maketitle

\section{Introduction}

A Polish group is a separable topological group $G$ which admits  a complete metric compatible with its topology. A metric $d$ on a Polish group $G$ is   {\bf left-invariant} if $d(gh,gh')=d(h,h')$ holds for all $g,h,h'\in G$ and it is {\bf two-side-invariant} if $d(gh,gh')=d(h,h')=d(hg,h'g)$ holds for all $g,h,h'\in G$. By Birkhoff-Kakutani \cite{birkhoff1936note,kakutani1936metrisation}, every Polish group admits a left-invariant metric that is compatible with its topology.  However, such metric cannot always be  taken to be  complete or two-side-invariant. This carves out two important classes of Polish groups with special properties. A Polish group is {\bf CLI} if it admits a  compatible metric  that is both complete and left-invariant, and it is {\bf TSI} if it admits a compatible metric  that is two-side-invariant. It turns out that every TSI group is CLI but there are several interesting examples of CLI groups which are not TSI.

The class of TSI groups can be thought of as a common generalization of the classes of Polish groups which are compact, discrete, or abelian. By a theorem of Klee, TSI groups are precisely those Polish groups  admitting  a conjugation invariant basis of open neighborhood of the identity \cite{Klee} --- a property of groups known as \emph{balanced} or \emph{small invariant neighborhood} (SIN).  Already in the realm of locally-compact groups, the TSI property has been extensively  studied in relation to  properties such us unimodularity and inner amenability \cite{TSI4,TSI5,TSI6,TSI7}. It is, however, in the realm of non-locally-compact Polish groups --- where the lack of Haar-measure 
prompts the search for additional structure --- that the TSI property has
played a central role. Namely, from  establishing new automatic continuity phenomena \cite{TSI2,TSI1,TSI3}, to extending the proof of the topological Vaught conjecture from  abelian \cite{sami} to all TSI Polish groups \cite{HSTSI}.

Similarly, the class of all CLI groups can be  thought of as a common generalization of  the classes of Polish groups which are locally compact, solvable, or TSI. By a theorem of Gao \cite{Gao1998},  later generalized to all metric structures \cite{yaacov2017metric}, the automorphism group $\mathrm{Aut}(M)$ of a countable structure $M$ is CLI   if and only if $M$ admits proper  $\mathcal{L}_{\omega_1\omega}$-embeddings into itself. 
The CLI property has also been used in generalizing various dynamical phenomena such as Glimm-Effros dichotomy and topological Vaught conjecture from the realm of locally-compact, nilpotent, or TSI groups \cite{glimm, HSTSI}, to the realm of all CLI groups \cite{Becker1}.  In contrast to TSI groups which are characterized by the SIN property, CLI groups do not admit a ``simple" characterization in terms of open neighborhoods of the identity, as they turn out to form a coanalytic but non-Borel collection of Polish groups \cite{Malicki2011}.

The main goal of this paper is to introduce for each countable ordinal $\alpha>0$ the class of \emph{$\alpha$-balanced Polish groups}: a new class of Polish groups  which interpolates between the classes of TSI and CLI groups  and which admits its own robust structure and dynamics.

\bigskip{}

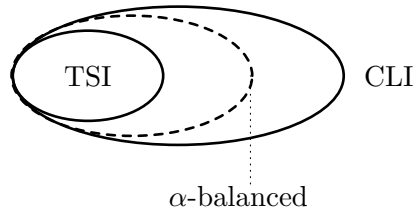
\begin{center}
\begin{tikzpicture}[line cap=round,line join=round,>=triangle 45,x=1cm,y=1cm,scale=0.8]
\draw [rotate around={0:(0.5,0)},line width=1pt] (0.5,0) ellipse (2.75cm and 1.15cm);
\draw [rotate around={0:(-1,0)},line width=1pt] (-1,0) ellipse (1.25cm and 0.75cm);
\draw [rotate around={0:(-0.27,0)},line width=1pt,dashed] (-0.27,0) ellipse (2cm and 1cm);
\draw (-1,0) node {TSI};
\draw (4,0) node {CLI};
\draw (1.5,-2) node {$\alpha$-balanced};
\draw[dotted] (1.7,-1.8) -- (1.7,0) ;
\end{tikzpicture}
\end{center}

 In short,  $\alpha$-balanced  Polish groups are defined as follows, where the notation  $V\subseteq_1 G$  will stand throughout the paper for ``$V$ is an open neighborhood of the identity in $G$".

\begin{definition}\label{Def:alpha-balanced}
A Polish group $G$ is {\bf $\alpha$-balanced} if and only if $\rk(G)\leq \alpha$, where
\[\rk(G):=\sup \{ \rk(V,G)+1 \colon V\subseteq_1 G \},\]
 and for every ordinal $\beta$ and every $V,U\subseteq_1 G$  we  define inductively:

 \noindent  $(a)$  \, $\rk(V,U)=0$ if $U\subseteq V$;

\noindent  $(b)$ \, $\rk(V,U)\leq \beta$ if there is $W\subseteq_1 G$ so that for all $g\in U$ we have   $\rk(V, gWg^{-1}) < \beta$;

\noindent  $(c)$ \, $\rk(V,U)=\infty$ if there is no ordinal $\beta$ so that  $\rk(V,U)\leq \beta$;
\end{definition}

This paper consists of two parts, each  of which is summarized in the remainder of this introduction. Part \ref{PartI}  develops the  topological group theory of $\alpha$-balanced Polish groups  and  fleshes out  connections with  model theory and descriptive set theory.  In Part II this theory yields several consequences of  $\alpha$-balancedness  for  dynamics of Polish group actions. The main takeaway is that $\alpha$-balanced Polish groups form a class of Polish groups with structure robust enough to leave a trace on the orbit equivalence relations $E^G_X$ of the actions $G \curvearrowright X$ of its members $G$. This structure is exploited to identify strong ergodicity phenomena and for carving out a new ladder of complexity classes in the Borel reduction hierarchy.

\subsection{Part I. The class of $\alpha$-balanced Polish groups} It is easy to see that the only $1$-balanced group is the trivial group and that the class of  $2$-balanced groups is precisely the class of all TSI groups; see Section \ref{S:alpha_balanced}.
Our first main result is that the classes of $\alpha$-balanced Polish groups, for $1<\alpha<\omega_1$, completely stratify the space between TSI and CLI:

\begin{theorem}\label{thm:CLI}
For every Polish group $G$ we have that:
\[G \text{ is } \mathrm{CLI} \;  \iff \;  \rk(G)<\infty \;  \iff \; \rk(G)<\omega_1\]
Moreover, for every $\alpha<\omega_1$ there exists a Polish group $G$ with $\rk(G)=\alpha$.
\end{theorem}

In analogy with Klee's characterization of TSI groups as Polish groups admitting a conjugation invariant basis of $V\subseteq_1 G$, Theorem \ref{thm:CLI} can be thought of as the simplest possible characterization of CLI groups in terms of open neighborhoods of the identity. It also provides an inductive procedure
for establishing properties of the class of  CLI groups.

Definition \ref{Def:alpha-balanced}  has been inspired by dynamical phenomena occurring in actions of  non-TSI groups \cite{AP2021}, as well as by  the ranks that Deissler  and Malicki developed in the context of  countable $\mathcal{L}$-structures \cite{Deissler1997}  and  Polish permutation groups  \cite{Malicki2011}, respectively.  In Section \ref{S:Deissler} we show that for automorphism groups $\mathrm{Aut}(M)$ of countable structures the rank from Definition \ref{Def:alpha-balanced}  corresponds to a slight modification of the Deissler rank. This allows us to import model-theoretic intuition into the study of $\alpha$-balanced Polish groups. In Section \ref{S:Malicki} we  show how a weakening of Definition \ref{Def:alpha-balanced}, which leads to the notion of \emph{weakly $\alpha$-balanced Polish groups}, relates to the rank Malicki developed for Polish  permutation groups.

One of the main results  from \cite{Malicki2011} is that the class of all CLI groups forms a coanalytic non-Borel set in the standard Borel space  $\mathcal{PG}$ of all Polish groups. The proof relied on exhibiting a collection of permutation group of unbounded \emph{Malicki rank}. A similar  collection of structures with unbounded \emph{Deissler rank} is used in \cite[Corollary 2.5]{Deissler1997}. 
Here we  provide a ``uniformly defined" such collection and use it to strengthen the above:
   
\begin{theorem}\label{T:CompleteCoanalytic}
    The class of CLI Polish groups forms a complete coanalytic set.
\end{theorem}

Finally, in  Section \ref{S:Boundedness} we  show that the rank from Definition \ref{Def:alpha-balanced} is a regular $\BPi^1_1$-rank on the coanalytic subset CLI of  $\mathcal{PG}$.  One consequence is the  following boundedness principle:

\begin{corollary}\label{cor:Analytic>Bounded}
If  $\mathcal{A}$ is a class of CLI  Polish groups which admits analytic definition, then there exists a countable ordinal  
 $\alpha=\alpha(\mathcal{A})$ so that every  $G$ in $\mathcal{A}$ is $\alpha$-balanced.
\end{corollary}

This can 
be used to
justify on abstract grounds why various classes of Polish groups, such as the class of locally compact Polish groups,  consist of $\alpha$-balanced groups for bounded $\alpha$.
 
\subsection{Part II: the dynamics of $\alpha$-balanced Polish groups}
One of the major ongoing research programs of invariant descriptive set theory seeks to identify the
intrinsic complexity of various classification problems by organizing them according to their relative complexity.
Such classification problems  can often be represented as  \emph{orbit equivalence relations} of a continuous action of a Polish group on a Polish space. 
This creates a deep link between the complexity of classification problems and topological dynamics.
Our inability to fully classify some mathematical objects up to some notion of equivalence using simple enough invariants can often be traced to dynamics.

Formally, a {\bf classification problem} is a pair $(X,E)$ where $X$ is a Polish space and $E$ is analytic equivalence relation on $X$.
We consider the classification problem $(X,E)$ to be of lower or equal complexity than $(Y,F)$, denoted by  $(X,E)\leq_B (Y,F)$ or simply $E\leq_B F$, if there exist a {\bf Borel reduction} from $E$ to $F$.
A Borel reduction is a Borel map $f\colon X\to Y$ so that $x \mathrel{E} y  \iff f(x) \mathrel{F} f(y)$ for all $x,y\in X$.
Given such $(X,E)$, one is interested in finding  the simplest type of invariants 
$Y/F$  needed for classifying $(X,E)$ in the above sense. Some classical complexity classes in the lower part of the Borel reduction hierarchy consist of  the \emph{concretely classifiable} and the \emph{classifiable by countable structures} classification problems. These correspond to  those $(X,E)$ which can be classified using real numbers or isomorphism types of countable structures as invariants, respectively.  From an dynamical point of view they are precisely those $(X,E)$ which are \emph{classifiable by compact groups} or \emph{classifiable by non-archimedean groups}, respectively, in the following formal sense.
A {\bf Polish $G$-space} is a Polish space $X$ together with a continuous action $G\curvearrowright X$. This induces a classification problem $(X,E^G_X)$, where  $E^G_X$ is the associated {\bf orbit equivalence relation} on $X$ given by 
$x \mathrel{E^G_X} x'$ if and only if $gx=x'$ for some $g \in G$. 
We say that {\bf $E$ is classifiable by $G$} if there is an orbit equivalence relation $E^G_Y$ such that $E \le_B E^G_Y$.
Given a class $\mathcal{C}$ of Polish groups, we say that $E$ is {\bf classifiable by a $\mathcal{C}$ group} iff there is a Polish group in $\mathcal{C}$ which classifies $E$.

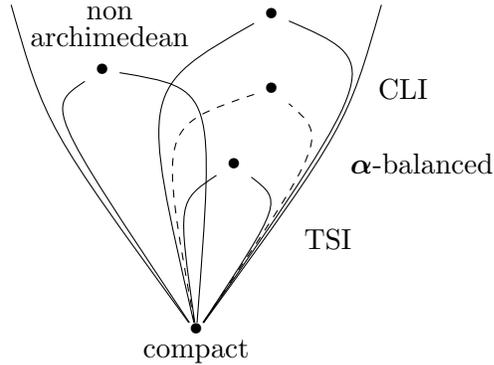
\begin{figure}[ht!]    
\begin{tikzpicture}[scale=0.5]
\node (-1) at (10, 10.5) {};
\node (+1) at (20, 10.5) {};
\node (b1) at (15, 1.5) {};
\draw (b1) .. controls (11,7) .. (-1);
\draw (b1) .. controls (19,7) .. (+1);
\node (b'1) at (15, 1.6) {$\bullet$};
\node (b''1) at (15, 1) {compact};

\node (d''1) at (12.7, 10) {non};
\node (d'1) at (12.7, 9.4) {archimedean};
\node (d1) at (12.5, 8.5) {$\bullet$};
\draw (b1) .. controls (11,7.5) .. (d1);
\draw (b1) .. controls (15.5,8) .. (d1);

\node (e'1) at (18.5,4) {TSI};
\node (e1) at (16, 6) {$\bullet$};
\draw (b1) .. controls (17.3,5) .. (e1);
\draw (b1) .. controls (14.5,5) .. (e1);
\node (e'1) at (21, 6) { $\boldsymbol{\alpha}$-balanced};
\node (e1) at (17, 8) {$\bullet$};
\draw[dashed] (b1) .. controls (18.5,6.7) .. (e1);
\draw[dashed] (b1) .. controls (14,7) .. (e1);
\node (kk1) at (20.5, 8) {CLI};
\node (k1) at (17, 10) {$\bullet$};
\draw (b1) .. controls (20,8.7) .. (k1);
\draw (b1) .. controls (13.5,8) .. (k1);
\end{tikzpicture}
\end{figure}

This dynamical interpretation of the Borel reduction hierarchy underlies  many negative anti-classification results. 
 Indeed,  \emph{generic ergodicity} is a classical dynamical condition for  $G\curvearrowright X$, going all the way back to \cite{Mackey+, Glimm++, Effros+}, which prevents $E^G_X$ from being classifiable by compact Polish groups. Similarly,  Hjorth's celebrated  \emph{theory of turbulence} \cite{Hjorth2000} provides  a much ``wilder" dynamical condition, which precludes   $E^G_X$  even from being classifiable by non-archimedean Polish groups. More recently a dynamical obstruction to classification by CLI groups was developed \cite{LupiniPanagio2018} as an alternative to the forcing theoretic notion of pinnedness  \cite{Kanovei2008, LarsonZapletal2020}. Finally,  in the precursor \cite{AP2021} to this current paper, the authors introduced \emph{generic unbalancedness}:
a dynamical property for  Polish $G$-spaces  $G\curvearrowright X$ which   precludes   $E^G_X$ from being  classifiable by TSI groups. 

Our ultimate goal  in this paper is to show that  {\bf classification by $\alpha$-balanced groups} forms a strictly increasing hierarchy of complexity classes in the Borel reduction hierarchy. In doing so we first introduce the following dynamical condition which turns out to entail \emph{strong generic ergodicity} properties against actions of  $\alpha$-balanced groups.

\begin{definition}\label{Def:Warrow}\label{D:GenericUnbalancedness}
Let $G\curvearrowright X$ be a Polish $G$-space and let  $V\subseteq_1 G$.
We recursively define binary relations $x \leftrightsquigarrow^{\alpha}_V y$ on $X$ for any countable ordinal $\alpha$ as follows.
\begin{enumerate}
\item   $x \leftrightsquigarrow^{0}_V y$   iff both $x \in \overline{V \cdot y}$ and $y \in \overline{V \cdot x}$ hold.
  \item  $x\leftrightsquigarrow^{\alpha}_V y$  iff  for every open $W\subseteq_1 G$ and  $U\subseteq X$, with either $x\in U$ or $y\in U$,   there exist   $g^x,g^y\in V$ with $g^xx,g^yy\in U$ so that   for every $\beta<\alpha$, $g^x x \leftrightsquigarrow^{\beta}_W  g^y y$ holds . 
\end{enumerate}  
We say that  $G\curvearrowright X$  is {\bf generically $\alpha$-unbalanced}  for some $\alpha>0$,
if   for every comeager $C\subseteq X$ there is a comeager $D\subseteq C$ so that for all $x,y\in D$  there exists a  finite sequence $x_0,\ldots, x_n\in C$  with $x_0=x, x_n =y$, so that  for all $i<n$ and  $\beta<\alpha$ we have  $x_i\leftrightsquigarrow^{\beta}_G x_{i+1}$.
\end{definition}

Let $\mathcal{C}$ be a class of Polish groups. A classification problem $(X,E)$  is  {\bf  generically ergodic against $\mathcal{C}$ groups} if for every Polish $H$-space $Y$, with  $H\in\mathcal{C}$, and any Baire-measurable map  $f\colon X\to Y$ with $x \mathrel{E} y \implies f(x) \mathrel{E^H_Y} f(y)$ there exists a comeager $C\subseteq X$ with $f(x) \mathrel{E^H_Y} f(y)$ for all $x,y\in C$. The following are  the  main results of Part II of our paper.

\begin{theorem}\label{thm:main2}
Let $\alpha$ be a countable ordinal. If $G\curvearrowright X$ is generically $\alpha$-unbalanced   Polish $G$-space, then $E^G_X$  is   generically ergodic against $\alpha$-balanced Polish groups.
\end{theorem}

\begin{corollary}\label{Cor:1}
Let $G\curvearrowright X$ be a generically $\alpha$-unbalanced   Polish $G$-space. If   $G\curvearrowright X$ has meager orbits, then  $(X,E^G_X)$ is not classifiable by actions of   $\alpha$-balanced Polish groups.
\end{corollary}

The proof of Theorem \ref{thm:main2}  combines the main arguments from  Hjorth's turbulence theorem and  a transfinite change of topology argument  based on a variant of  Scott-Hjorth analysis of orbits, following the general lines of the proof of
\cite[Theorem 1.3]{AP2021}. 

In order to show that ``classification by $\alpha$-balanced groups" forms a strictly increasing hierarchy  of complexity classes, it suffices  to exhibit for each $\alpha<\omega_1$ an action of an $(\alpha+1)$-balanced Polish group satisfying Corollary \ref{Cor:1}. We show that the Bernoulli shift of the automorphism group a certain  $\alpha$-scattered linear order $\mathbb{Z}[\alpha]$ has these properties.

\begin{theorem}\label{T:mainBernoulli}
The Bernoulli shift  of $\mathrm{Aut}(\mathbb{Z}[\alpha])$ is generically $\alpha$-unbalanced, for all $\alpha<\omega_1$.
\end{theorem}
 
The proof of Theorem \ref{T:mainBernoulli} turns out to be surprisingly elaborate,  as a naive transfinite induction based on successor and limit stages does not seem to apply. Indeed, while some basic theory of the \emph{wreath-product  jump}   $(G\curvearrowright X) \mapsto \big((\mathbb{Z} \mathrm{Wr} G) \curvearrowright X^{\mathbb{Z}}\big)$ from \cite{AP2021} can be used to propagate  the induction from $\alpha$ to $\alpha+1$ in the statement of Theorem \ref{T:mainBernoulli}, one cannot simply thread these arguments together to deal with limit stages. It turns out that one needs instead to induct  on the length $m$ of the Cantor normal form of $\alpha$ 
\[
\alpha = \omega^{\lambda_m} + \omega^{\lambda_{m-1}}+\cdots  +\omega^{\lambda_1}, \quad \quad \lambda_m\geq \lambda_{m-1}\geq \ldots\geq \lambda_1, \quad \quad m\geq 1,\]
as the length $n$ of the path needed in witnessing generic $\alpha$-unbalancedness, according to  Definition \ref{Def:Warrow},  turns out to be  a function of $m$.  To deal with the ``atomic" cases  $\alpha=\omega^{\lambda}$ we introduce, similarly to the wreath-product jump, a  jump operation  for each ordinal of the form $\omega^{\lambda}$ and based on the fact that these ordinals are closed under Hessenberg addition $\oplus$ we show that these jumps admit  a uniform path-doubling ``fusion" procedure.

\subsection*{Acknowledgments}

We would like to thank the anonymous referee whose comments helped us improve the presentation. We  also thank L. Ding  for bringing to our attention a mistake in an earlier version of this manuscript, which has now been corrected.

\begin{LARGE}
\part{The class of $\alpha$-balanced Polish groups}\label{PartI}
\end{LARGE}

In Part I of this paper we develop the theory of $\alpha$-balanced Polish groups. Although we are primarily interested in Polish groups, some of the general theory we develop makes sense in the context of the more general classes of topological groups or  metrizable topological groups. In particular, Sections \ref{S:alpha_balanced} and  \ref{S:Closure} define the ``balanced rank"  assignment $G\mapsto \rk(G)$  and study its closure properties in the most general category of all topological groups. In Section \ref{S:Metrizable} we  prove  Theorem \ref{thm:CLI} by deriving it from a more general statement, which addresses the category of all metrizable topological groups.

Section \ref{S:Deissler} studies  $\alpha$-balancedness in the context of automorphism groups $\mathrm{Aut}(\mathcal{M})$ of countable structures $\mathcal{M}$. In particular, we show that $\rk(\mathrm{Aut}(\mathcal{M}))$ admits a  model-theoretic interpretation in terms of (a minor modification of) the Deissler rank  \cite{Deissler1997} of the structure $\mathcal{M}$. As an application of this model-theoretic viewpoint, we establish Theorem \ref{T:CompleteCoanalytic}. In Section \ref{S:Boundedness} we derive the ``boundedness principle" stated in Corollary \ref{cor:Analytic>Bounded}.  Finally, in Section \ref{S:Malicki} we discuss a natural weakening of Definition \ref{Def:alpha-balanced}, leading  to the class of \emph{weakly $\alpha$-balanced} Polish groups, and we connect this class to a rank that Malicki developed for Polish permutation groups in  \cite{Malicki2011}.

We now recall some definitions and establish some conventions that  will be used in the paper.

\smallskip{}

\subsection{Definitions and notation} A {\bf topological group} is a group $G$ together with a topology on $G$, rendering both the multiplication map $G\times G\to G$ and the inversion map $G\to G$  continuous. Recall that if $G$ is a topological group, then $V\subseteq_1 G$ is the notation we use for ``\emph{$V$ is an open neighborhood for the identity element $1_G$ of $G$}." 
If $N$ is a set, then 
 $\mathrm{Sym}(N)$ denotes the topological group of all permutations of $N$, endowed with the pointwise convergence topology, putting the discrete topology on $N$. A basic  $V\subseteq _1\mathrm{Sym}(N)$ is the stabilizer of finitely many points of $N$. A {\bf Polish permutation group} is a closed subgroup $P\leq \mathrm{Sym}(N)$ where $N$ is countable.
Finally, we will assume some familiarity with standard algebraic operations on ordinals. For example, the reader should be familiar with the fact that $\alpha+\beta = \beta+\alpha$ does not necessarily hold for two arbitrary ordinals $\alpha,\beta$.

\section{The $\alpha$-balanced topological groups}\label{S:alpha_balanced}

Let  $G$ be a topological group. For an ordinal $\beta$  and any $U,V\subseteq_1 G$ , recall that we define inductively:
\begin{itemize}
    \item $\rk(V,U)= 0$ if and only if $U\subseteq V$;
    \item  $\rk(V,U)\leq \beta$ if there is  $W\subseteq_1  G$ so that for any $g\in U$ we have $\rk(V,gW g^{-1})< \beta$;
    \item $\rk(V,U)= \infty$ if there is no ordinal $\beta$ for which $\rk(V,U)\leq \beta$.
\end{itemize}
When dealing with several topological groups, we will more verbosely write 
\[\rk(V,U;G)\]
to keep track of the ambient topological group in which the computation is  taking place.

\begin{definition}
Let $G$ be a topological group. We say that $G$ is $\alpha$-balanced if and only if
\[\rk(G):=\sup\{\rk(V,G)+1 \colon V\subseteq_1 G\}\leq \alpha.\]
\end{definition}

Trivially, there exist no $0$-balanced topological groups and  the only $1$-balanced topological group is the trivial  group consisting only of the identity element.
The class of all $2$-balanced topological groups coincides with the class of all topological $G$ which admit a basis of conjugation-invariant open neighborhoods of $1_G$. Indeed, $G$ is $2$-balanced if and only if for every $V\subseteq_1 G$ there exists some $W\subseteq_1 G$ so that $gWg^{-1}\subseteq V$. In particular, every open neighborhood $V$ of $1_G$ contains the conjugation-invariant open neighborhood of $1_G$ that is given by 
\[\widetilde{W}:=\bigcup_{g\in G}gWg^{-1}\]
As a consequence of the metrizability theorem of Klee \cite{Klee} we have that the class of all $2$-balanced Polish groups coincides with the class of all  TSI Polish groups. 
\begin{example} \label{Example:LocComp}
We spell out the property that $G$ is $3$-balanced: given any $V \subseteq _1G$ we have
\[\exists W \subseteq_1 G \quad \forall g\in G \quad \exists W'\subseteq_1 G \quad  \forall h \in gWg^{-1} \quad \big(h W' h^{-1} \subseteq V\big). \]
Any locally-compact topological group is $3$-balanced. Indeed, let $W\subseteq_1 G$ with $\overline{W}$ compact and notice that for all $V\subseteq_1 G$ and  $g\in G$ the set $K:=\overline{g W g^{-1}}$ is compact  and
\[F:=\{(h,v)\in K\times V \colon  h vh^{-1}\not\in V \}\]
is a closed subset of $K\times V$. Since $(h,1)\not\in F$ for all $h\in K$, by the ``tube lemma" as well as the continuity of the action by conjugacy, there exists some $W'\subseteq_1 G$ so that $hW'h^{-1}\subseteq V$ for all $h\in K$. 
Though we prove this more generally later, we can already see that the class of $3$-balanced Polish groups strictly contains the class of $2$-balanced Polish groups, since not every locally-compact Polish group is TSI (for example, take $\mathrm{SL}_2(\mathbb{R})$ \cite{Gao2008,Becker1}).
\end{example}

The next proposition  collects some basic properties that will be useful in what follows.

\begin{proposition}\label{prop:basic_properties_of_rk}
Let $G$ be a topological group and let $U,V,U',V'\subseteq_1 G$ and  $h\in G$. Then:
\begin{enumerate}
\item if  $V' \subseteq V$ and $U\subseteq U'$, then $\rk(V,U)\leq \rk(V',U')$;
\item $\rk(hVh^{-1}, hUh^{-1}) = \rk(V, U)$; and
\item $\rk(V \cap V', U) \leq \max\{\rk(V, U), \rk(V', U)\}$.
\end{enumerate}
\end{proposition}
\begin{proof}
(1) follows by induction on $\alpha := \rk(V', U')$.
If  $\rk(V', U')=0$ holds, then 
\[U \subseteq U' \subseteq V' \subseteq V\]
and thus   $\rk(V, U) = 0$  holds.
Let now $\alpha > 0$ and assume by inductive hypothesis that (1)  holds if   $\rk(V', U')<\alpha$. But if  $\rk(V', U')=\alpha$ holds, then there is $W\subseteq_1 G$ so that
\[\rk(V', gWg^{-1}) < \alpha  \text{  for all } g \in U'.\]
But then, by the induction hypothesis and the assumption $U\subseteq U'$, we have 
\[\rk(V, gWg^{-1}) < \alpha  \text{  for all } g \in U.\]
We therefore conclude  that $\rk(V, U) \le \alpha$ also holds. \smallskip{}
  
(2) follows by a similar induction on $\alpha := \rk(V, U)$. 
For $\alpha = 0$, simply observe that 
\[U \subseteq V \quad  \iff \quad hUh^{-1} \subseteq hVh^{-1}.\] 
Assume now that $\alpha>0$ and fix some $W\subseteq_1 G$ so that for every $g \in U$ we have that $\rk(V, gWg^{-1}) < \alpha$. 
By the induction hypothesis, for every $g \in U$ we have  that
\[\rk(hVh^{-1},hgWg^{-1}h^{-1}) < \alpha.\]
Hence,   $\widehat{W} = hWh^{-1}$ witnesses that for every $f=hgh^{-1} \in hUh^{-1}$ we have that
\[\rk(hVh^{-1}, f\widehat{W}f^{-1}) < \alpha,\]
and thus $\rk(hVh^{-1}, hUh^{-1}) \le \alpha$ also holds.\smallskip{}

(3) is proved by induction on $\alpha := \max\{\rk(V, U), \rk(V', U)\}$.
For $\alpha = 0$, we have
\[U \subseteq V \text{ and }  U\subseteq V' \quad \iff\quad U \subseteq V \cap V'.\]
Assume now that $\alpha>0$ and fix  $W, W'\subseteq_1 G$ such that for every $g \in U$ we have 
\[\rk(V, gWg^{-1})< \alpha \text{ and } \rk(V', gW'g^{-1}) < \alpha.\]
Taking $\widehat{W} := W \cap W'$ and invoking  (1),  for every $g \in U$  we have that
\[\rk(V, g\widehat{W}g^{-1}) \text{ and }  \rk(V', g\widehat{W}g^{-1}) < \alpha.\]
By the induction hypothesis $\rk(V \cap V', g\widehat{W}g^{-1}) < \alpha$ holds for all $g\in U$, and thus
\[\rk(V \cap V', U) \leq \alpha\]
as desired.
\end{proof}

We close this section with a lemma which implies that, when it comes to second-countable topological groups, $\rk(V,U)$ can be computed using a countable amount of data.

\begin{lemma}\label{L:rk*}
Let $Q$ be a dense subgroup of $G$ and let $\mathcal{V}$ be a local basis of open neighborhoods of $1_G$. 
Suppose $\rk_{Q,\mathcal{V}}(V,U)$ is defined the same way as $\rk(V,U)$ except that in the definition of $\rk_{Q,\mathcal{V}}(V,U)<\beta$ for $\beta>0$ we additionally require that $W\in\mathcal{V}$ and $g\in Q\cap U$.
Then for every $U,V\subseteq_1 G$ we have that $\rk(V, U) = \rk_{Q,\mathcal{V}}(V,U)$.
\end{lemma}
\begin{proof}
Let $U,V\subseteq_1 G$. The fact that $\rk_{Q,\mathcal{V}}(V,U)\leq \rk(V, U;G)$ holds follows from a routine induction argument using Proposition \ref{prop:basic_properties_of_rk} (1). For the reverse inequality, we assume that $\rk_{Q,\mathcal{V}}(V,U)=\alpha$ holds and we show that  $\rk(V, U)\leq\alpha$ holds by induction on $\alpha$. Since the two definitions agree for $\alpha = 0$, the base case is covered.
Now suppose $\alpha > 0$ and assume the claim is true below $\alpha$.
Let $W\in \mathcal{V}$ with $\rk_{Q,\mathcal{V}}(V,gWg^{-1})<\alpha$, and thus $\rk(V,gWg^{-1})<\alpha$ for all $g \in Q\cap U$.
Let $W_0\subseteq_1 G$ by symmetric with  $W_0^3 \subseteq W$ and $h\in U$ be arbitrary. Choosing any $g\in Q\cap U\cap hW^{-1}_0$ we have:
\[ \rk(V, hW_0h^{-1}) \le \rk(V, (gW_0)W_0(gW_0)^{-1}) \le \rk(V, gWg^{-1}) < \alpha\]
as desired, where both inequalities follow from Proposition \ref{prop:basic_properties_of_rk} (1).
\end{proof}

\begin{remark}
The term ``$\alpha$-balanced Polish group"  was  suggested
(but not defined) in our earlier work  \cite[Section 7]{AP2021}.
An attempt to define this term appears in \cite{DiZh22}, which is unfortunately a bit different from our definition  here. The authors of the latter paper take a more indirect approach in that the ``$\alpha$-balancedness" (in their sense) of a Polish group $G$ is inferred from the action of $c_0(G):=\{(g_n) \in G^{\omega} \colon g_n \to 1_G \}$ on $G^{\omega}$.
However, they do not succeed in producing CLI Polish groups which are not $\omega$-balanced in their sense.
\end{remark}

\section{Constructions and closure properties}\label{S:Closure}

In this section, we study how the rank assignment function $G\mapsto\rk(G)$ behaves in the context of various classical constructions in the category of topological groups. As an application, we  end this section by  deriving the following  theorem. 

\begin{theorem}\label{T:All_Ordinals}
For each ordinal $\alpha>0$, there exist a topological group $G$ with $\rk(G)=\alpha$.    
\end{theorem}

\subsection{Topological subgroups and quotients} An {\bf embedding} of topological groups is an  injective homomorphism $H\to G$ that is continuous and open on its image. An {\bf epimorphism} of topological groups is a surjective homomorphism $G\to H$ that is continuous.

\begin{proposition}\label{Proposition:Subgroup_Quotient}
Let $G$ and $H$ be topological groups. Then:
\begin{enumerate}
    \item if there exists an embedding $i\colon H\to G$, then $\rk(H) \le \rk(G)$; and
    \item if there exists an open epimorphism $q\colon G\to H$, then $\rk(H) \le \rk(G)$.
\end{enumerate}
\end{proposition}
\begin{proof}
For (1), since the homomorphism $i$ is open on its range, every open neighborhood of $1_H$ in $H$ is of the form $i^{-1}(O)$, for some $O\subseteq_1 G$. Hence, (1) follows from the next claim.     

\begin{claim}
For all $V,U\subseteq_1 G$ we have that $\rk(i^{-1}(V),i^{-1}(U);H)\leq \rk(V,U;G)$.    
\end{claim}
\begin{proof}[Proof of Claim]
Clearly $U\subseteq V$ implies  that  $i^{-1}(U) \subseteq i^{-1}(V)$ and therefore the claim holds if $\rk(V,U;G)\leq 0$. Assume now the claim holds for all  $V,U\subseteq_1 G$  with $\rk(V,U;G)< \alpha$ and let  $V,U\subseteq_1 G$, for which there is some $W\subseteq_1 G$ so that for all $g\in U$, $\rk(V,gWg^{-1};G)< \alpha$ holds. But then, for $\widetilde{W}:=i^{-1}(W)$ and any $h\in i^{-1}(U)$ we have  $g_h:=i(h)\in U$ and 
\[h \widetilde{W} h^{-1}= i^{-1}\big(g_h W h_h^{-1}\big)\]
Hence, by inductive hypothesis, for every $h\in i^{-1}(U)$ we have that
\[\rk(i^{-1}(V),h \widetilde{W} h^{-1};H)\leq \rk(V,g_hWg_h^{-1};G)< \alpha\]
as desired.
\end{proof}

Similarly, (2) follows from the next claim, since $q$ is continuous and surjective, 

\begin{claim}
If $V,U\subseteq_1 H$ and $\widetilde{V},\widetilde{U}\subseteq_1 G$ satisfy $q(\widetilde{V})=V$ and $q(\widetilde{U})=U$, then
\[\rk(V,U;H) \leq \rk(\widetilde{V},\widetilde{U};G).\]    
\end{claim}
\begin{proof}[Proof of Claim]
If $q(\widetilde{V})=V$ and $q(\widetilde{U})=U$ hold, then  $\widetilde{U}\subseteq \widetilde{V}$ implies  that $U\subseteq V$  and therefore the claim holds if $\rk(\widetilde{V},\widetilde{U};G)\leq 0$. Assume now the claim holds when $\rk(\widetilde{V},\widetilde{U};G)< \alpha$ and let  $V,U\subseteq_1 H$  and $\widetilde{V},\widetilde{U}\subseteq_1 G$, as in the claim,   for which there is some $\widetilde{W}\subseteq_1 G$ so that  for all $g\in \widetilde{U}$ we have that $\rk(\widetilde{V},g\widetilde{W}g^{-1};G)< \alpha$. 
Since $q$ is open, we have that $W:=q(\widetilde{W})$ is open. Since $U=q(\widetilde{U})$, for all $h\in U$ there is some $g_h\in \widetilde{U}$ so that $q(g_h)=h$. But then, for all $h\in U$ we have that:
\[ \rk(V,h W h^{-1};H)\leq  \rk(\widetilde{V},g_h\widetilde{W}g_h^{-1};G)< \alpha   \]
where the first inequality follows from the inductive step, since $q$ is a homomorphism.
\end{proof}
This completes the proof.
\end{proof}

Let now $H$ be a closed normal subgroup of a topological group $G$ and let $G/H$ be the quotient group endowed with the quotient topology. Applying Proposition \ref{Proposition_Product_Group} to the inclusion map $i\colon H\to G$ and the quotient map $q\colon G\to G/H$ we see that  the property of being $\alpha$-balanced is closed under passing to topological subgroups or quotients.

A natural question is whether there is a general formula which bounds $\rk(G)$  in terms of   $\rk(H)$ and $\rk(G / H)$. More precisely consider the following problem:

\begin{problem}\label{Problem_1}
Let $H$ be a closed normal subgroup of a topological group $G$. Is it true that 
\[\rk(G) \le \sup\big\{\rk(V,H;H)+ \rk(G/H) \colon V\subseteq_1 H\big\}?\]
\end{problem}

At this point it is unclear to the authors whether this problem has a positive answer even in the special case when the associated short exact sequence of topological groups  {\bf topologically splits}, or even when it (fully) {\bf splits}. That is, even when  there is a continuous function (continuous homomorphism, respectively) $s\colon G/H \to G$, so that $q \circ s(\hat{g})= \hat{g}$, for all $\hat{g}\in G/H$.

\begin{figure}[ht!]
    \centering
   \begin{tikzcd}
    1\arrow{r} & H\arrow{r}{i} & G\arrow{r}{q}  & G/H\arrow{r} 
    \arrow[shift left=1.2ex, dashed]{l}{s} & 1
\end{tikzcd}
\end{figure}

That being said, from \cite[Theorem 2.2.11]{Gao2008} and  Theorem \ref{thm:CLI} above, we have that $\rk(G) < \infty$ holds if $G,H$ are as in Problem  \ref{Problem_1} and  both $\rk(G/H) < \infty$, $\rk(H) < \infty$ hold.

\subsection{Products of topological groups} Let $(G_i)_{i\in I}$ be a family of topological groups indexed by some set $I$ and let $G:=\prod_i G_i$ the product group with the product topology. In particular, a basis of open neighborhoods of $1_G$ in $G$ consists of sets of the form $V=\prod_i V_i$  with $V_i\subseteq_1 G_i$ for all $i\in I$ and $V_i=G_i$ for all but finitely many $i\in I$.

\begin{proposition}\label{Proposition_Product_Group}
If     $G:=\prod_i G_i$, then we have that $\rk(G)=\sup\big\{\rk(G_i)\colon i\in I\big\}$.
\end{proposition}

The proof is a direct consequence of the following lemma.

\begin{lemma}\label{Lemma_Product_Group}
Let $G:=\prod_i G_i$ and let basic $V,U\subseteq_1 G$ with $V:=\prod_i V_i$, $U:=\prod_i U_i$. Then,     
\[ \rk(V,U;G)=\max\big\{\rk(V_i,U_i;G_i)\colon i\in I\big\}.\]
\end{lemma}
We omit the proof of this lemma as it is a consequence of the more general Lemma \ref{L:local_direct_product}. Indeed, by setting $O_i:= G_i$ in Lemma \ref{L:local_direct_product} we get  $\rk(V,U;G)=\sup\big\{\rk(V_i,U_i;G_i)\colon i\in I\big\}$.
To see that  the supremum is realized, notice that  $V_i=G_i$ for all but finitely many $i\in I$.

\begin{remark}\label{Remark:Polish_product}
Notice that if $I$ is countable and  $G_i$ is Polish for all $i\in I$, then 
so is $\prod_i G_i$.
\end{remark}

\subsection{Local direct products of topological group pairs}

Let $(G_i,O_i)_{i\in I}$ be a family of topological groups $G_i$ together with an open subgroup $O_i\leq G_i$, indexed by some set $I$. For any finite $A\subseteq I$ consider the topological group
\[H_A:=\big(\prod_{i\in A} G_i\big)\times \big(\prod_{i\in I\setminus A} O_i\big) \]
endowed with the product topology. For every finite $A,B\subseteq I$ with $A\subseteq B$ the inclusion
\[i^A_B\colon H_A\to H_B\]
is a continuous and open group homomorphism. As a consequence the direct limit
\[\oplus^{\mathrm{loc}}_{i}(G_i,O_i) := \mathrm{colim}\big(H_A, i^A_B \big)  \]
is a topological group if endowed with the colimit topology. Recall that this is the finest topology which makes all the inclusion maps $i_{\infty}^A\colon H_A\to \oplus^{\mathrm{loc}}_{i}(G_i,O_i)$ continuous. We call 
$\oplus^{\mathrm{loc}}_{i}(G_i,O_i)$ the  {\bf local direct product} of the family $(G_i,O_i)_{i\in I}$.

It is not difficult to see that $\oplus^{\mathrm{loc}}_{i}(G_i,O_i)$ admits the following more concrete description:

\[\oplus^{\mathrm{loc}}_{i}(G_i,O_i):=\{(g_i)_i\in\prod_i G_i \colon g_i\in O_i \text{ for all but finitely many } i\in I \}\]
and the topology on $\oplus^{\mathrm{loc}}_{i}(G_i,O_i)$ is the coarsest \emph{group topology} which refines the product topology (inherited by the inclusion $\oplus^{\mathrm{loc}}_{i}(G_i,O_i)\subseteq \prod_i G_i$) by declaring $O$ open, where
\[O:=\prod_i O_i\]

\begin{remark}\label{Remark:Polish_local_direct}
Notice that if $I$ is countable and  $G_i$ is Polish for all $i\in I$, then $\oplus^{\mathrm{loc}}_{i}(G_i,O_i)$ is also Polish. To see this, consider the countable discrete collection of cosets 
\[N= \oplus^{\mathrm{loc}}_{i}(G_i,O_i)/O\]
and notice that $\oplus^{\mathrm{loc}}_{i}(G_i,O_i)$ can be identified with a closed subset of the Polish space 
\[\big(\prod_{i\in I}G_i\big) \times N.\]
\end{remark}

\begin{proposition}\label{Prop:local_direct_product}
If $(G_i,O_i)_{i\in I}$ is a family of pairs $O_i\leq G_i$ of topological groups with $O_i$ open in $G_i$ and $O:=\prod_i O_i$, then we have that:
\[\rk(O, \oplus^{\mathrm{loc}}_{i}(G_i,O_i);\oplus^{\mathrm{loc}}_{i}(G_i,O_i))= \mathrm{sup} \{\rk(O_i,G_i;G_i) \colon i \in I\}.\]
\end{proposition}

Notice that, by the concrete description of $\oplus^{\mathrm{loc}}_{i}(G_i,O_i)$ above, we see that sets of the following form constitute a basis of the open neighborhoods of $1$ in $\oplus^{\mathrm{loc}}_{i}(G_i,O_i)$:
\[V:=\prod_iV_i \subseteq \prod_iG_i\]
where $V_i=O_i$ for all but finitely many $i\in I$. Below, we refer to such sets as ``basic open". As a consequence, Proposition \ref{Prop:local_direct_product} follows directly from the next lemma.

\begin{lemma}\label{L:local_direct_product}
For all basic open $V,U\subseteq_1 \oplus^{\mathrm{loc}}_{i}(G_i,O_i)$ we have that:
\[\rk(V,U;\oplus^{\mathrm{loc}}_{i}(G_i,O_i)) = \sup \{\rk(V_i,U_i;G_i) \colon i \in I\}. \]
\end{lemma}
\begin{proof}
Clearly $U\subseteq V$ holds if and only if $U_i\subseteq V_i$ holds for all $i\in I$.    

Assume now that $\rk(V,U;\oplus^{\mathrm{loc}}_{i}(G_i,O_i))\leq \alpha$ for some $\alpha>0$. By Proposition \ref{prop:basic_properties_of_rk}(1) we may choose a basic $W\subseteq_1 \oplus^{\mathrm{loc}}_{i}(G_i,O_i)$ so that for all $g\in U$ we have that 
\[\rk(V,gW g^{-1};\oplus^{\mathrm{loc}}_{i}(G_i,O_i))<\alpha.\]
Since every $h\in U_i$ lifts to some $g\in U$ with $g_i=h$, by inductive assumption we have 
\[\rk(V_i,U_i,G_i)\leq \alpha \]
for all $i\in I$. Hence, it follows that $\sup \{\rk(V_i,U_i;G_i) \colon i \in I\}\leq \alpha$ also holds.

Conversely, assume that for some $\alpha>0$ and all $i\in I$ we have that  $\rk(V_i,U_i;G_i)\leq \alpha$ holds. Let $A\subseteq I$ be a finite set so that $V_i=U_i=O_i$ for all $i\in I\setminus A$ and consider the basic open  $W\subseteq_1 \oplus^{\mathrm{loc}}_{i}(G_i,O_i)$ with: $W_i=O_i$, if $i\not\in A$; and with $W_i$ chosen so that for all $h\in U_i$ we have   $\rk(V_i,h W_i h^{-1};G_i)<\alpha$, if $i\in A$.  It follows, that for all $g\in U$ we have 
\[\rk(V,gW g^{-1};\oplus^{\mathrm{loc}}_{i}(G_i,O_i) )<\alpha\]
Indeed, let $g=(g_i)_i\in U$ and notice that by the choice if $A$ we have  $g_i W_i g_i^{-1}=O_i=V_i$ for all $i\not\in A$. Hence, $\rk(V_i,g_i W_i g_i^{-1};G_i)=0$ if $i\not\in I$, and by inductive assumption we have
\begin{eqnarray*}
    \rk(V,gWg^{-1};\oplus^{\mathrm{loc}}_{i}(G_i,O_i))&=&\mathrm{sup} \{\rk(V_i,g_i W_i g_i^{-1};G_i) \colon i\in I\}\\
    &=& \max  \{\rk(V_i,g_i W_i g_i^{-1};G_i) \colon i\in I\}<\alpha.
\end{eqnarray*}
\end{proof}

\subsection{Wreath products and $\mathbb{Z}$--jumps of Polish groups} Let   $P\leq \mathrm{Sym}(I)$ be a group of permutations of a set $I$. Given any group $G$, the  {\bf unrestricted wreath product} of $G$ by $P$ is the following semidirect product:

\[P \mathrel{\mathrm{Wr}} G \; := \; P\ltimes\prod_{i\in I}  G \; = \; P\ltimes G^I.\]
More concretely, $P \mathrel{\mathrm{Wr}} G$ consists of all pairs $(p,(g_i)_i)$ with $p\in P$ and $(g_i)_i\in G^I$, and the group multiplication in $P \mathrel{\mathrm{Wr}} G$ is given by the formula:
\[(q,(h_i)_i) \; \cdot \; (p,(g_i)_i)=\big(qp, (h_{p(i)} g_{i})_i\big). \]

In what follows we will refer to  $P\mathrel{\mathrm{Wr}} G$, simply, as the 
 {\bf wreath product} of $G$ by $P$. The reader is warned that this  term is often  used in the literature for the variant of the above construction, where one uses direct sum $\oplus_{i\in I}$ in place of the full product $\prod_{i\in I}$.

If $P$ and $G$ are topological groups and $P\curvearrowright I$ is continuous, then the product topology on the topological space $P\times\prod_{a}G$ renders $P \mathrel{\mathrm{Wr}} G$ a topological group. In particular

\begin{remark}\label{Remark:Wreath_Polish}
If $I$ is countable and both $P,G$ are Polish groups then so is  $P \mathrel{\mathrm{Wr}} G$.    
\end{remark}

The {\bf $\mathbb{Z}$-jump} of a topological group $G$ is the topological group  $\mathbb{Z} \mathrel{\mathrm{Wr}} G$. Here we identify $\mathbb{Z}$ with the countable discrete permutation group $\mathbb{Z}\leq \mathrm{Sym}(\mathbb{Z})$ given by $p(k)=p+k$ for all $p,k\in \mathbb{Z}$. The following solves a special case of Problem \ref{Problem_1}.

\begin{proposition}\label{Prop:Z-jump}
Let $G$ be a topological group and let $\rk(G)=\alpha$. Then,
\begin{enumerate}
    \item if $\alpha$ is a successor ordinal then  $\rk(\mathbb{Z} \mathrel{\mathrm{Wr}} G)=\alpha+1$;
    \item if $\alpha$ is a limit ordinal then  $\rk(\mathbb{Z} \mathrel{\mathrm{Wr}} G)=\alpha$;
\end{enumerate}
\end{proposition}

Recall that a basic open $V\subseteq_1 G^{\mathbb{Z}}$  is of the form  $V=\prod_{k\in\mathbb{Z}} V_k$, with: $V_k\subseteq_1 G$, for all $k\in\mathbb{Z}$; and $V_k=G$, for all but finitely many $k\in\mathbb{Z}$. For any such $V$, consider the set:
\[\{0\}\times V \subseteq \mathbb{Z} \times G^{\mathbb{Z}} \]
Notice that  sets of this form constitute a basis open neighborhoods of $1$ in $\mathbb{Z} \mathrel{\mathrm{Wr}} G$. Hence, the proof of Proposition \ref{Prop:Z-jump} is a direct consequence of the following lemma. 

\begin{lemma}\label{L:Z-jump}
Let $V\subseteq_1 G^{\mathbb{Z}}$ be a basic open neighborhood of $1$ in $G^{\mathbb{Z}}$. Then we have that: 
\[\rk( \{0\} \times V,\, \mathbb{Z}\mathrel{\mathrm{Wr}}G; \,\mathbb{Z}\mathrel{\mathrm{Wr}}G)= \rk(V,G^{\mathbb{Z}};G^{\mathbb{Z}})+1.\]
\end{lemma}
\begin{proof}
Since $W:=\{0\}\times G^{\mathbb{Z}}$ is a normal subgroup of $\mathbb{Z}\mathrel{\mathrm{Wr}}G$, for all $g\in \mathbb{Z}\mathrel{\mathrm{Wr}}G$ we have 
\[\rk( \{0\} \times V, gWg^{-1}; \,\mathbb{Z}\mathrel{\mathrm{Wr}}G)=\rk( \{0\} \times V, \{0\}\times G^{\mathbb{Z}}; \,\mathbb{Z}\mathrel{\mathrm{Wr}}G)= \rk(V,G^{\mathbb{Z}};G^{\mathbb{Z}}) \]
from which it follows that $\rk( \{0\} \times V,\, \mathbb{Z}\mathrel{\mathrm{Wr}}G; \,\mathbb{Z}\mathrel{\mathrm{Wr}}G)\leq \rk(V,G^{\mathbb{Z}};G^{\mathbb{Z}})+1$.

Conversely, assume towards contradiction that   $\rk( \{0\} \times V,\, \mathbb{Z}\mathrel{\mathrm{Wr}}G; \,\mathbb{Z}\mathrel{\mathrm{Wr}}G)\leq\beta$ holds, where $\beta:= \rk(V,G^{\mathbb{Z}};G^{\mathbb{Z}})$.  By Proposition \ref{prop:basic_properties_of_rk}(1) we may find a basic $W=\prod_kW_k
\subseteq_1 G^{\mathbb{Z}}$ so that for every $g\in \mathbb{Z}\mathrel{\mathrm{Wr}}G$ we have that
\[\rk( \{0\} \times V,\, g \cdot \big(\{0\} \times W\big)\cdot g^{-1}; \,\mathbb{Z}\mathrel{\mathrm{Wr}}G)<\beta\]
Since the sets $A:=\{k\in \mathbb{Z}\colon V_k \neq G\}$, $B:=\{k\in \mathbb{Z}\colon W_k \neq G\}$ are finite, there exists $\ell_*\in\mathbb{Z}$ so that $(\ell_*+B)\cap A=\emptyset$. Let now $g\in \mathbb{Z}\mathrel{\mathrm{Wr}} G$ be  of  the form $g=(\ell_*,(g_k)_k)$ for any choice of $(g_k)_k\in G^{\mathbb{Z}}$ ---take, for example, $g_k:=1_G$ for all $k\in\mathbb{Z}$. By the choice of $\ell_*$, we have:
\[ g \cdot \big(\{0\} \times W\big)\cdot g^{-1}= \{0\} \times \widetilde{W}, \]
for some basic $ \widetilde{W}=\prod_k \widetilde{W}_k
\subseteq_1 G^{\mathbb{Z}}$ with the property $V_k\neq G \implies   \widetilde{W}_k=G$ for all $k\in\mathbb{Z}$.   We  have thus arrived at the following contradiction:
\[\beta=\rk(V,G^{\mathbb{Z}};G^{\mathbb{Z}})= \rk(V,\widetilde{W};G^{\mathbb{Z}})=\rk( \{0\} \times V, \{0\} \times \widetilde{W}; \,\mathbb{Z}\mathrel{\mathrm{Wr}}G)<\beta \]
To see why the second equality holds, notice that $\rk(V,G^{\mathbb{Z}};G^{\mathbb{Z}})\geq  \rk(V,\widetilde{W};G^{\mathbb{Z}})$ follows from Proposition \ref{prop:basic_properties_of_rk} (1). On the other hand, by Lemma \ref{Lemma_Product_Group} there exist some $k_{*}\in\mathbb{Z}$ with 
\[\rk(V,G^{\mathbb{Z}};G^{\mathbb{Z}})=\rk(V_{k_{*}},G;G),\]
and since $V_k\neq G$ implies  $ \widetilde{W}_k=G$, by another use of Lemma \ref{Lemma_Product_Group} we also have that:  
\[\rk(V,G^{\mathbb{Z}};G^{\mathbb{Z}})=\rk(V_{k_{*}},G;G) = \rk(V_{k_{*}},\widetilde{W}_{k_{*}};G) \leq \rk(V,\widetilde{W};G^{\mathbb{Z}}).\] 
\end{proof}

\subsection{An application}

Theorem \ref{T:All_Ordinals}, which stated that for every ordinal there is a topological group with that rank, is a simple consequence  of the above constructions:

\begin{proof}[Proof of Theorem \ref{T:All_Ordinals}]
For each $\alpha>0$, we define a topological group with $\rk(G_{\alpha})=\alpha$. In the process, we will also need specify an   open subgroup  $O_{\alpha}\leq G_{\alpha}$, which satisfies
\begin{equation}
\rk(G_{\alpha})=\rk(O_{\alpha},G_{\alpha},G_{\alpha})+1,\label{Eq:Application}    
\end{equation}

whenever $\alpha$ is a successor ordinal.

If $\alpha=1$,  then simply set $G_{\alpha}:=O_{\alpha}:=\{1\}$ be the trivial group.

If $\alpha$ is a limit ordinal, then let  $G_{\alpha}=\prod_{\beta<\alpha} G_{\beta}$. By Proposition \ref{Proposition_Product_Group} we have that 
\[\rk(G_{\alpha})=\sup\{ \rk(G_{\beta}) \colon \beta<\alpha \}=\sup\{ \beta \colon \beta<\alpha \}=\alpha,\]
where the middle equality follows by inductive assumption. Since $\alpha$ is not a successor ordinal, we can arbitrarily choose $O_{\alpha}$ to be, say, equal to $G_{\alpha}$.

If $\alpha$ is a successor ordinal then we distinguish between the following two cases. 

If $\alpha=\lambda+1$ for a limit ordinal $\lambda$ then set $G_{\alpha}:=\oplus^{\mathrm{loc}}_{\beta<\lambda}(G_{\beta},O_\beta)$ and $O_{\alpha}:=\prod_{\beta<\lambda}O_{\beta}$. By Proposition \ref{Prop:local_direct_product},  inductive assumption and the cofinality of  successor ordinals $\beta<\lambda$  in $\lambda$, 
\[\rk(O_{\alpha}, G_{\alpha};G_{\alpha})=\sup\{\rk(O_{\beta},G_{\beta},G_{\beta}) \colon \beta<\alpha \}=\lambda \]
It follows that $\rk(G_{\alpha})\geq \lambda+1 =\alpha$. The fact that  $\rk(G_{\alpha})\leq \alpha$ also holds, and hence so does $\rk(G_{\alpha})=\alpha$, is an elaboration on Lemma \ref{L:local_direct_product} and the fact the tightness of choice of $O_{\beta}\leq G_{\beta}$ with $\beta<\lambda$, coming from (\ref{Eq:Application}) above.

If $\alpha=\beta+1$ for some successor ordinal, then let $G_{\beta},O_{\beta}$ satisfying (\ref{Eq:Application}) above and let
\[G_{\alpha}:= \mathbb{Z}\mathrel{\mathrm{Wr}}G_{\beta}\quad \text{ and } \quad O_{\alpha}:= \big(O_{\beta}\big)^{\mathbb{Z}}\subseteq \big(G_{\beta}\big)^{\mathbb{Z}} \]
But then, by Proposition \ref{Prop:Z-jump} and Lemma \ref{L:Z-jump} it follows  that
\[\alpha=\beta+1=\rk(G_{\alpha})= \rk(O_{\alpha}, G_{\alpha}; G_{\alpha})+1\]
as desired.
\end{proof}

Notice that the extension procedures  used in the proof of Theorem \ref{T:All_Ordinals} generate topological groups with  non-trivial closed normal subgroups. A natural question\footnote{We thank the anonymous referee for raising this question} is whether, for any given $\alpha$, there exists a topologically simple topological group $G$ with $\rk(G)=\alpha$.

\begin{problem}
For every ordinal $\alpha$, show that there exists some  topological group $G$, with $\rk(G)=\alpha$, which has no  non-trivial closed normal subgroups. 
\end{problem}

\section{Stratifying the class of all CLI groups}\label{S:Metrizable}

In this section, we formulate and prove a version of Theorem \ref{thm:CLI} ---which stated that the CLI Polish groups are precisely those Polish groups with (countable) ordinal rank--- for
the more general class of metrizable topological groups. We then  specialize this more general statement to the class of Polish groups and recover Theorem \ref{thm:CLI}.

A {\bf metrizable} group is a topological group $G$ which admits a metric $d$ that is compatible with its topology. By Birkhoff-Kakutani \cite{birkhoff1936note,kakutani1936metrisation} a topological group is metrizable if and only if its topology is Hausdorff and first-countable. Moreover, every such group admits a metric $d_{\ell}$ that is both compatible with the topology and {\bf left-invariant}:
\[\forall f,g,h\in G \quad d_{\ell}(fg,fh)=d_{\ell}(g,h). \]
Setting $d_r(g,h):=d_{\ell}(g^{-1},h^{-1})$ we see that $G$ also  admits a {\bf right-invariant} metric $d_r$ compatible with the topology. Combining  the two metrics we get the {\bf strong metric}:
\[d_s(g,h):=d_{\ell}(g,h)+d_r(g,h)\]
While $d_s$ is neither left-invariant nor right-invariant, it has the following important property: the completion $\widehat{G}^s$ of $G$ with respect to $d_s$ carries a natural topological group structure which extends the group structure of $G$. In contrast, the completions $\widehat{G}^{\ell}$, $\widehat{G}^{r}$ of $G$  with respect to $d_{\ell}$ and $d_r$, respectively, fail to carry a topological group structure that extends that of $G$. That is,  unless  $\widehat{G}^{\ell}=\widehat{G}^{s}=\widehat{G}^{r}$ holds. 
As it turns out, $\rk(G)$   measures the difficulty of the task of confirming whether $\widehat{G}^{\ell}=\widehat{G}^{s}$ holds. Before we make this precise, we recall that the definitions of  $\widehat{G}^{\ell}$,  $\widehat{G}^{r}$ and $\widehat{G}^{s}$ do not depend on the choice of $d_{\ell}$. Indeed, these completions can be directly defined using the canonical left, right, and two-side  \emph{uniform structures} that any topological group carries \cite{roelcke}. Since we will restrict the discussion to the realm of metrizable groups, it suffices to work with sequences.

Let $G$ be a metrizable group. A sequence  $(g_n)_n$ in $G$ is  \textbf{left-Cauchy} if for every $V\subseteq_1 G$ there exists $n_0\in\mathbb{N}$ so that  for all $n \ge n_0$ we have that $g_n \in g_{n_0}V$, and it is  \textbf{right-Cauchy} if for every $V\subseteq_1 G$ there exists $n_0\in\mathbb{N}$ so that  for all $n \ge n_0$ we have that $g_n \in Vg_{n_0}$. It is finally \textbf{strongly Cauchy},  if it is both left-Cauchy and right-Cauchy. It is not difficult to see that $(g_n)_n$ is left-Cauchy if and only if it is Cauchy to some (equivalently, any) left-invariant metric $d_{\ell}$ on $G$. The same relation exists between right-Cauchy (respectively, strongly Cauchy) sequences and Cauchy sequences with respect to $d_r$ (respectively, $d_s$). In particular, this provides a ``metric-free" definition of the completions $\widehat{G}^{\ell}$,  $\widehat{G}^{r}$ and $\widehat{G}^{s}$.

\begin{theorem}\label{Th:Main_metrizable}
If $G$ is a metrizable topological group, then we have that:
\[\widehat{G}^{\ell}=\widehat{G}^{s} \;  \iff \;  \rk(G)<\infty \;  \iff \; \rk(G)<\omega_1\]
\end{theorem}
\begin{proof}
To establish the first equivalence, it suffices to show that the following two properties of $G$ are equivalent:
\begin{enumerate}
    \item every left-Cauchy sequence in $G$ is also right-Cauchy;
    \item there is some ordinal $\alpha$ so that for every  $V\subseteq_1 G$,  $\rk(V, G) \le \alpha$ holds.
\end{enumerate}
\medskip

\noindent We first prove $(2)\implies (1)$. Assume that  $\rk(V, G) \le \alpha$ for every  $V\subseteq_1 G$ and let $(g_n)_n$ be a left-Cauchy sequence in $G$. Fix also some  $V\subseteq_1 G$. By applying the next claim to $U=G$, we get some $n_0\in\mathbb{N}$ so that for all $n \geq n_0$ we have that $g_n\in Vg_{n_0}$, as desired.

\begin{claim}
Suppose that for some  $U, V\subseteq_1 G$ and  $m_0\in\mathbb{N}$ we have $\rk(V, U)\leq \alpha$ and $g_n \in Ug_{m_0}$, for all $n\geq m_0$. Then, there exists  $n_0\in\mathbb{N}$ so that $g_n \in Vg_{n_0}$, for all $n\geq n_0$. 
\end{claim}
\begin{proof}[Proof of Claim]
If $\alpha = 0$ then $U\subseteq V$ and hence the claim follows for $n_0:=m_0$.

Assume now that  $\alpha > 0$ and that the claim holds below $\alpha$.
Let $W\subseteq_1 G$ so that for all $g \in U$, we have $\rk(V, gWg^{-1}) < \alpha$.
Since $(g_n)_n$ is left-Cauchy, there exists some $m_1\in\mathbb{N}$ so that $g_n \in g_{m_1}(g_{m_0}^{-1}Wg_{m_0})$, for all $n\geq m_1$. Assuming without loss of generality that $m_1\geq m_0$, we get   $g\in U$ so that  $g_{m_1} = gg_{m_0}$.
But then, for all $n\geq m_1$ we have that
\[ g_n \in g_{m_1}(g_{m_0}^{-1}Wg_{m_0}) = g_{m_1}(g_{m_0}^{-1}W)g_{m_0} = gg_{m_0}(g_{m_0}^{-1}W)g^{-1}g_{m_1} = (gWg^{-1})g_{m_1}.\]
Since  $\rk(V, gWg^{-1}) < \alpha$ holds, by the inductive hypothesis we get $n_0\in \mathbb{N}$ so that for every $n \ge n_0$ we have $g_n \in Vg_{n_0}$ as desired.
\end{proof}

\noindent Next we prove $(1)\implies (2)$. Assume that there is no ordinal $\alpha$ as in (2). Since $G$ is a set, this implies that  there exists some $V\subseteq_1 G$ in $G$ so that $\rk(V,G)=\infty$. We will  define a left-Cauchy sequence $(g_n)_n$ which is not right-Cauchy.
In fact, fix any $\widetilde{V}\subseteq_1 G$ with  $\widetilde{V}^3 \subseteq V$. We will prove the following stronger claim:

\begin{claim}
There exists a left-Cauchy  $(g_n)_n$ in $G$ so that for all $ n\in\mathbb{N}$, $ g_{n+1} \not\in \widetilde{V}g_n$ holds.
\end{claim}
\begin{proof}
  Fix $G=: W_{-1}\supseteq W_0\supseteq W_1\supseteq \cdots$ a basis of neighborhoods around $1_G$ in $G$ so that $W^{-1}_n=W_n$ and $W^2_n\subseteq W_{n-1}$ for all $n\in\mathbb{N}$. 
First we recursively define some sequence $(h_n)_n$ in $G$ together with open neighborhoods  $\widetilde{W_n}$ of $1_G$  in $G$, so that:
\begin{enumerate}
    \item  $h_n\in h_{n-1} \widetilde{W}_{n-1} h^{-1}_{n-1}$ for all $n\in\mathbb{N}$;
    \item$\rk(V,h_n\widetilde{W}_nh^{-1}_n)=\infty$ for all $n\in\mathbb{N}$;
    \item $h_0\ldots h_{n-1}\widetilde{W}_nh_{n-1}^{-1}\ldots h_0^{-1} \subseteq W_n$; and
    \item $\widetilde{W}_n \subseteq \widetilde{V}$.
\end{enumerate}

Suppose now that we have successfully defined such a sequence, and for each $n\in\mathbb{N}$ set
\[g_n:=h_n\ldots h_0.\]
By (1) above we have that 
\[ g_{n+1}^{-1}g_n = h_0^{-1}\ldots h_n^{-1}h_{n+1}h_n\ldots h_0 \in h_0^{-1}\ldots h_{n-1}^{-1} \widetilde{W}_n h_{n-1}\ldots h_0 \]
which is contained in $W_n$ by (3). 
Thus $(g_n)_n$ is left-Cauchy. 
On the other hand,  we also have that for every $n\in\mathbb{N}$, $g_{n+1}g_n^{-1} = h_{n+1} \not\in \widetilde{V}$ holds. Indeed,  otherwise we would have 
\[h_{n+1}\widetilde{W}_{n+1}h_{n+1}^{-1} \subseteq \widetilde{V}^{3} \subseteq V\]
by (4) and the choice of $\widetilde{V}$, which would contradict (2).

We finish by constructing $(h_n)_n$.
Since $\rk(V,W_{-1})=\infty$, we may choose $h_0$ satisfying
\[\rk(V,h_{0}W_{0}h_{0}^{-1})=\infty,\]
and set $\widetilde{W}_0 := W_0 \cap \widetilde{V}$. Next,
assume for the inductive step that $h_0, \ldots , h_n$ and $\widetilde{W}_0, \ldots , \widetilde{W}_n$ have been defined so that they satisfy  properties (1)--(4) above.
Choose $\widetilde{W}_{n+1}$ to be any open neighborhood of $1_G$  that is small enough so that both  $h_0\ldots h_n\widetilde{W}_{n+1}h_n^{-1}\ldots h_0^{-1} \subseteq W_n$ and $\widetilde{W}_{n+1} \subseteq \widetilde{V}$ hold. By property (2) above we can  find some $h_{n+1} \in h_n \widetilde{W}_n h_n^{-1}$ so that 
\[\rk(V, h_{n+1}\widetilde{W}_{n+1} h_{n+1}^{-1}) = \infty.\]
It's easy to check that properties (1)--(4)  are  satisfied for $h_{n+1}$ and $W_{n+1}$ as well.  
\end{proof}
This concludes the proof of the first equivalence in the statement of Theorem \ref{Th:Main_metrizable}.  The second and last equivalence is a direct consequence of Lemma \ref{L:Main_metrizable} below. 
\end{proof}

\begin{lemma}\label{L:Main_metrizable}
If  $V,U\subseteq_1 G$ for metrizable $G$, then $\rk(V, U) < \infty$ implies $\rk(V, U) < \omega_1$.
\end{lemma}
\begin{proof}[Proof of Lemma \ref{L:Main_metrizable}]
Let $\alpha := \rk(V, U)$ be an ordinal and fix a countable dense subgroup $Q\leq G$ and a countable basis $\mathcal{V}$ of open neighborhoods of $1_G$ so that $gOg^{-1}\in \mathcal{V}$ for all $O\in\mathcal{V}$ and $g\in \mathcal{Q}$. Since $\mathcal{V}$ is countable it suffices to prove the following claim.

\begin{claim}
For every $\beta  < \alpha$ there exists some $O\in\mathcal{V}$  such that $\rk(V, O) = \beta$.
\end{claim}
\begin{proof}[Proof of Claim]
Assume that this fails for some $\beta <\alpha$.  But then we can define a sequence $(O_n, \alpha_n)_{n\in\mathbb{N}}$,  where $O_n\in\mathcal{V}$ and   $\alpha_n$ is an ordinal with $\alpha_n>\beta$, so that $\rk(V, O_n) = \alpha_n$ and  $\alpha_{n+1} < \alpha_n$ hold. This would clearly contradicts the well-foundedness of ordinals.

The sequence  $(O_n, \alpha_n)_{n\in\mathbb{N}}$ is defined recursively. Set $\alpha_0 = \alpha$ and $O_0=U$. Assume now that  $O_n$ and $\alpha_n$ have been defined with $\alpha_n > \beta$. By inductive assumption we have that $\rk(V, O_n) = \alpha_n$. By Lemma \ref{L:rk*} we can fix some $W\in \mathcal{V}$ so that for every $g \in \mathcal{Q}\cap O_n$, $\rk(V, gWg^{-1}) < \alpha_n$.
But then there must exist $g \in \mathcal{Q}\cap O_n$ with $\rk(V, gWg^{-1}) > \beta$.  Indeed, otherwise we would have    $\rk(V, gWg^{-1}) \leq \beta$ for all $g\in \mathcal{Q}$, and since by the choice of $\beta$,  $gWg^{-1}\in\mathcal{V}$ implies  $\rk(V, gWg^{-1})\neq \beta$, we get the following  contradiction:  
\[\alpha_n=  \mathrm{rk}(V,O_n) = \mathrm{rk}_{Q,\mathcal{V}}(V,O_n)\leq\beta. \]
Hence, we may simply set $O_{n+1}:=gWg^{-1}$ and $\alpha_{n+1}:=\rk(V, O_{n+1})$,  for any such  set $gWg^{-1}\in\mathcal{V}$ which satisfies  $\rk(V, gWg^{-1})>\beta$. 
\end{proof}

This completes the proof.
 \end{proof}

We can now put everything together and derive Theorem \ref{thm:CLI} from the introduction.

\begin{proof}[Proof of Theorem  \ref{thm:CLI}]
Recall that for any Polish group $G$ we have that $\widehat{G}^s=G$. The reason is that $\widehat{G}^s$ is a Polish group and $G$ is a dense $G_{\delta}$ subgroup of $G$. Hence, by the Baire-category theorem, the only coset of $G$ in $\widehat{G}^s$ is $G$ itself.  As a consequence of this observation, the first statement of  Theorem \ref{thm:CLI} follows directly from Theorem \ref{Th:Main_metrizable}.
The second statement of  Theorem \ref{thm:CLI} follows from the proof of Theorem \ref{T:All_Ordinals} and the Remarks \ref{Remark:Wreath_Polish}, \ref{Remark:Polish_local_direct}, \ref{Remark:Polish_product}.  Another family of examples of Polish groups $G_{\alpha}$ with $\rk(G_{\alpha})=\alpha$ for all (successor) ordinals $\alpha<\omega_1$ will be  given in Section \ref{S:Deissler}. 
\end{proof}

\section{Automorphism groups of countable structures} \label{S:Deissler}

In this section we consider $\alpha$-balancedness in the context of  automorphism groups $\mathrm{Aut}(\mathcal{M})$ of countable structures $\mathcal{M}$. In particular, we show that  $\rk(\mathrm{Aut}(\mathcal{M}))$ corresponds to a minor modification of a rank that Deissler developed for countable $\mathcal{L}$-structures. Loosely speaking,  $\rk(\mathrm{Aut}(\mathcal{M}))=\alpha<\infty$ holds if  and only if every attempt to build some $\mcL_{\infty, \omega}$-elementary embedding $\mathcal{M}\to \mathcal{M}$  results in a surjective map and, moreover, the proof that surjectivity is unavoidable requires strategies which are ``$\alpha$-complex".
This connection allows us to import model-theoretic intuition into the study of $\alpha$-balanced groups. We use this  to generate a ``uniformly defined" family of CLI Polish groups with unbounded balanced rank and derive Theorem \ref{T:CompleteCoanalytic}. This family consists of 
 automorphism groups of certain scattered linear orders and it will  play an important role in Part II of the paper.

We start by recalling the classical correspondence between non-archimedean Polish groups and countable model-theoretic structures. A Polish group $G$ is \textbf{non-archimedean}, if it admits a  basis of  neighborhoods of the identity consisting of open \emph{subgroups}. The prototypical non-archimedean Polish group is the group $S_{\infty}:=\mathrm{Sym}(\mathbb{N})$, of all permutations of $\mathbb{N}$, endowed  with the pointwise-convergence topology. In fact, up to isomorphism,  non-archimedean Polish group are precisely the closed subgroups of $S_{\infty}$.

The following proposition justifies, among other things, the use of basic countable model theory to understand the non-archimedean Polish groups. 
By a countable $\mathcal{L}$-structure we mean an $\mathcal{L}$-structure $\mathcal{M}$ so that both the language $\mathcal{L}$ and the domain $M:=\mathrm{dom}(\mathcal{M})$ are countable. We endow  the group $\Aut(\mcM)$ of all automorphisms of $\mcM$ with the pointwise-convergence topology. For a proof of the following classical result,  see \cite[1.5.1]{BeckerKechris1996}.

\begin{proposition}\label{prop:na}
Given a Polish group $G$, the following are equivalent:
\begin{enumerate}
    \item $G$ is non-archimedean;
    \item $G$ is isomorphic to a closed subgroup of $S_\infty$;
    \item $G$ is isomorphic to $\Aut(\mcM)$ for a countable  $\mcL$-structure $\mcM$ in some countable $\mcL$;
    \item $G$ has a compatible left-invariant ultrametric.
\end{enumerate}
\end{proposition}

Recall that an ultrametric is metric which satisfies the following strengthening of the triangle inequality: for all $x,y,z$ we have $d(x,y)\leq \max\{d(x,z),d(z,y)\}$. For example,  a natural left-invariant ultrametric $d_\ell$ on $S_\infty$ is given  by
\[d_\ell(g, h) := \max \{1/2^n \mid n \in \omega, \; g(n) \neq h(n)\}.\]
This metric is not complete, and thus $S_\infty$ is not CLI.
To see this, view $S_\infty$ as a $G_\delta$ subset of $\mathbb{N}^\mathbb{N}$ equipped with the Polish topology of pointwise convergence.
Then the injections in $\mathbb{N}^\mathbb{N}$ are precisely the limits of $d_\ell$-Cauchy sequences in $S_\infty$.

With this discussion in mind, we will speak only of automorphism groups $\Aut(\mcM)$ of countable structures $\mcM$ for the remainder of this section, with the understanding that this discussion applies to the entire class of non-archimedean Polish groups.

We now state a result of Gao, which characterizes the CLI Polish automorphism groups in terms of model theoretic concepts.
We will follow with a review of all the relevant model-theoretic definitions.

\begin{theorem}[Gao, \cite{Gao1998}]\label{thm:gao_characterization}
For a countable $\mcL$-structure $\mcM$, the following are equivalent:
\begin{enumerate}
    \item $\Aut(\mcM)$ is CLI;
    \item there is no non-surjective $\mcL_{\infty, \omega}$-elementary embedding of $\mcM$ into $\mcM$;
    \item there is no non-trivial $\mcL_{\infty, \omega}$-elementary substructure of $\mcM$;
    \item there is no uncountable model with the same $\mcL_{\infty, \omega}$-theory as $\mcM$.
\end{enumerate}
\end{theorem}

For any language $\mcL$, the $\mcL_{\infty, \omega}$ formulas are those that are built up from atomic formulas using symbols in $\mcL$ allowing finite quantification, and disjunction and conjunctions of any collection of $\mcL_{\infty, \omega}$-formulas on a common finite variable context.
If  disjunctions and conjunctions are further restricted to  be of countable size then we get the $\mcL_{\omega_1, \omega}$-formulas.
Given an $\mcL$-structure $\mcM$, formula $\varphi(\bar{x})$, and $\bar{a} \in M^{|x|}$, we write $\mcM \models \varphi(\bar{a})$ iff $\varphi(\bar{a})$ holds in $\mcM$.
Scott's theorem says that if $\mcM$ is a countable $\mcL$-structure, then for any tuple $\bar{a}$ in $\mcM$ there is a $\mcL_{\omega_1, \omega}$-sentence $\varphi(\bar{x})$ such that $\mcM \models \varphi(\bar{a})$ holds and for any other tuple $\bar{b}$, if $\mcM \models \varphi(\bar{b})$ holds, then there is an automorphism $g$ of $\mcM$ such that $g(\bar{a}) = \bar{b}$.
We call $\varphi$ a \emph{Scott sentence} of $\bar{a}$. For a review of the theory of Scott sentences see \cite[Chapter 12]{Gao2008}, though only the knowledge that they exist is necessary for the following discussion.

An $\mcL_{\infty, \omega}$-elementary embedding of an $\mcL$-structure $\mcM$ into an $\mcL$-structure $\mcN$ is an injection $f : M \rightarrow N$ satisfying that for any tuple $\bar{a}$ from $M$ and any $\mcL_{\infty, \omega}$-formula $\varphi(\bar{x})$ we have $\mcM \models \varphi(\bar{a})$ iff $\mcN \models \varphi(f(\bar{a}))$.
An $\mcL_{\infty, \omega}$-elementary substructure of an $\mcL$-structure $\mcM$ is a substructure $\mcN \le \mcM$ satisfying that for any tuple $\bar{a}$ from $N$ and any $\mcL_{\infty, \omega}$-formula $\varphi(\bar{x})$ we have $\mcM \models \varphi(\bar{a})$ iff $\mcN \models \varphi(\bar{a})$.
The $\mcL_{\infty, \omega}$-theory of an $\mcL$-structure is the set of $\mcL_{\infty, \omega}$ sentences (formulas with no free variables) that hold in the structure.

\subsection{The Deissler rank}

The following rank function, which we call \emph{Deissler rank}, is due to Deissler in \cite{Deissler1997}, except that here we will use tuples $\bar{a},\bar{b}$ instead of  points $a,b$ in the recursion. 
Let  $\mcM$ be a countable $\mcL$-structure and let $\bar{a}$ and $\bar{b}$ be tuples from $\mcM$.
Write $\Drk(\bar{a}, \bar{b}) \le 0$ iff $\bar{a}$ is definable from $\bar{b}$ in the sense that there is some $\mcL_{\infty, \omega}$-formula $\varphi(\bar{x}, \bar{y})$ such that $\mcM \models \varphi(\bar{a}, \bar{b})$ and moreover there is no other tuple $\bar{c}$ such that $\mcM \models \varphi(\bar{c}, \bar{b})$.
For $\alpha > 0$ we recursively define $\Drk(\bar{a}, \bar{b}) \le \alpha$ to mean that there is a tuple $\bar{c}$ and some $\mcL_{\infty, \omega}$-formula $\varphi(\bar{x}, \bar{y})$ such that $\mcM \models \varphi(\bar{c}, \bar{b})$ and for every $\bar{c}'$, if $\mcM \models \varphi(\bar{c}', \bar{b})$ then we have $\Drk(\bar{a}, \bar{b}\bar{c}') < \alpha$.
If $\Drk(\bar{a}, \bar{b}) \le \alpha$ does not hold for any $\alpha$, then we write $\Drk(\bar{a}, \bar{b}) = \infty$. 

By Scott's theorem, since we are working with countable structures, we can give an equivalent definition in terms of the automorphisms of $\mcM$ that makes the connection to our rank on CLI Polish groups more transparent.
In particular, in terms of the natural action $\Aut(\mcM) \curvearrowright M^{<\omega}$ on the tuples of $\mcM$, we can say $\Drk(\bar{a}, \bar{b}) \le 0$ iff $\Stab(\bar{b}) \cdot \bar{a} = \{\bar{a}\}$.
And for $\alpha > 0$ we can say $\Drk(\bar{a}, \bar{b}) \le \alpha$ if and only if there is a tuple $\bar{c}$ such that for every $\bar{d} \in \Stab(\bar{b}) \cdot \bar{c}$, we have $\Drk(\bar{a}, \bar{b}\bar{d}) < \alpha$.
Here, $\Stab(\bar{b})$, the stabilizer of $\bar{b}$, denotes the clopen subgroup of $\Aut(\mcM)$ consisting of automorphisms $\pi$ such that $\pi \cdot \bar{b} = \bar{b}$.
The equivalence of this definition with the previous one is a trivial application of Scott's theorem, and from now on we will use this second, more natural definition.

The following proposition establishes a direct correspondence between our rank from Definition \ref{Def:alpha-balanced} and the Deissler rank, in the context of automorphism groups of countable structures.

\begin{proposition}\label{Prop:RanksAreTheSame}
Let $\mcM$ be a countable $\mathcal{L}$-structure. Then, for any $\bar{a}, \bar{b} \in M^{<\omega}$ we have
\[\Drk(\bar{a}, \bar{b}) = \rk\big(\Stab(\bar{a}), \Stab(\bar{b}); \mathrm{Aut}(\mathcal{M})\big).\]
\end{proposition}

\begin{proof}
    This is a straightforward induction on $\alpha$ that for every $\bar{a}$ and $\bar{b}$, $\Drk(\bar{a}, \bar{b}) \le \alpha$  holds if and only if $\rk(\Stab(\bar{a}), \Stab(\bar{b})) \le \alpha$ holds. The key observation is that if $\bar{b}$ and $\bar{c}$ are tuples with $W := \Stab(\bar{b}\bar{c})$ and $g \in \Stab(\bar{b})$ then $gWg^{-1}$ is equal to $\Stab(\bar{b}\bar{d})$ for some tuple $\bar{d} \in \Stab(\bar{b})$, namely $\bar{d} = g \cdot \bar{c}$.
\end{proof}

In the context of Proposition \ref{Prop:RanksAreTheSame}, the first part of Theorem \ref{thm:CLI} can be seen as an extension of Deissler's theorem below (and its consequent corollary) to the realm of all Polish groups.

\begin{theorem}[Deissler, \cite{Deissler1997}]\label{thm:deissler_characterization}
For any $\mcL$-structure $\mcM$, the following are equivalent:
\begin{enumerate}
    \item there is no nontrivial $\mcL_{\infty, \omega}$-elementary substructure of $\mcM$;
    \item $\Drk(a, \emptyset) < \infty$ for every $a \in M$; and
    \item $\Drk(a, \emptyset) < \omega_1$ for every $a \in M$.
\end{enumerate}
\end{theorem}

By Proposition \ref{Prop:RanksAreTheSame} and Theorem \ref{thm:gao_characterization} we also have the following corollary for countable $\mathcal{L}$-structures $\mathcal{M}$:

\begin{corollary}\label{Cor:Rev1}
 $\Aut(\mcM)$ is CLI if and only if $\Drk(a, \emptyset) < \omega_1$ for every $a \in M$.
\end{corollary}

\subsection{Scattered linear orders and CLI Polish groups.}

We first remind the reader of the basic theory of scattered linear orders, and establish some notation.
Given a linear order $L$, recursively define equivalence relations $H_\alpha$ on $L$ for each ordinal $\alpha$, where each $H_\alpha$-class is an interval.
Let $H_0$ be equality and define $H_\alpha = \bigcup_{\beta < \alpha} H_\beta$ for limit $\alpha$.
Define $H_{\alpha+1}$ by saying $\ell \mathrel{F_{\alpha+1}} \ell'$ iff only finitely-many distinct $H_\alpha$ classes lie between $[\ell]_{H_\alpha}$ and $[\ell']_{H_\alpha}$.
We  call these equivalence relations $H_\alpha$ the \textbf{strong Hausdorff derivatives} of $L$.

There must be some least ordinal $\alpha(L)$ such that $H_\alpha = H_{\alpha(L)}$ for every $\alpha \ge \alpha(L)$. 
We call this the \textbf{strong Hausdorff rank} of $L$. 
When $L$ is countable, $\alpha(L)$ is easily seen to be countable.

Recall that a linear order is called \textbf{scattered} if and only if it does not have a suborder isomorphic to $\mathbb{Q}$.
If $H_{\alpha(L)}$ only has one class, then $L$ is scattered.
Indeed, if $H_{\alpha(L)}$ has more than one class, then it must have infinitely-many classes, and if we let $I$ be a suborder of $L$ taking one point out of each $H_{\alpha(L)}$-class, it must be dense, and thus must have a suborder isomorphic to $\mathbb{Q}$.
On the other hand, if $I$ is a suborder of $L$ isomorphic to $\mathbb{Q}$, then for any $\ell < \ell'$ in $I$, an easy argument shows that $\neg (\ell \mathrel{H_{\alpha(L)}} \ell')$.
In particular, there must be more than one $H_{\alpha(L)}$--class.

We use the adjective ``strong" here in the development of the Hausdorff rank and derivative as the usual successor case of the derivative relates $\ell$ and $\ell'$ at stage $\alpha + 1$ iff the quotient linear order of $H_\alpha$ classes that lie between $[\ell]_{H_\alpha}$ and $[\ell']_{H_\alpha}$ is a \emph{well-order} (as opposed to being finite in our definition).
In Hausdorff's original paper \cite{Hausdorff1908}, he showed a linear order has ordinal Hausdorff rank if and only if it is scattered (for a more modern reference on on this, see  \cite{Rosenstein1982}).
Thus our formulation of Hausdorff rank, which is more suitable for our purposes, does not exactly align with Hausdorff's original formulation but still characterizes the scattered linear orders.

\begin{lemma}\label{lem:scattered_cli}
If $L$ is scattered with Hausdorff rank $\le \alpha$ then $\mathrm{Aut}(L)$ is $(\alpha+1)$-balanced.
\end{lemma}
\begin{proof}
By Proposition \ref{Prop:RanksAreTheSame} we may compute $\rk(\mathrm{Aut}(L))$ using the Deissler rank.

We show that for any ordinal $\beta$, if $a \mathrel{H_{1 + \beta}} b$ holds, then we have $\Drk(a, b) \le \beta$.
We proceed  by induction on $\beta$. 
For $\beta = 0$, if $a \mathrel{H_1} b$ then $a$ is finitely-many steps away from $b$ and just definable from $b$, and so $\rk(a, b) \le 0$. 
Let now $\beta > 0$ and assume the claim is true below $\beta$.
For the case of limit $\beta$, the result is immediate by the definition of $H_\beta$.
Thus assume $\beta = \gamma + 1$ and suppose $a \mathrel{H_{1 + \gamma+1}} b$.
Observe that for every automorphism $\pi \in \Aut(L)$ fixing $b$, we have $\pi(a) \mathrel{H_{1 + \gamma}} a$.
By  induction hypothesis  $\Drk(a, \pi(a)) \le 1 + \gamma$, for any such $\pi$.
Thus by the definition of Deissler rank we have $\Drk(a, b) \le 1 + \beta$ as desired.

Now taking into consideration that $L$ is scattered and thus $H_\alpha$ consists of a single class, we observe finally that $\Drk(a, \emptyset)$ is at most $\alpha(L)+1$ for every $a$, as desired.
\end{proof}

As a direct consequence of this lemma and Corollary \ref{Cor:Rev1}, we get the following:

\begin{theorem}\label{thm:scattered_cli}
    If $L$ is scattered then $\Aut(L)$ is CLI.
\end{theorem}

The converse is not true, as it is possible to have a non-scattered linear order that is \emph{rigid} (meaning that the only automorphism is the trivial one).
For example, fix an enumeration $q_n$ of $\mathbb{Q}$ in ordertype $\omega$ and consider the linear order consisting of pairs $(q_n, i)$ where $0 \le i \le n$ where we declare $(q_n, i) < (q_m, j)$ iff either $q_n < q_m$ or $q_n = q_m$ and $i < j$.
This is easily seen to be non-scattered as the suborder of pairs $(q_n, 0)$ is isomorphic to $\mathbb{Q}$.
However, we will see  next  that for certain nice $L$, the converse is indeed true.

\subsection{Homogeneous scattered linear orders}\label{SS:Scattered_LOs}
Given a linear order $(L, <)$, let $\mathbb{Z}[L]$ be the set of all functions $a : L \rightarrow \mathbb{Z}$ so that $a(\ell) = 0$ for all but finitely-many $\ell \in L$. 
We endow $\mathbb{Z}[L]$ with the ``backwards-lexicographic" order  $\prec^*_L$. Namely, $a_1 \prec^*_L a_2$ iff $a_1 \neq a_2$ and for the $L$-\emph{greatest} $\ell \in L$ with $a_1(\ell) \neq a_2(\ell)$, we have $a_1(\ell) < a_2(\ell)$. 
Then $( \mathbb{Z}[L], \prec^*_L )$ is easily seen to be a linear order, which we simply denote by $\mathbb{Z}[L]$. 
When $L$ is countable, so is $\mathbb{Z}[L]$.
For example, we have that $\mathbb{Z}[1]$ is a $\mathbb{Z}$--line and $\mathbb{Z}[2]$ is ``the $\mathbb{Z}$--line of $\mathbb{Z}$--lines".

We caution that the order we are considering is \emph{not} the usual lexicographic ordering.
This subtle fact only becomes truly apparent for infinite $L$.
For example, when given the usual lexicographic ordering, $\mathbb{Z}[\omega]$ is a dense linear order with a left endpoint, but with the  ordering $\prec^*_{\omega}$ that we have endowed it with, it is a scattered linear order with strong Hausdorff rank $\omega$.
Expanding on this, we observe the following:

\begin{proposition}\label{prop:wellfounded_iff_scattered}
A  linear order $L$ is a  well-order if and only if  $\mathbb{Z}[L]$ is scattered. 
\end{proposition}

\begin{proof}
We start with the forward  direction.
Without loss of generality assume $L = \alpha$ for some ordinal $\alpha$. 
For $\beta < \alpha$, define the equivalence relation $F_\beta$ on $\mathbb{Z}[L]$ where $a_1 \mathrel{F_\beta} a_2$ iff $a_1(\gamma) = a_2(\gamma)$ for every $\gamma \ge \beta$, and $F_\alpha = L \times L$. 
It is easy to check that these are precisely the Hausdorff derivatives of $\mathbb{Z}[L]$ and thus $\mathbb{Z}[L]$ is scattered of rank $\alpha$.

Now we proceed to the reverse direction, which we prove by showing that if $L$ is not  a well-order, then $\mathbb{Z}[L]$ has dense suborder.
First we show that if $L$ has no minimal element, then $\mathbb{Z}[L]$ is dense. 
To this end, suppose $a \prec^*_L b$ in $\mathbb{Z}[L]$. 
Let $\ell \in L$ be minimal such that $a(\ell) \neq 0$ or $b(\ell) \neq 0$. 
By the assumption there is  $\hat{\ell} \in L$  with $\hat{\ell} < \ell$. 
Define $c \in \mathbb{Z}[L]$ by
\[ c(\ell') = \begin{cases}
a(\ell') & \ell' \neq \hat{\ell} \\
1 & \ell' = \hat{\ell}
\end{cases}.
\]
It is easy to check that $a \prec^*_L c \prec^*_L b$, as desired.

Second we observe that if $I$ is a suborder of $L$ then there is a suborder of $\mathbb{Z}[L]$ isomorphic to $\mathbb{Z}[I]$.
Indeed, the map $a \mapsto \hat{a}$ where
\[\hat{a}(\ell) = \begin{cases}
    a(\ell) & \ell \in I\\
    0 & \ell \not\in I
\end{cases}
\]
is an embedding of $\mathbb{Z}[I]$ into $\mathbb{Z}[L]$.
Of course, every ill-founded order has a suborder with no minimal element, and thus we are done.
\end{proof}

The  following theorem is the main result of this section.

\begin{theorem}\label{thm:scattered_iff_cli}
For any linear order $L$, the following are equivalent:
\begin{enumerate}
    \item $L$ is a well-order;
    \item $\mathbb{Z}[L]$ is scattered;
    \item $\Aut(\mathbb{Z}[L])$ is CLI;
\end{enumerate}
\end{theorem}

The equivalence of $(1)$ and $(2)$ is Proposition \ref{prop:wellfounded_iff_scattered}.
The fact that $(2)$ implies $(3)$ follows from Theorem \ref{thm:scattered_cli}. So we are left to prove that $(3)$ implies $(1)$.
Instead of proving this directly, we will first introduce  an additional clause $(4)$ and show that $(3)\rightarrow (4) \rightarrow (1)$.

Clause $(4)$ will be used in Section \ref{SS:PI11Rank} for  establishing that the class of CLI Polish groups is $\BPi^1_1$-complete.
In short, the natural strategy there would be to show that the map $L \mapsto \Aut(\mathbb{Z}[L])$ is Borel.
One would  attempt to prove that $L \mapsto \Aut(\mathbb{Z}[L])$ is Borel by viewing it  as the composition of  $L \mapsto \mathbb{Z}[L]$ and $\mathbb{Z}[L] \mapsto \Aut(\mathbb{Z}[L])$, and then by showing that both of these maps are Borel when viewed as functions between the appropriate Polish spaces. While the first map is easily seen to be Borel,  the second map need not be Borel. Clause (4) considers an expansion $M_L$ of  $\mathbb{Z}[L]$, with additional structure which ensures that $\mathrm{Aut}(M_L)$ is sufficiently easy to compute, resulting in a Borel assignment $L \mapsto \Aut(M_L)$.

To describe this additional structure,  let $(L, <)$ be an arbitrary linear order. For each $\ell \in L$ and $z \in \mathbb{Z}$, let $E_{(\ell,z)}$ be the binary relation on $\mathbb{Z}[L]$ where $E_{(\ell,z)}(a, b)$ holds iff $b(\ell) = a(\ell) + z$ and $a(\ell') = b(\ell')$ for all $\ell'$ with  $\ell' > \ell$.
Observe that for all $a \neq b$ in $\mathbb{Z}[L]$ there is exactly one  $(\ell, z)$ for which $E_{(\ell,z)}(a, b)$ holds. Let now $\mcL_L$ be the language consisting of a binary relation $\prec^*$ and $E_{(\ell,z)}$ for each $z \in \mathbb{Z}$ and $\ell \in L$.
We set $M_L$  to be the $\mcL_L$-structure on domain $\mathbb{Z}[L]$ equipped with the relations  $E_{(\ell,z)}$, in addition to $\prec^*_L$.

\begin{proposition}\label{prop:scattered_iff_cli}
If $L$ is a linear order, we can add the following clause to Theorem \ref{thm:scattered_iff_cli}:
\begin{enumerate}
    \item[(4)] $M_L$ is CLI
\end{enumerate}
\end{proposition}

The fact that $(3)$ implies $(4)$ follows from the fact that any closed subgroup of a CLI Polish group is CLI, see \cite[Theorem 2.2.11]{Gao2008}. Indeed, notice that $\mathrm{Aut}(M_L)$ is closed  subgroup of  $\mathrm{Aut}(\mathbb{Z}[L])$. We also observe that the inclusion $\mathrm{Aut}(M_L)\subseteq \mathrm{Aut}(\mathbb{Z}[L])$ is strict in general (e.g. when $L$ is not a well-order), but this is not going to be  important below.

Before we turn to the final implication  $(4) \rightarrow (1)$ we first establish the pertinent property of $M_L$, which lead us to consider it in the first place.

Let $\bar{a}=(a_0,\ldots,a_{n})$ and $\bar{b}=(b_0,\ldots,b_{n})$ be tuples of the same length in some $\mathcal{L}$-structure $M$. We say that $\bar{a}$ and $\bar{b}$ have the same quantifier-free type, if the assignment $a_i\mapsto b_i$ induces an isomorphism between the substructures $\langle \bar{a} \rangle_M$ and $\langle \bar{b} \rangle_M$ of $M$. We say that $M$ is {\bf ultrahomogeneous} if whenever  $\bar{a}$ and $\bar{b}$ in $M$ have the same quantifier-free type, there exists an automorphism $\varphi\in \mathrm{Aut}(M)$ satisfying $\varphi(a_i)=b_i$ for all $i\leq n$.

\begin{lemma}\label{L:rev1}
$M_L$ is ultrahomogeneous.
\end{lemma}
\begin{proof}
We proceed by induction on the length of the tuples $\bar{a},\bar{b}$ in $M_L$. Let $n\in\mathbb{N}$ and fix tuples $(a_0,\ldots,a_n)$ and $(b_0,\ldots,b_n)$ in  $M_L$ of the same quantifier-free type. We will find some $\varphi\in\mathrm{Aut}(M_L)$, with  $\varphi(a_i)=b_i$ for all $i\leq n$. By our inductive hypothesis,  there exists  $\psi\in\mathrm{Aut}(M_L)$ with $\psi(a_i)=b_i$ for all $i < n$. Hence, the proof reduces to the following claim, by plugging $\bar{c}:=(b_0,\ldots,b_{n-1})$, $a:=\psi(a_n)$, and $b:=b_n$.

\begin{claim}
Let $a,b, \bar{c}$ in $M_L$ so that $\bar{c}a$ and $\bar{c}b$ have the same quantifier-free type. Then, there exists some $\varphi\in\mathrm{Aut}(M_L)$ with  $\varphi(a)=b$, so that $\varphi(c)=c$ for all entries $c$ in $\bar{c}$.
\end{claim}

To prove the claim, we proceed by induction on the size of the finite set
\[F:=\{\ell\in L \colon a(\ell)\neq b(\ell)\}.\]
If $F=\emptyset$, then  simply  take $\varphi:=\mathrm{id}$. Otherwise, let $\ell^*:=\min(F)$ and  $t:= (b(\ell^*)-a(\ell^*))\in\mathbb{Z}$. 
Consider the automorphism $\psi^{*}\in\mathrm{Aut}(M_L)$,  where for all $d\in \mathbb{Z}[L]$  and $\ell\in L$ we have
\[
\psi^*(d)(\ell)=
\begin{cases}
d(\ell)+ t  & \text{ if } \ell=\ell^* \text{ and for all }  \ell' >\ell  \text{ one has }  d(\ell')=a(\ell');\\
d(\ell) & \text{ otherwise}.\\
\end{cases}
\]
The fact $\psi^{*}$ is an automorphism readily follows by taking  $d,d'\in \mathbb{Z}[L]$ and considering three cases, depending on whether both, one, or none of $d,d'$ agree with $a$ above $\ell^*$.

Notice also that $\psi^*(c)=c$ holds for all  entries $c$ of $\bar{c}$. Indeed, if $\psi^*(c)\neq c$ holds for some $c$ in $\bar{c}$, then  
$c(\ell^*)=a(\ell^*)$ holds. But then, setting $z:=c(\ell^*)-a(\ell^*)$, we have
\[M_L\models E_{(\ell^*,z)}(c,a) \quad \text{and} \quad M_L\models  \neg \, E_{(\ell^*,z)}(c,b), \]
where the second statement follows from   $b(\ell^*) \neq a(\ell^*)$. This 
 contradicts the fact that $\bar{c}a$ and $\bar{c}b$ have the same quantifier-free type. 

Setting $a':=\psi^*(a)$, we have  that   $\bar{c}a'$ and $\bar{c}b$ still have the same quantifier-free type, and 
the size of the set $\{\ell\in L \colon \psi^*(a)(\ell)\neq b(\ell)\}$ is strictly smaller than the size of $F$. By inductive hypothesis, there exists $\psi\in\mathrm{Aut}(M_L)$ with  $\psi(\psi^*(a))=b$, and $\varphi(c)=c$ for all entries $c$ in $\bar{c}$.  Hence, the claim follows by simply setting $\varphi:=\psi\circ \psi^*$.

This completes the proof.
\end{proof}

We now proceed with the proof of the  implication  $\neg (1) \rightarrow \neg(4)$. 

\begin{proposition}\label{P:Rev2}
   If $L$ is not a well-order then $\Aut(M_L)$ is not CLI.
\end{proposition}

For the proof we will need the next lemma, which, in conjunction with Corollary \ref{Cor:Rev1},   provides a handy criterion for establishing when an automorphism group is not CLI.

\begin{lemma}\label{L:RevHandy}
  Let $\mcM$ be any $\mcL$-structure and let $a \in M$. If $\Drk(a, \bar{b}) > 0$ holds for every  tuple $\bar{b}$ in $M$ with $a \not\in \bar{b}$, then $\Drk(a, \emptyset) = \infty$.
\end{lemma}
\begin{proof}
We will show that whenever  $\Drk(a, \bar{b}) \le \alpha$ holds for some $\bar{b}$ in $M$ with $a \not\in \bar{b}$, then
$\Drk(a, \bar{b}\bar{e}) \le 0$ holds for some  $\bar{e}$ in $M$ with $a \not\in \bar{e}$. We will use transfinite induction on $\alpha$. 

If $\alpha = 0$, then we can take $\bar{d}$ to be the empty tuple. Assume now that  $\alpha > 0$.
Fix some  $\bar{c}$ in $M$, so that $\Drk(a, \bar{b}\bar{d}) < \alpha$ holds for every $\bar{d} \in \Stab(\bar{b}) \cdot \bar{c}$.

If there is no  $\bar{d}$  in $\Stab(\bar{b}) \cdot \bar{c}$ with $a \not\in \bar{d}$, then we can take $\bar{e}$ to be $\bar{c}$ with $a$ removed, in which case we get $\Drk(a, \bar{b}\bar{e}) \le 0$.
Otherwise, choose $\bar{d}$ to be any element of $\Stab(\bar{b}) \cdot \bar{c}$ in which $a \not\in \bar{d}$.
Since $\Drk(a, \bar{b}\bar{d}) < \alpha$, by induction there is some  $\bar{d}'$ with $a \not\in \bar{d}'$ and $\Drk(a, \bar{b}\bar{d}\bar{d}') \le 0$. In which case, we can take $\bar{e} = \bar{d}\bar{d}'$.
\end{proof}

\begin{proof}[Proof of Proposition \ref{P:Rev2}] 
Since $L$ is not well-ordered,  we can decompose it as $L = I + J$ where $J$ has no minimum element.
Notice that  $\Aut(M_L)$ \emph{involves} $\Aut(M_J)$, meaning that there is a closed subgroup $H$ of $\Aut(M_L)$ and a continuous surjective homomorphism of $H$ onto $\Aut(M_J)$. Indeed, Let $H^J_L$ be the closed subgroup of $\Aut(M_L)$ consisting of automorphisms $\pi$ satisfying $\pi(a)(\ell) = a(\ell)$ for every $\ell \not\in J$. It readily follows that the map $\pi \mapsto \pi_J$ from $H^J_L$ to $\Aut(M_J)$ given by
    \[ \pi_J(a) = \pi(\tilde{a}) \upharpoonright J,  \quad \text{ where } \quad  \tilde{a}(\ell) = \begin{cases}
        0 & \ell \in I \\
        a(\ell) & \ell \in J,
    \end{cases}\]
is a continuous surjective homomorphism.
Since any Polish group involved in a CLI group is CLI (see \cite[Theorem 2.2.11]{Gao2008}) it suffices to show that $\Aut(M_J)$ is not CLI. 
That is, it suffices to prove Proposition \ref{P:Rev2} under the stronger assumption that $L$ has no minimum.

Assume that $L$ has no minimum and let $a\in M_L$ be the constant zero map: $a(\ell)=0$ for all $\ell\in L$. By the previous lemma, and since $M_L$ is ultrahomogeneous, it suffices to show that for any  $\bar{b}$ in $M_L$ with $a\not\in\bar{b}$, there is $a'\neq a$ so that $ \bar{b}a$ and  $\bar{b}a'$ have the same quantifier-free type.  
But since $L$ has no minimum we can always pick some $\ell^*\in L$ with
\[\ell^* <\min \{\ell \in L \colon b_i(\ell)\neq 0 \text{ for some } b_i \text{ in } \bar{b} \} \]
and let $a'\in M_L$ be defined by:   $a'(\ell)=1$, if $\ell=\ell^*$; and $a'(\ell)=0$, if $\ell\neq \ell^*$.
\end{proof}

\begin{remark}
For a more direct but perhaps less informative proof of Proposition \ref{P:Rev2}, one can use a  strictly decreasing sequence $\ell_0> \ell_1 >\cdots > \ell_n > \cdots$ in $L$ to explicitly  define 
an $\mathcal{L}_{\infty,\omega}$-elementary embedding $j\colon M_L\to M_L$ which is not surjective, and then appeal to Theorem 5.2.(2). For example, consider the map  $j\colon M_L\to M_L$ given by $j(a)=b$, where:
\begin{itemize}
\item $b(\ell)=a(\ell)+1$, if $\ell=\ell_n$ for some $n\in\mathbb{N}$ and  for all $m<n$ we have $a(\ell_m)=-1$;
\item $b(\ell)=a(\ell)$, otherwise.
\end{itemize}
Clearly $j$ is not surjective, since the constant zero  map $L \to \mathbb{Z}$ is not in the range of $j$. The reader may check that $j$ is an $\mathcal{L}_{\infty,\omega}$-elementary embedding by expressing $j$ as a point-wise limit of a sequence of automorphisms of $M_L$, as in the proof of \cite[Theorem 3.3]{Gao1998}.
\end{remark}

We now have a complete proof of Theorem \ref{thm:scattered_iff_cli} and Proposition \ref{prop:scattered_iff_cli}.
The main application of these  is Theorem \ref{T:CompleteCoanalytic}, whose proof appears in Section \ref{SS:PI11Rank}.

\subsection{The class of CLI Polish groups is $\BPi^1_1$-complete}\label{SS:PI11Rank}

In \cite{Malicki2011}, Malicki showed that the set of CLI Polish groups in the standard Borel space of all Polish groups is $\BPi^1_1$ but not $\BSigma^1_1$. We may now prove Theorem \ref{T:CompleteCoanalytic}, which   strengthens Malicki's result to the following: 
\begin{center}
 \emph{CLI Polish groups form a $\BPi^1_1$-complete set in the space of all Polish groups.}
\end{center}

In order to make this statement precise recall that, for any Polish group $G$, the collection $\mathcal{SG}(G)$ of all closed subgroups of $G$ is a Borel subset of the space $\mathcal{F}(G)$ of all closed subsets of $G$, endowed with the  Effros Borel structure; see \cite[Section 12.C]{Kechris1995}. As a consequence $\mathcal{SG}(G)$ is a standard Borel space; see e.g.  \cite[Proposition 1]{Malicki2011}. Let now $(\mathbb{U},d)$ be the universal Urysohn metric space \cite{Ury} and let $G_{\mathbb{U}}:=\mathrm{Iso}(\mathbb{U},d)$ be its isometry group. Recall that  $G_{\mathbb{U}}$ is a universal  Polish group: a Polish group with the property that every other Polish group  embeds in  $G_{\mathbb{U}}$; see \cite[Section 2.5]{Gao2008}. As in \cite{Malicki2011},  we take
\[\mathcal{PG}:=\mathcal{SG}(G_{\mathbb{U}})\]
to be  {\bf space of all Polish groups}. One may  endow  $\mathcal{PG}$ with a Wijsman topology  which is Polish (see \cite{Beer})  and  which induces the exact same Borel structure on  $\mathcal{PG}$ (see \cite{Wij}), but in what follows it will suffice to consider $\mathcal{PG}$  as a standard Borel space.

Let  $\mathrm{CLI}$ be the subset of $\mathcal{PG}$ consisting of all Polish groups which are CLI.  The fact that $\mathrm{CLI}$ is coanalytic is established in \cite{Malicki2011}, and can be also derived directly, from the fact that  the existence of a sequence $(g_n)_n$ in $G$ which is left-Cauchy but not right-Cauchy is an analytic condition. We are left with establishing completeness of the coanalytic set $\mathrm{CLI}$.

Let $\mathrm{LO}$ be the Polish space of all linear orders $L:=(\mathbb{N},<_L)$ on domain $\mathbb{N}$
\[\mathrm{LO} \subseteq \{0,1\}^{\mathbb{N}\times\mathbb{N}}\]
 which we identify with a closed subset  of the space of all binary relations on $\mathbb{N}$, and recall that the set $\mathrm{WO}\subseteq\mathrm{LO}$  of all wellorderings is complete coanalytic set \cite[Section 32.B]{Kechris1995}.

 By Proposition \ref{prop:scattered_iff_cli}, it suffices to show that the assignment $L\mapsto \mathrm{Aut}(M_L)$, from Section \ref{SS:Scattered_LOs}, can be realized as a Borel map from 
$\mathrm{LO}$  to  $\mathcal{PG}$. Moreover, since groups of the form  $\mathrm{Aut}(M_L)$ embed as closed subgroups of $S_{\infty}$, and since any embedding $S_{\infty}$ as a closed subgroup of $G_{\mathbb{U}}$  induces a Borel injection from $\mathcal{SG}(S_{\infty})$ to $\mathcal{PG}$, it suffices to realize
$L\mapsto \mathrm{Aut}(M_L)$ as a Borel map  from 
$\mathrm{LO}$ to  $\mathcal{SG}(S_{\infty})$.

Recall from  Section \ref{SS:Scattered_LOs}, that $M_L$ is a certain countable structure associated to $L$. By construction, since $\mathrm{dom}(L)=\mathbb{N}$ for all $L\in\mathrm{LO}$, it is natural to view every $M_L$ as a certain structure whose domain is the countable set $\mathbb{F}$  of all maps $a\colon \mathbb{N}\to \mathbb{Z}$ with $a(\ell)=0$ for all but finite $\ell\in\mathbb{N}$. 
We will realize $L\mapsto \mathrm{Aut}(M_L)$  as the composition of  the maps:
\[L \mapsto M_L  \quad \text{ and } \quad M_L\mapsto \Aut(M_L)\]
where $M_L$ takes values in a Polish space $\mathrm{Str}(\mathbb{F})$ of certain structures on domain $\mathbb{F}$ and $\Aut(M_L)$ takes values in $\mathcal{SG}(S_{\infty})$ under the identification $S_{\infty}=\mathrm{Sym}(\mathbb{F})$. In general, the map  $M\mapsto \Aut(M)$ from $\mathrm{Str}(\mathbb{F})$ to $\mathcal{SG}(\mathrm{Sym}(\mathbb{F}))$ is not Borel; see \cite[Theorem 7.1.2]{BeckerKechris1996}. The fact that the composition of the two maps above remains Borel relies on the ultrahomogeneity of each $M_L$. We may now proceed to the details.

Let  $\mathrm{Str}(\mathbb{F})$ be  the Polish space of all structures on domain $\mathbb{F}$, in the language which  consists of the binary symbols $\{\prec^*\}\cup\{E_{(\ell,z)}\colon \ell\in\mathbb{N}, z\in\mathbb{Z}\}$; see Section \ref{SS:Scattered_LOs}. That is:
\[\mathrm{Str}(\mathbb{F}):=\{0,1\}^{\mathbb{F}\times\mathbb{F}} \times \prod_{\ell\in\mathbb{N}, z\in\mathbb{Z}}\{0,1\}^{\mathbb{F}\times\mathbb{F}}\]

\begin{claim}\label{C:rev2}
The assignment $L\mapsto M_L$    from
$\mathrm{LO}$ to  $\mathrm{Str}(\mathbb{F})$ is continuous.
\end{claim}
\begin{proof}
Let $a,b\in\mathbb{F}$ and  $(\ell^*,z)\in \mathbb{N}\times\mathbb{Z}$. Let $\widetilde{A}$  and $\widetilde{B}$ be the preimages
 of the basic clopen
\[A_{ab}:=\{M\in \mathrm{Str}(\mathbb{F})\colon a\prec_M^*b\} \quad \text{ and } \quad B^{\ell,z}_{ab}:=\{M\in \mathrm{Str}(\mathbb{F})\colon aR^M_{(\ell^*,z)}b\}\]
subsets of $\mathrm{Str}(\mathbb{F})$. Then  $\widetilde{A}$ and $\widetilde{B}$ are clopen subsets of $\mathrm{LO}$. Indeed, consider the finite sets
\[F_a:=\{\ell\in \mathbb{N} \colon a(\ell)<_{\mathbb{Z}} b(\ell) \} \quad \text{ and } \quad F_b:=\{\ell\in \mathbb{N} \colon b(\ell)<_{\mathbb{Z}} a(\ell) \} \]
and notice that
\[\widetilde{A}=\bigcup_{\ell \in F_a} \bigcap_{k \in F_b} \{ <_L \text{ in } \mathrm{LO} \colon k<_L \ell  \}\]
Similarly,   $\widetilde{B}$ is either empty, if $a(\ell^*)+z\neq b(\ell^*)$; or it collects all $<_L$ in $\mathrm{LO}$ for which  the basic clopen condition $\ell<_L\ell^*$ holds for all $\ell \in F_a\cup F_b \setminus\{\ell^*\}$.
\end{proof}

Next, let $\mathrm{Uh}(\mathbb{F})$ be the collection of all ultrahomogeneous structures in  $\mathrm{Str}(\mathbb{F})$.
This turns out to be a Borel subset of  $\mathrm{Str}(\mathbb{F})$, as ultrahomogeneity is characterized by the extension property 
\cite[Lemma 7.1.4.(b)]{Hodges} ---a $G_{\delta}$ condition.
However, this will not be used here.

\begin{claim}\label{C:rev3}
The assignment $M\mapsto \mathrm{Aut}(M)$  from  $\mathrm{Uh}(\mathbb{F})$ to $\mathcal{SG}(\mathrm{Sym}(\mathbb{F}))$ is Borel.
\end{claim}
\begin{proof}
The $\sigma$-algebra of Borel subsets of $\mathcal{F}(\mathrm{Sym}(\mathbb{F}))$ is generated by sets of the form 
\[N(U):=\{ F \in \mathcal{F}(\mathrm{Sym}(\mathbb{F})) \colon F \cap U\neq \emptyset\},\]
where  $U$ ranges over all basic open sets of $\mathrm{Sym}(\mathbb{F})$. 
Consider now a basic open
\[U_{\bar{a}\bar{b}}:=\big\{f \in \mathrm{Sym}(\mathbb{F}) \colon f(a_i)=b_i  \text{ for all } i\leq n  \big\}\]
where $\bar{a}=(a_0,\ldots,a_n)$ and $\bar{b}=(b_0,\ldots,b_n)$ are tuples in $\mathbb{F}$.
But, by ultrahomogeity, the structures $M\in\mathrm{Uh}(\mathbb{F})$ with $\mathrm{Aut}(M)\in N(U_{\bar{a}\bar{b}})$ are precisely those $M\in\mathrm{Uh}(\mathbb{F})$ for which the  $\bar{a}$ and $\bar{b}$ have the same quantifier-free type, which is a closed condition.   
\end{proof}

We may now conclude with the proof of Theorem \ref{T:CompleteCoanalytic}.

\begin{proof}[Proof of Theorem \ref{T:CompleteCoanalytic}]
    By Lemma \ref{L:rev1} and the last two claims, $L \mapsto M_L$ and $M_L\mapsto \Aut(M_L)$ compose to a Borel map $f\colon \mathrm{LO}\to \mathcal{PG}$.  
 By Proposition \ref{prop:scattered_iff_cli},  we have 
\[L\in\mathrm{WO} \iff f(L)\in\mathrm{CLI}.\] 
It follows that $\mathrm{CLI}$ is complete coanalytic, since so is $\mathrm{WO}$; see \cite[Section 32.B]{Kechris1995}.
\end{proof}

\section{A boundedness principle for analytic classes of CLI groups}\label{S:Boundedness}
Let  $\mathcal{PG}$ be the standard Borel space of all Polish groups and  $\mathrm{CLI}\subseteq \mathcal{PG}$ be the complete coanalytic subset of all CLI  groups in $\mathcal{PG}$; see Section \ref{SS:PI11Rank}.
The general structure theory of coanalytic sets yields a function \[\varphi  \colon \mathrm{CLI} \to \omega_1\]
which constitutes a {\bf regular $\BPi^1_1$-rank on  $\mathrm{CLI}$}. That is, a function as above, which if we  naturally extend to $\varphi  \colon \mathcal{PG} \to \omega_1\cup \{\infty\}$, setting $\varphi(\mathcal{PG}\setminus    \mathrm{CLI} )=\{\infty\}$, and let for $H,G\in\mathcal{PG}$:
\begin{eqnarray}
H\leq_{\varphi} G \; &\iff & \;   \varphi(H)\leq \varphi(G)  \text{ and }   \varphi(H)<\omega_1 \label{EqPHI1} \\
H<_{\varphi} G \; &\iff & \;  \varphi(H)<\varphi(G),\label{EqPHI2} 
\end{eqnarray}
we have that both   $\leq_\varphi,<_\varphi$  are coanalytic  subsets of  $\mathcal{PG} \times  \mathcal{PG}$; see \cite[Section 34.B,C]{Kechris1995}.

Of course, this abstract ranking need not be ``natural" in any sense. A pertinent question is whether $\mathrm{CLI}$ admits a regular $\BPi^1_1$-rank  which reflects the structural properties of the class $\mathrm{CLI}$. To  quote \cite[page 270]{Kechris1995}:
   \emph{In many concrete situations, however, it is important to be able to find a ``natural"  $\BPi^1_1$-rank on a given $\BPi^1_1$ set which reflect the particular structure of this set.} The following is the main result of this section.

\begin{theorem}\label{T:RegularRank}
The assignment $G \mapsto \rk(G)$ is a regular $\BPi^1_1$-rank on $\mathrm{CLI}$.
\end{theorem}

Our strategy is to associate each group $G \in \PG$ with an open game where Player I has a winning strategy if and only if $G$ is CLI.
This map can be performed in a Borel fashion, and moreover the set of open games in which Player I has a winning strategy has a regular $\BPi^1_1$-rank which happens to correspond to $\rk$ on $\mathrm{CLI}$.

\subsection{Review of infinite games}

We begin with a brief review of infinite games and  strategies.
Let $X$ be a countable discrete set, and let $A \subseteq X^\omega$ be any subset.
A \textbf{game on $X$ with payoff set $A$} is a combinatorial game between two players, Player I and Player II, in which both players alternate in selecting elements of $X$ in order to construct an infinite sequence or ``run" $(a_i)_{i \in \mathbb{N}}$.
Player I wins if the run is in $A$, and Player II wins otherwise.

A winning strategy for I is a function $\sigma: X^{<\omega} \rightarrow X$ such that for any sequence $a_1, a_3, \ldots $ if $a_0, a_2, \ldots $ is recursively defined by $a_{2n} = \sigma(a_0, \ldots , a_{2n-1})$, then $(a_i)_{i \in \mathbb{N}} \in A$.
Conversely, a winning strategy for II is a function $\sigma: X^{<\omega} \rightarrow X$ such that for any sequence $a_0, a_2, \ldots $ if $a_1, a_3, \ldots $ is recursively defined by $a_{2n+1} = \sigma(a_0, \ldots , a_{2n})$, then $(a_i)_{i \in \mathbb{N}} \not\in A$.
We say $A$ is \textbf{determined} if and only if either Player I or Player II has a winning strategy.
We let $\sigma * \langle a_1, a_3, \ldots  \rangle$ denote the run as constructed when $\sigma$ is a winning strategy for I, and $\sigma * \langle a_0, a_2, \ldots  \rangle$ denote the run as constructed when $\sigma$ is a winning strategy for II.

The study of determined sets is a deep and interesting theory which includes the study of the axiom of determinacy (an axiom stating that every set is determined, which happens to be incompatible with the axiom of choice under the ordinary rules of mathematics), as well as a celebrated result of Martin that every Borel set is determined when viewed as a subset of the product topological space $X^{\mathbb{N}}$ with the discrete topology on $X$.
A more basic fact, however, is the result of Gale-Stewart that open sets (and thus also closed sets) are determined.
The proof of this is straightforward, but we will give a refinement of this result which strongly connects with our refinement of the theory of CLI Polish groups.

We will want to consider the standard Borel space $\OPEN_X$ of all open games on $X$. Associating each game with its (open) payoff, we simply set $\OPEN_X:= \mcO(X^\omega)$ to be
the space of open subsets of $X^\omega$. The standard Borel structure on 
$\mcO(X^\omega)$ is the pull-back of the Borel structure on the Effros Borel space 
$\mathcal{F}(X^\omega)$ of closed subsets on $X^\omega$, under the bijection $\mcO(X^\omega)\to \mathcal{F}(X^\omega)$ of taking complements $A \mapsto X^\omega \setminus A$.

Let $\OPEN^I_X$ be the set of all open games in which Player I has a winning strategy, and let $\OPEN^{II}_X$ be the set of open games in which Player II has a winning strategy.
By the Gale-Stewart result, $\OPEN_X = \OPEN^I_X \bigsqcup \OPEN^{II}_X$, and it's straightforward to calculate that $\OPEN^I_X$ is a coanalytic set and $\OPEN^{II}_X$ is an analytic set.

Given an open game $A \in \OPEN_X$ and a partial run $\bar{x} = x_0 \dots x_n \in X^n$ of even length, we define the \textbf{game rank of $A$ and $\bar{x}$}, denoted as follows.
We say $\Grk(A, \bar{x}) \le 0$ iff $N_{\bar{x}} := \{\bar{a} \in X^\omega \mid \bar{x} \sqsubseteq \bar{a}\}$ is a subset of $A$.
More generally, for $\alpha > 0$, we say $\Grk(A, \bar{x}) \le \alpha$ iff there is some $x_{n+1} \in X$ such that for every $x_{n+2} \in X$, we have $\Grk(A, \bar{x}x_{n+1}x_{n+2}) < \alpha$.
We write $\Grk(A, \bar{x}) = \alpha$ iff $\alpha$ is the least ordinal in which $\Grk(A, \bar{x}) \le \alpha$ holds, or $\Grk(A, \bar{x}) = \infty$ iff there is no such $\alpha$.
We define $\Grk(A)$ to be $\Grk(A, \langle \rangle)$.
It's straightforward to check that $\Grk(A) < \infty$ iff $\Grk(A) < \omega_1$ iff Player I has a winning strategy for $A$.

The fact that the map $A \mapsto \Grk(A)$ is a regular $\BPi^1_1$-rank on $\OPEN^I_X$ can be proved by an appeal to \cite[34.18]{Kechris1995}, but for completeness we sketch an argument here.
Given two open games $A, B \in \OPEN_X$, we consider two games 
$G^{\le}_{A, B}$ and $G^{<}_{A, B}$ on $X \times X$.
Both games have the same format but have different payoff sets.
First, Player I makes a move in game $A$ as Player I.
Next, Player II makes a move in game $A$ as Player II, \emph{and} makes a move in game $B$ as Player I.
Then, Player I makes a move in game $A$ as Player I, \emph{and} makes a move in game $B$ as Player II.
The game continues like this to produce a sequence $(a_i, b_i)_{i \in \mathbb{N}}$ where $(a_i)_{i \in \mathbb{N}}$ is a run of game $A$ and $(b_{i+1})_{i \in \mathbb{N}}$ is a run of game $B$.
Since Player I does not make a move in game $B$ in their first move, we treat $b_0$ as a ``dummy" move which is discarded.
We define payoff sets:
\[ G^{\le}_{A, B} := \{(a_i, b_i)_{i \in \mathbb{N}} \mid \forall n \in \mathbb{N}, N_{b_1 \dots b_n} \subseteq B \rightarrow N_{a_0 \dots a_{n-1}} \subseteq A\}\]
and
\[ G^{<}_{A, B} := \{(a_i, b_i)_{i \in \mathbb{N}} \mid \exists n \in \mathbb{N}, N_{a_0 \dots a_{n-1}} \sqsubseteq A \; \text{and} \; N_{b_1 \dots b_n} \not\sqsubseteq B\}.\]
Note again that $b_0$ is ignored in both cases because in the very first move of the game, Player I does not make a move in game $B$.
\begin{lemma}
    The games $G^{\le}_{A, B}$ and $G^{<}_{A, B}$ are open, and moreover:
    \begin{enumerate}
        \item Player I has a winning strategy in $G^{\le}_{A, B}$ if and only if $\Grk(A) \le \Grk(B)$; and
        \item Player I has a winning strategy in $G^{<}_{A, B}$ if and only if $\Grk(A) < \Grk(B)$. 
    \end{enumerate}
\end{lemma}

\begin{proof}
The set $G^{<}_{A, B}$ is evidently open, but for $G^{\le}_{A, B}$ we must see that this can be written as a union over $n$ of open sets 
\[\{(a_i, b_i)_{i \in \mathbb{N}} \mid \forall k \le n, N_{b_1 \dots b_n} \subseteq B \rightarrow N_{a_0 \dots a_{n-1}} \subseteq A\}.\]

For (1), let's consider the forward direction.
Suppose Player I has a winning strategy for $G^{\le}_{A, B}$.
We prove that $\Grk(A) \le \Grk(B)$ by proving the following stronger claim.

\begin{claim}
Suppose $\bar{x} = (a_0, b_0) \dots (a_n, b_n)$ is an odd-length run of the game $G^{\le}_{A, B}$ played according to Player I's winning strategy.
Then we claim that 
\[\Grk(A, a_0 \dots a_{n-1}) \le \Grk(B, b_1 \dots b_n).\]
\end{claim}

We prove by transfinite induction on $\alpha$, simultaneously for all such $\bar{x}$, that if $\Grk(B, b_1 \dots b_n) \le \alpha$ then $\Grk(A, a_0 \dots a_{n-1}) \le \alpha$.

For the base case, assume $\Grk(B, b_1 \dots b_n) \le 0$.
This means that $N_{b_1 \dots b_n} \subseteq B$.
Since Player I's strategy in $G^{\le}_{A, B}$ is winning, we can extend $\bar{x}$ to a sequence $(a_i, b_i)_{i \in \mathbb{N}}$ which lands in the payoff set for $G^{<}_{A, B}$.
By the definition of this payoff set, we have $N_{a_0 \dots a_{n-1}} \subseteq A$ and thus $\Grk(A, a_0 \dots a_{n-1}) \le 0$.

Now suppose $\alpha > 0$, and $\Grk(B, b_1 \dots b_n) \le \alpha$, and assume the claim holds below $\alpha$.
We claim $\Grk(A, a_0 \dots a_{n-1}) \le \alpha$.
By its recursive definition, we just need to show that for an arbitrary $a_{n+1}$,
$\Grk(A, a_0 \dots a_n a_{n+1}) < \alpha$.
Choose some $b_{n+1}$ such that for every $b_{n+2}$ we have $\Grk(B, b_1 \dots b_{n+1} b_{n+2}) \le \alpha$.
Suppose that we have Player II make the move $(a_{n+1}, b_{n+1})$ next in the game $G^{<}_{A, B}$ and $(a_{n+2}, b_{n+2})$ is Player I's response according to their winning strategy.
In particular, we have $\Grk(B, b_1 \dots b_{n+2}) \le \alpha$ and so by the induction hypothesis we also have $\Grk(A, a_0 \dots a_{n+1}) \le \alpha$ as desired. This ends the proof of the claim and thus the forward direction of (1).

For the reverse direction suppose $\Grk(A) \le \Grk(B)$.
We will construct a winning strategy for $G^{\le}_{A, B}$.
For the first move, since $\Grk(A) \le \Grk(B)$ we may select some $a_0$ so that for every $b_1$ and $a_1$ there is $b_2$ with $\Grk(A, a_0a_1) \le \Grk(B, b_1b_2)$.
We have Player I play $(a_0, b_0)$ where $b_0$ is the ignored dummy move, and suppose $(a_1, b_1)$ is Player II's response.

In general, suppose $\bar{x} = (a_0, b_0) \dots (a_n, b_n)$ is an even-length run of the game $G^{\le}_{A, B}$ and so far Player I has been able to maintain for some $b_{n+1}$ that
\[\Grk(A, a_0 \dots a_k) \le \Grk(B, b_1 \dots b_{k+1})\]
for all even $1 \le k \le n$.
Since $\Grk(A, a_0 \dots a_n) \le \Grk(B, b_1 \dots b_{n+1})$ we may find some $a_{n+1}$ such that for every $b_{n+1}$ and $a_{n+2}$ there is $b_{n+2}$ such that
\[\Grk(A, a_0 \dots a_{n+1}) \le \Grk(B, b_1 \dots b_{n+2}).\]
Have Player I play $(a_{n+1}, b_{n+1})$ and suppose Player II's response is $(a_{n+2}, b_{n+2})$.
Then fix $b_{n+3}$ such that
\[\Grk(A, a_0 \dots a_{n+2}) \le \Grk(B, b_1 \dots b_{n+3}).\]
Thus Player I of $G^{\le}_{A, B}$ can always maintain the property above and by definition of $G^{\le}_{A, B}$ will be a winning strategy.

Proceeding to (2), by determinacy of open games, it suffices to show that Player II has a winning strategy in $G^{<}_{A, B}$ if and only if $\Grk(B) \le \Grk(A)$.
This argument proceeds in essentially the same way as in the proof of (1).
\end{proof}

Given that $\OPEN^I_{X}$ is coanalytic, we finish by observing that the maps $(A, B) \mapsto G^{\le}_{A, B}$ and $(A, B) \mapsto G^{<}_{A, B}$ are Borel as maps from $\OPEN_X \times \OPEN_X$ to $\OPEN_{X \times X}$.
We conclude that $\le_{\varphi}$ and $<_{\varphi}$ are coanalytic, where $\varphi : \text{OPEN}_X \rightarrow \omega_1 \cup \{\infty\}$ is the map $A \mapsto \Grk(A, \langle \rangle)$. 

\subsection{The CLI game}
Next we define the CLI game and connect it to our rank notion.

Let $G^0_{\mathbb{U}}$ be a countable dense subgroup of $G_{\mathbb{U}}$.
Assume $G^0_{\mathbb{U}}$ includes the identity.
Let $\mcB$ be a countable local basis of the identity of $G_{\mathbb{U}}$ of symmetric \emph{regular} open sets which is closed under conjugation by elements of $G^0_{\mathbb{U}}$.
Assume $\mcB$ includes the whole group $G_{\mathbb{U}}$. 
Let $G \in \PG$, in which case $G$ is a closed subgroup of $G_{\mathbb{U}}$.
For any $V \in \mcB$, the \emph{CLI game for $V$ and $G$} is the game played on $\mcB$ as follows:
\begin{itemize}
    \item In the first round, Player I chooses some $U_0 \in \mcB$;
    \item In the second round, Player II chooses some conjugate $W_0 := gU_0g^{-1}$ where $g \in G^0_{\mathbb{U}}$;
    \item In round $2n+1$, Player I chooses some $U_{n} \in \mcB$;
    \item In round $2n+2$, Player II chooses some  $W_n := gU_ng^{-1}$ where $g \in (G \cap W_{n-1})U_{n}  \cap G^0_{\mathbb{U}}$.
\end{itemize}

Player I wins if for some $n$, we have $W_n \cap G \subseteq V$ (we call this ``the win condition'').
Otherwise, Player II wins.
This is easily an open game, which we denote by $\mcC_{V, G}$.

\begin{proposition}
    For every $G \in \PG$ and $V \in \mathcal{B}$, we have $\mathrm{Grk}(\mcC_{V, G}) = \rk(V, G)$.
\end{proposition}

\begin{proof}
    We prove by induction on $\alpha$ simultaneously for any partial even-length run $\bar{x}$ that, if $W$ is Player II's last move in $\bar{x}$ (or $W = G_{\mathbb{U}}$ if $\bar{x} = \langle \rangle$), we have
    \[ \Grk(\mcC_{V, G}, \bar{x}) \le \alpha \quad \text{iff} \quad \rk(V, W \cap G) \le \alpha.\]
    Assume $\alpha = 0$.
    The backwards direction is immediate.
    Indeed, if $\rk(V, W \cap G) \le 0$ and thus $W \cap G \subseteq V$, then that means Player I has satisfied the explicit win condition of the CLI Game for $V$.
    
    The forward direction is slightly subtle.
    Assume $\rk(V, W \cap G) \le 0$ fails, thus $W \cap G \not\subseteq V$.
    Of course, this means that Player I has not satisfied the win condition at this stage, but we need to show that $\mathcal{N}_{\bar{x}}$ is not contained in the payoff set.
    This can be done by observing Player I can always avoid the win condition by simply choosing to repeatedly play the entirety of $G_{\mathbb{U}}$ on their turn for the rest of the game.

    Now let $\alpha > 0$ and assume the claim is true below $\alpha$.

    We start with the forward direction.
    Suppose $\Grk(\mcC_{V, G}, \bar{x}) \le \alpha$ where $W$ is Player II's last move.
    Then there is some $U \in \mathcal{B}$ such that for every $g \in (G \cap W)U \cap G^0_{\mathbb{U}}$ we have $\Grk(\mcC_{V,  G}, \bar{x} \langle U, gUg^{-1} \rangle) < \alpha$.
    Choose symmetric $U' \in \mathcal{B}$ such that $U' \subseteq W$ and $(U')^3 \subseteq U$.
    It suffices to show that for any $g \in W \cap G$ we have $\rk(V, gU'g^{-1}) < \alpha$.
    Fix any such $g \in W \cap G$, and choose $\epsilon \in U'$ so that $g\epsilon \in G^0_{\mathbb{U}}$. 
    Then we have that $\Grk(\mcC_{V,  G}, \bar{x} \langle U, g\epsilon U\epsilon^{-1}g^{-1} \rangle) < \alpha$.
    Then we have by the induction hypothesis that $\rk(V, g\epsilon U\epsilon^{-1}g^{-1} \cap G) < \alpha$.
    Then by monotonicity (Lemma \ref{prop:basic_properties_of_rk}.(1)) we have that $\rk(V, g U' g^{-1} \cap G) < \alpha$.

    We proceed to the reverse direction.
    Fix some $U \subseteq_1 G$ such that for every $g \in W \cap G$ we have $\rk(V, gUg^{-1}) < \alpha$.
    Choose symmetric $U' \in \mathcal{B}$ with $U' \subseteq W$ such that $(U')^3 \cap G \subseteq U$.
    Fix any $g \in (G \cap W)U' \cap G^0_{\mathbb{U}}$.
    We wish to show that $\Grk(\mcC_{V, G }, \bar{x} \smf \langle U', gU'g^{-1} \rangle) < \alpha$.
    By the induction hypothesis, it suffices to show $\rk(V, gU'g^{-1} \cap G) < \alpha$.
    Indeed, write $g = hu$ for $h \in G \cap W$ and $u \in U'$.
    Then we have $\rk(V, hUh^{-1}) < \alpha$.
    Observe that
    \[gU'g^{-1} \cap G = huU'u^{-1}h^{-1} \cap G \subseteq hUh^{-1}\]
    so by monotonicity we have $\rk(V, gU'g^{-1} \cap G) < \alpha$ as desired.
\end{proof}

Now we define the CLI game for $G$ as follows.
In round 1, Player I does not make a move.
In round 2, Player II chooses some $V \in \mcB$.
Then Player I and Player II begin playing $\mcC_{V, G}$ to determine the winner.
Let $\mcC_G$ denote the CLI game for $G$.
The following is immediate:

\begin{proposition}
    $\mathrm{Grk} (\mcC_G) = \rk(G)$.
\end{proposition}

What remains is to show that the map $G \mapsto \mcC_G$ is Borel as a function from $\PG$ to $\OPEN_{\mcB}$.
The standard Borel space on $\OPEN_{\mcB} = \mathcal{O}(\mathcal{B}^{\omega})$ is generated by sets $\{C \in \mathcal{O}(\mathcal{B}^{\omega}) \mid A \not\subseteq C\}$ where $A$ ranges over basic open subsets of $\mcB^\omega$; see \cite[Section 12.C]{Kechris1995}.
Suppose $A$ is the basic open subset of all runs that start with
\[ \langle *, V, U_0, W_0, \ldots , U_n, W_n \rangle.\]
We want to show that the set of $G \in \PG$ such that $A \not\subseteq \mcC_G$ is Borel.
This is precisely the set of $G$ for which there is no extension
\[  U_{n+1}, W_{n+1}, \ldots , U_m, W_m\]
which enters the payoff set.
This is a countable universal statement, so it is enough to check that given the run
\[ \langle *, V, U_0, W_0, \ldots , U_m, W_m \rangle \]
the following sets of $G$ are Borel:
\begin{enumerate}
    \item The set of $G$ in which this run achieves Player I's win condition;
    \item The set of $G$ in which this run breaks the rules for Player I;
    \item The set of $G$ in which this run breaks the rules for Player II.
\end{enumerate}

For the first, we are looking for the set of $G$ in which $W_m \cap G \subseteq V$.
By the regularity of $V$, this is precisely the complement of the set of $G$ in which $G \cap (W_m \cap \overline{V}^c) \neq \emptyset$, which is Borel in Effros Borel structure.
The second set is always either all of $\PG$ or $\emptyset$, as the rules for Player I do not depend on $G$.
For the third, we are looking for the set of $G$ in which there is some $0 \le k < m$ such that $W_k$ was not a valid move.
Observe $W_k$ is not a valid move exactly when for every $g \in G^0_{\mathbb{U}}$ if $g \in (G \cap W_{k-1}) U_k$ then $W_k \neq gU_kg^{-1}$.
It's enough to check that for any $U_k \in \mathcal{B}$ and $g \in G_0$, the set of $G \in \PG$ with $g \in (G \cap W_{k-1}) U_k$ is Borel.
Indeed, this is equivalent to saying $G \cap (gU_k^{-1} \cap W_{k-1}) \neq \emptyset$ which is Borel in the Effros Borel structure.
This concludes our proof that the map $G \mapsto \rk(G)$ is a regular $\BPi^1_1$-rank on $\mathrm{CLI}$.

\section{Weakly $\alpha$-balanced groups}\label{S:Malicki}

In proving that CLI groups form a class of Polish groups that is not Borel, Malicki developed  a rank for Polish permutation groups  similar to Deissler's rank. 
Here, we show that Malicki's rank  generalizes to a rank for all topological groups. 
This leads to another stratification of CLI Polish groups into the classes of  \emph{weakly $\alpha$-balanced Polish groups}. 
Here, we only briefly discuss these connections, as the theory of weakly $\alpha$-balanced Polish groups follows exactly the same lines as the theory of $\alpha$-balanced Polish groups.

Let $G$ be a topological group and $U, V \subseteq_1 G$. Consider the rank function $\rk^*$ we would get if we modify the base case of $\rk(-,-)$  by declaring $\rk^*(V, U) = 0$ if and only if $U \subseteq FV$ for some finite $F \subseteq G$.
We leave the other clauses unchanged, so $\rk^*(V,U)\leq \beta$ if there is $W\subseteq_1 G$ so that if $g\in U$, then $\rk^*(V,gW g^{-1})< \beta$, and $\rk^*(V,U)= \infty$ if there is no ordinal $\beta$, for which $\rk^*(V,U)\leq \beta$.
We say that $G$ is \textbf{weakly $\alpha$-balanced} if and only if
\[\rk^*(G):=\sup\{\rk^*(V,G)+1 \colon V\subseteq_1 G\}\leq \alpha.\]

\begin{example}\label{E:Lamp}
  Consider the   full lamplighter group $L:= \mathbb{Z} \mathrel{\mathrm{Wr}} (\mathbb{Z}/2\mathbb{Z})= \mathbb{Z}\ltimes \prod_{k\in\mathbb{Z}}  \mathbb{Z}/2\mathbb{Z}$. By Proposition \ref{Prop:Z-jump} it follows that  $\rk(L)=3$, and it is easy to check that $\rk^{*}(L)=2$.
\end{example}

The  difference between $\rk$ and $\rk^*$ (see  Example \ref{E:Lamp}) turns out to be rather marginal:

\begin{proposition}
For any Polish group $G$, we have $\rk(G) \le \rk^*(G) \le \rk(G) + 1$, and moreover $\rk(G) = \rk^*(G)$ whenever either is infinite.  
\end{proposition}

\begin{proof}
Fix any $U, V \subseteq_1 G$.
A straightforward induction shows that $\rk^*(V, U) \le \rk(V, U)$.
On the other hand, we show that $\rk(V^3, U) \le 1 + \rk^*(V, U)$.
We prove this simultaneously for every $U$ and $V$.
For the base case, suppose $\rk^*(V, U) = 0$.
Fix some finite $F \subseteq G$ such that $U \subseteq FV$.
Now we can define $W \subseteq_1 G$ to be the set 
\[ W := \bigcap_{g \in F} gVg^{-1}.\]
We show that for any $u \in U$, we have $uWu^{-1} \subseteq V^3$.
To that end, fix any $u \in U$, in which case $u^{-1} \in U$ by symmetry of $U$, and thus there is some $g \in G$ with $g^{-1}u^{-1} \in V$.
We also have $ug \in V$ by symmetry of $V$.
Then $uWu^{-1} \subseteq ugVg^{-1}u^{-1} \subseteq V^3$.

The induction step is completely routine.
Suppose $\alpha > 0$ and $\rk(V^3, U) \le 1 + \rk^*(V, U)$ for every $U$ and $V$ with $\rk^*(V, U) < \alpha$.
Now suppose $\rk^*(V, U) \le \alpha$.
We may then fix some $W \subseteq_1 G$ with $\rk^*(V, uWu^{-1}) < \alpha$ for every $u \in U$.
By the induction hypothesis we have $\rk(V^3, uWu^{-1}) \le 1 + \rk^*(V, uWu^{-1}) < 1 + \alpha$ for every $u \in U$.
Thus $\rk(V^3, U) \le \alpha$ as desired.
The proposition follows.
\end{proof}

This weaker rank easily corresponds to a weaker variation of the Deissler rank, where the base case is modified to replace definable closure with algebraic closure.
Let $\mcL$ be a language and $\mcM$ an $\mcL$-structure.
For every pair of tuples $\bar{a}, \bar{b} \in M^{<\omega}$ we define
$\Drk^*(\bar{a}, \bar{b}) = 0$ iff there is some $\mcL_{\infty, \omega}$-formula $\varphi(\bar{x}, \bar{y})$ such that $\varphi^{\mcM}(\bar{a}, \bar{b})$ holds and moreover there are only finitely many $\bar{c} \in M^{<\omega}$ such that $\varphi^{\mcM}(\bar{c}, \bar{b})$ holds.
For general $\alpha$, we define $\Drk^*(\bar{a}, \bar{b})$ inductively the exact same way as $\Drk(\bar{a}, \bar{b})$ but replacing every instance of $\Drk$ with $\Drk^*$.
By essentially the same argument as Proposition \ref{Prop:RanksAreTheSame}, we have $\Drk^*(\bar{a}, \bar{b}) = \rk^*(V_{\bar{a}}, V_{\bar{b}})$.

\subsection{Malicki's orbit tree}

Let $P$ be a closed subgroup of $\mathrm{Sym}(\mathbb{N}_{+})$, where $\mathbb{N}_{+} := \{1, 2, \ldots \}$.
For any $n \ge 0$, let $[n] = \{1, \ldots , n\}$.
For any $\bar{a} \in (\mathbb{N}_{+})^{< \omega}$, and any $n \ge 0$, let $\mcO^n_{\bar{a}}$ be the orbit of $\bar{a}$ under the stabilizer $P_{[n]}$ of $[n]$ in $P$, acting with the diagonal action on $(\mathbb{N}_{+})^{|\bar{a}|}$.

In \cite[Section 3]{Malicki2011} Malicki defines the  \textbf{orbit tree $T_P$ of $P$}  to be the tree on all infinite $\mcO^n_{\bar{a}}$ ordered by reverse inclusion.
Note that Malicki only defined this for orbits of individual points rather than tuples, but the definition generalizes naturally. The {\bf Malicki rank}  $\mathrm{Mrk}(P)$ of $P$ is then  defined to be  Cantor-Bendixson rank $\CBrk(T_P)$ of the tree $T_P$, where the Cantor-Bendixson rank of a tree $T$ is defined as follows. 
Given $a \in T$, if $a$ is a leaf of $T$ then $\CBrk(a) = 0$.
For $\alpha > 0$, we inductively define $\CBrk(a) \le \alpha$ iff $\CBrk(b) < \alpha$ holds for every $b \in T$ with $b \ge a$.
If $\CBrk(a) \le \alpha$ does not hold for any $\alpha$, then  $\CBrk(a) = \infty$, otherwise  $\CBrk(a) = \alpha$ for the least $\alpha$ so
that $\CBrk(a) \le \alpha$ holds.
Finally, $\CBrk(T)$ is defined to be the sup of $\CBrk(a)+1$ where $a$ ranges over all of $T$.

\begin{proposition}
For any $P\leq\mathrm{Sym}(\mathbb{N}_{+})$, we have $\rk^*(P)$ is precisely $\mathrm{Mrk}(P)$.
\end{proposition}

\begin{proof}
    We prove that $\rk(V_{\bar{a}}, V_{[n]}) = \CBrk(\mcO^n_{\bar{a}})$.
    Specifically, we prove by induction that for any ordinal $\alpha$, we have $\rk^*(V_{\bar{a}}, V_{[n]}) \le \alpha$ if and only if $\CBrk(\mcO^n_{\bar{a}}) \le \alpha$.

    For $\alpha = 0$, we simply observe that $\rk^*(V_{\bar{a}}, V_{[n]}) \le 0$ exactly when $\mcO^n_{\bar{a}}$ is finite, which happens exactly when it is a leaf in $T_G$.
    Now let $\alpha > 0$ and assume for every $\beta <\alpha$ we have $\rk^*(V_{\bar{a}}, V_{[n]}) \le \beta$ if and only if $\CBrk(\mcO^n_{\bar{a}}) \le \beta$.
    Observe that $\rk^*(V_{\bar{a}}, V_{[n]}) \le \alpha$ if and only if there is some $W \subseteq_1 G$ with $\rk^*(V_{\bar{a}}, vWv^{-1}) < \alpha$ for every $v \in V_{[n]}$.
    By Proposition \ref{prop:basic_properties_of_rk}.(1) we may as well take $W$ to be $V_{[m]}$ for some $m > n$, and by Proposition \ref{prop:basic_properties_of_rk}.(2) this is equivalent to saying that $\rk^*(V_{\bar{c}}, V_{[m]}) < \alpha$ for every $\bar{c} \in \mcO^n_{\bar{a}}$.
    By the induction hypothesis, this is equivalent to saying that $\CBrk(\mcO^m_{\bar{c}}) < \alpha$ for every $\bar{c} \in \mcO^n_{\bar{a}}$.
\end{proof}

\begin{LARGE}
\part{The dynamics of $\alpha$-balanced Polish groups}\label{PartII}
\end{LARGE}

In Part II of this paper we explore the consequences  of $\alpha$-balancedness of a Polish group  $G$  to its dynamics. Our first goal is to prove Theorem \ref{thm:main2} which  establishes that  \emph{generic $\alpha$-unbalancedness}  is an obstruction to classification by  $\alpha$-balanced groups. 
The proof of Theorem \ref{thm:main2} is given in Section  \ref{S:II:Proof1st} and it relies on a series of intermediate results from Sections  \ref{S:II:GenericAlpha},  \ref{S:II:BandF},  
\ref{S:II:Criterion}, which are   interesting in their own right.
 Our second goal is to prove Theorem \ref{T:mainBernoulli} which provides, for each $\alpha<\omega_1$, an example of an action of some $\alpha$-balanced Polish group (by Lemma \ref{lem:scattered_cli}), which is  generically $\beta$-unbalanced for all $\beta<\alpha$.
 
 In Section \ref{S:II:Examples} we discuss these actions and we introduce the ``fusion lemma", see Lemma \ref{L:Ind}. We then  show how Theorem \ref{T:mainBernoulli} follows from this lemma.  The proof of Lemma \ref{L:Ind} is then broken into two main cases which are proved in Sections \ref{S:lambda=0} and \ref{S:lambda>0}, respectively.

We now recall and fix some usual conventions that will be used throughout Part II.

\smallskip{}

\subsection{Definitions and notation} Let $G$ be a Polish group. Recall that a {\bf Polish $G$-space}   is a Polish space $X$ along with a continuous (left) action $G \curvearrowright X$. The associated {\bf orbit equivalence relation} $E^G_X$ is defined by setting $x \mathrel{E^G_X} y \iff \exists g\in G \; gx=y$ for all $x,y\in X$, and induces a classification problem $(X,E^G_X)$.    
If $V\subseteq G$,  then ``$\forall^*g\in V\, P(g)$"  stands for ``the set $\{g\in V \colon P(g)\}$ is comeager  in $V$" and ``$\exists^*g\in V\, P(g)$"  stands for ``the set $\{g\in V \colon P(g)\}$ is nonmeager in $V$". Here $V$ has been endowed with the subspace topology.
If $A\subseteq X$ and $V\subseteq G$ then the {\bf Vaught transforms}  of $A$ in $G\curvearrowright X$ are the sets given by  
\[A^{*V}:=\{x\in X\colon \forall^* g\in V, \, gx\in A \} \quad \text{ and } \quad A^{\Delta V}:=\{x\in X\colon \exists^* g\in V, \, gx\in A \}.\]

A {\bf classification problem} is a pair $(X,E)$ where $X$ is a Polish space and $E$ is an analytic equivalence relation on $X$.
Let $(X,E)$ and $(Y,F)$ be two classification problems.  A map $f\colon X\to Y$ is an {\bf $(E,F)$-homomorphism} if for all $x,y\in X$ we have that $xEy\implies f(x)Ff(y)$. It is a {\bf reduction from $E$ to $F$} if  $xEy\iff f(x)Ff(y)$ holds for all $x,y\in X$. We say $E$ is {\bf generically ergodic with respect to $F$} if for every Baire-measurable $(E,F)$-homomorphism $f\colon X\to Y$ there is a comeager $C\subseteq X$ so that $f(x) F f(y)$ for all $x,y\in C$. Let $H$ be a Polish group. We say that $E$ is {\bf generically ergodic against actions of $H$} if for every Polish $H$-space $Y$ we have that $E$ is generically ergodic with respect to $E^H_Y$.

\section{Generic $\alpha$-unbalancedness}\label{S:II:GenericAlpha}

 In this section we develop some basic theory around generic $\alpha$-unbalancedness. We start by recalling the main definitions from the introduction.

\begin{definition}\label{Def:Warrow2}
Let $G\curvearrowright X$ be a Polish $G$-space and let  $V\subseteq_1 G$.
We recursively define binary relations $x \leftrightsquigarrow^{\alpha}_V y$ on $X$ for any ordinal $\alpha<\omega_1$ as follows:
\begin{enumerate}
\item   $x \leftrightsquigarrow^{0}_V y$   iff both $x \in \overline{V \cdot y}$ and $y \in \overline{V \cdot x}$ hold.
 \item  $x\leftrightsquigarrow^{\alpha}_V y$  iff  for all open $W\subseteq_1 G$ and  $U\subseteq X$, with $x\in U$ or $y\in U$,   there exist   $g^x,g^y\in V$ with $g^xx,g^yy\in U$ so that   $g^x x \leftrightsquigarrow^{\beta}_W  g^y y$ holds for every $\beta<\alpha$. 
\end{enumerate}  
\end{definition}

\begin{definition}\label{D:GenericUnbalancedness2}
We say that  $G\curvearrowright X$  is {\bf generically $\alpha$-unbalanced}  for some $\alpha>0$,
if   for every comeager $C\subseteq X$ there is a comeager $D\subseteq C$ so that for all $x,y\in D$  there exists a  finite sequence $x_0,\ldots, x_n\in C$  with $x_0=x, x_n =y$, so that for all $\beta<\alpha$ we have that:
\[x_0 \leftrightsquigarrow^{\beta}_G x_1 \leftrightsquigarrow^{\beta}_G \cdots \leftrightsquigarrow^{\beta}_G x_{n-1} \leftrightsquigarrow^{\beta}_G x_n. \]
\end{definition}

The following lemma summarizes some basic properties that  we will  use repeatedly.

\begin{lemma}\label{Aris:LSimpleSquiggleProperties}
Let $G\curvearrowright X$ be a Polish $G$-space and let $x,y\in X$ and   $V\subseteq_1 G$ with  $ x \leftrightsquigarrow^\alpha_{V} y$.
The following all hold:
\begin{enumerate}
\item the symmetric statement $y \leftrightsquigarrow^\alpha_{V} x$ holds;
\item for any $\beta<\alpha$, we also have $ x \leftrightsquigarrow^\beta_{V} y$;
\item for any $W\supseteq V$, we also have $ x \leftrightsquigarrow^\alpha_{W} y$;
\item for any $g\in G$, we also have  $ gx \leftrightsquigarrow^\alpha_{g Vg^{-1}} g y$; and
\item if $D\subseteq G$ is comeager  and  $ x \leftrightsquigarrow^{\alpha,D}_{V} y$ is  defined exactly as  $ x \leftrightsquigarrow^{\alpha}_{V} y$, except that  in  Definition \ref{Def:Warrow2}(2) we require $g^x,g^y\in V\cap D$, then we also have   $ x \leftrightsquigarrow^{\alpha,D}_{V} y$.
\end{enumerate}
\end{lemma}
\begin{proof}
The first three properties follow  straight from the definition of  $ x \leftrightsquigarrow^\alpha_{V} y$. For  (4), let  $U\subseteq X$ and  $W\subseteq_1 G$ be open, say with $gy\in U$ (the case $gx\in U$ is similar).
Since   $g^{-1}U$ contains $y$,  $g^{-1} W g$ contains $1_G$, and  $ x \leftrightsquigarrow^\alpha_{V} y$ holds, we can find $h^x,h^y\in V$ with $h^xx,h^yy\in g^{-1}U$ and  $h^x x \leftrightsquigarrow^{\beta}_{g^{-1} W g}  h^y y$  for all $\beta<\alpha$. But then $gh^xx, gh^yy\in U$ and by inductive hypothesis   $gh^x x \leftrightsquigarrow^{\beta}_{W }  gh^y y$. Setting $g^x:=g h^x g^{-1}$ , $g^y:=g h^y g^{-1}$, we found  $g^x,g^y\in g Vg^{-1}$  so that $g^{x}(gx)\in U$,  $g^{y}(gy)\in U$ and 
$g^{x} (g  x)\leftrightsquigarrow^{\beta}_{W} g^{y}  (g y)$, as desired.

For (5), let  $W\subseteq_1 G$ and $U\subseteq X$ open with $x\in U$  (the case $y\in U$ is similar). Find $\widehat{W}\subseteq_1G$ with $\widehat{W}=\widehat{W}^{-1}$ and $\widehat{W}^3\subseteq W$. Since   $ x \leftrightsquigarrow^\alpha_{V} y$ holds  we can find $g^x,g^y\in V$ so that  $g^xx \leftrightsquigarrow^{\beta}_{\widehat{W}} g^yy$ for all $\beta<\alpha$ . Since $D$ is comeager in $G$ we can find  $g \in \widehat{W}$ so that $gg^x, gg^y \in  V\cap D$ and    $(gg^x)x, (gg^y)y\in U$. By  (3) and (4)  we have    $(gg^x)x \leftrightsquigarrow^{\beta}_{W} (gg^y)y$.
\end{proof}

Generic $\alpha$-unbalancedness for the action $G\curvearrowright X$ is a strong form of ``path-connectedness" for the edge relation   $ \leftrightsquigarrow^\alpha_{G}$ on $X$. What is important for what follows is that, generically, these edges push forward under Baire--measurable $(E^G_X,E^H_Y)$-homomorphisms.

\begin{theorem}\label{Th:SguigglePushesForward}
Let $X$ be a Polish $G$-space and $Y$ be  a Polish $H$-space, for Polish groups $G,H$. If $f : X \rightarrow Y$ is a Baire--measurable $(E^G_X,E^H_Y)$--homomorphism, then there exists a comeager set $C\subseteq X$ so that for all $x,y\in C$ and any countable ordinal $\alpha$ we have that:
\[x \leftrightsquigarrow^\alpha_{G} y\implies f(x) \leftrightsquigarrow^\alpha_{H} f(y).\]
\end{theorem}

The proof of Theorem \ref{Th:SguigglePushesForward} relies on the following  ``orbit continuity" lemma. This lemma  is essentially 
\cite[Lemma 3.17]{Hjorth2000}, modified as in the beginning of the proof of \cite[Lemma 3.18]{Hjorth2000}.  For a direct proof see  \cite{LupiniPanagio2018}.

\begin{lemma}\label{L:OrbitContinuity}
Let $X$ be a Polish $G$-space and $Y$ a Polish $H$-space, and let $f : X \rightarrow Y$ be a Baire-measurable $(E^G_X,E^H_Y)$-homomorphism. Then there is a comeager set $C \subseteq X$ with:
\begin{enumerate}
    \item $f$ is continuous on $C$;
    \item for every $x \in  C$, there is a comeager set of $g \in G$ such that $gx \in C$.
    \item for all  $x_0 \in C$ and  every $W \subseteq_1 H$, there is a $V \subseteq_1 G$ and open $U\subseteq X$ with $x_0\in U$ such that for all $x \in C \cap U$ and comeager many  $g \in V$ we have $f(gx) \in Wf(x)$.
\end{enumerate}
\end{lemma}


We may now turn to the proof of  Theorem \ref{Th:SguigglePushesForward}.

\begin{proof}[Proof of Theorem \ref{Th:SguigglePushesForward}]
Let  $f : X \rightarrow Y$ be a Baire--measurable  $(E^G_X,E^H_Y)$--homomorphism and let $C$ be the set from Lemma \ref{L:OrbitContinuity}.  For open sets $U\subseteq X$, $V\subseteq_1 G$,  $W\subseteq_1 H$ we say that $(U,V)$ {\bf captures} $W$ if  $\forall x\in U\cap C$ and $\forall^* g\in V$  we have  $f(gx) \in Wf(x)$ and $gx\in C$. Notice that:
\begin{enumerate}[(i)]
\item $(X,G)$ captures $H$;
\item for all  $x\in C$ and $W\subseteq_1 H$ there are open $U\ni x$, $V\subseteq_1 H$  so that  $(U,V)$ captures $W$.
\end{enumerate}
Indeed (i) follows from (2) of Lemma \ref{L:OrbitContinuity}  and (ii) follows from (2), (3) of Lemma \ref{L:OrbitContinuity}. Theorem \ref{Th:SguigglePushesForward} now follows
 from the following claim and point (i) above.

\begin{claim}
If $(U,V)$ captures $W$ and $x,y\in U\cap C$, then for any countable ordinal $\alpha$ 
 \[x \leftrightsquigarrow^\alpha_{V} y\implies f(x) \leftrightsquigarrow^\alpha_{W} f(y).\]
\end{claim}
\begin{proof}[Proof of Claim]
If $\alpha=0$ and $x \leftrightsquigarrow^0_{V} y$, then  $y \in \overline{V \cdot x}$. Since  $y\in U$ and  $(U,V)$ captures $W$, we can find a sequence $(g_n)_n$ in $V$  with $g_n x\to y$ so that for $n\in \mathbb{N}$ we have that  $f(g_nx) \in Wf(x)$ and $g_nx\in C$. Pick $h_n\in W$ so that   $h_nf(x)=f(g_n x)$. Since $f$ is continuous on $C$, we have that $f(g_n x)\to f(y)$ and hence $h_nf(x)\to f(y)$. That is $f(y)\in \overline{W \cdot f(x)}$. Similarly we show that $f(x)\in \overline{W \cdot f(y)}$. As a consequence  $f(x) \leftrightsquigarrow^0_{W} f(y)$ holds.

Assume now that the statement of the claim is true for all  ordinals $\beta$ with $\beta<\alpha$ and assume that $x \leftrightsquigarrow^\alpha_{V} y$ holds.  We will show that so does  $f(x) \leftrightsquigarrow^\alpha_{W} f(y)$. Let  $\widehat{W}\subseteq_1 H$ and   $O\subseteq Y$  open, say with  $f(y)\in O$ (the case where $f(x)\in O$ is similar). By continuity of $f\upharpoonright C$ there is some open set $\widehat{O}\subseteq X$, with $y\in \widehat{O}$, so that $f(C\cap \widehat{O})\subseteq O$. By  applying (ii) above to  $y\in C$, $\widehat{W}\subseteq_1 H$, we get $\widehat{V}\subseteq_1 G$ and the ability to shrink 
$\widehat{O}$, if necessary, so that  $(\widehat{O},\widehat{V})$ captures $\widehat{W}$, in addition to $y\in \widehat{O}$. Since  $x \leftrightsquigarrow^\alpha_{V} y$ holds, we get $g^x,g^y\in V$ with $g^xx,g^yy\in \widehat{O}$ and  $g^x x \leftrightsquigarrow^{\beta}_{\widehat{V}}  g^y y$ for all $\beta<\alpha$. By Lemma \ref{Aris:LSimpleSquiggleProperties}(5), and since $(U,V)$ captures $W$, we can arrange so that $g^x,g^y$ come from this comeager subset of $V$ which guarantees that  we also have $g^xx,g^yy\in C$ and $h^xf(x)=f(g^xx)$,  $h^yf(y)=f(g^yy)$ for some $h^x,h^y\in W$. Since $g^xx,g^yy\in \widehat{O}\cap  C$ and $(\widehat{O},\widehat{V})$ captures $\widehat{W}$, by inductive hypothesis    $g^x x \leftrightsquigarrow^{\beta}_{\widehat{V}}  g^y y$ implies that   $f(g^xx) \leftrightsquigarrow^{\beta}_{\widehat{W}} f(g^yy)$. Hence, there exist  $h^x,h^y\in W$ so that  $h^xf(x),h^yf(y)\in O$ and  $h^xf(x)\leftrightsquigarrow^{\beta}_{\widehat{W}} h^yf(y)$ holds for all $\beta<\alpha$.   
\end{proof}
This finishes the proof of Theorem \ref{Th:SguigglePushesForward}.
\end{proof}

\section{Dynamical Back \& Forths}\label{S:II:BandF}

Let $G$ be a Polish group and $X$ a Polish $G$-space. 
Recall the following relations from \cite{Hjorth2000}, which are a generalization of the back-and-forth relations of model theory into the more general context of Polish group actions.
For $x, y \in X$ and open $U, V \subseteq G$, write $(x, U) \le_1 (y, V)$ iff $U \cdot x \subseteq \overline{V \cdot x}$.
For $\alpha \ge 1$, write $(x, U) \le_{\alpha+1} (y, V)$ iff for every open $U' \subseteq U$ there is some $V' \subseteq V$ such that $(y, V') \le_\alpha (x, U')$.
For limit $\gamma$, write $(x, U) \le_\gamma (y, V)$ iff for every $\alpha < \gamma$, $(x, U) \le_\alpha (y, V)$.

\begin{proposition}\label{prop:hjorth}
Let $X$ be a Polish $G$-space and let $x, y \in X$ and $\alpha > 0$. Then $(x, U) \le_\alpha (y, V)$ if and only if for every $\BPi^0_\alpha$ set $A$, $y \in A^{*V}$  implies $x \in A^{*U}$.
\end{proposition}

\begin{proof}
The forward direction is stated and proved in Hjorth's original preprint \cite{Hjorth2010}. 
The full statement of this appeared in \cite{Drucker_2021} but without proof.
A complete proof of the statement appears in \cite{Allison2020}.
\end{proof}

In \cite{AP2021}, we introduced a system of relations that are similar  to Hjorth's relations but slightly stronger.
These appear to be more appropriate for the following arguments.

\begin{definition}
Let  $X$ be a Polish $G$-space. For any $x, y \in X$ and  $V\subseteq_1 G$ we write:
\begin{enumerate}
\item  $x \preceq^0_V y$, if and only if $x \in \overline{Vy}$ holds;
\item  $x \sim^\alpha_V y$, for an ordinal  $\alpha$,   if and only if  both $x \preceq^\alpha_V y$ and $y \preceq^\alpha_V x$ hold;
\item   $x \preceq^\alpha_V y$,   for an ordinal  $\alpha > 0$,  if and only if for every  $W\subseteq_1 G$, there exists some $v \in V$ such that for every $\beta < \alpha$ we have that $vy\sim^\beta_W x$.
\end{enumerate}
\end{definition}

Below we collect a few basic properties that are straightforward to prove by induction.
\begin{proposition}\label{prop:sim_basic_facts}
Let $X$ be a Polish $G$-space and $x, y, z \in X$,  $\alpha \ge 0$, and let $V$ and $W$ be symmetric open neighborhoods of $1_G$. Then
\begin{enumerate}
    \item if $x \sim^\alpha_V y$ then $y \sim^\alpha_V x$;
    \item if $V \subseteq W$ and $x \sim^\alpha_V y$ then $x \sim^\alpha_W y$;
    \item for any $v \in V$, $vx \sim^\alpha_V x$;
    \item if $x \sim^\alpha_V y$ and $y \sim^\alpha_V z$ then $x \sim^\alpha_{V^2} z$;
    \item if $x \preceq^\alpha_V y$ and $g \in G$ then $g x \preceq^\alpha_{gVg^{-1}} g y$; and
    \item if $x \preceq^\alpha_V y$ with $\alpha \ge 1$ then for any open neighborhood $W$ of $1_G$ there is some $v_0 \in G_0 \cap V$ such that $v_0 y \sim_W^\beta x$ for every $\beta < \alpha$, where $G_0$ is some fixed countable dense subgroup of $G$.
\end{enumerate}
\end{proposition}

\begin{proof}
Statement (1) is immediate.
Statement (2) is proved by a  straightforward induction, as is statement (3).
Statements (4), (5), are \cite[Lemma 4.3]{AP2021}.

We prove Statement (6).
Let $\alpha \ge 1$ and assume $x \preceq^\alpha_V y$. Fix any open neighborhood $W$ of $1_G$.
Let $W_0$ be an open neighborhood of $1_G$ with $W_0^2 \subseteq W$.
Now fix some $v \in V$ such that $v \cdot y \sim^\beta_{W_0} y$ for every $\beta < \alpha$.
Let $v_0 \in W_0v \cap G_0$.
For any $\beta < \alpha$ we have by (3) that $v_0 \cdot y \sim^\beta_{W_0} v \cdot y$ and by (4) we have $v_0 \cdot y \sim^\beta_{W_0^2} x$ and thus by (2) we have $v_0 \cdot y \sim^\beta_W x$, as desired.
\end{proof}

Here we show that the relations $\sim^\alpha_G$ are just as strong as the Hjorth relations.

\begin{proposition}\label{prop:sim_to_hjorth}
Let $X$ be a Polish $G$-space and $x, y \in X$.Then for any $\alpha \ge 0$ if $x \preceq^\alpha_G y$ then $(x, G) \le_{\alpha+1} (y, G)$.
\end{proposition}

\begin{proof}
We prove the stronger claim that for any ordinal $\alpha$, open neighborhood $W_0$ of $1_G$, and $x, y \in X$ that $x \preceq^\alpha_{W_0} y$ implies $(x, W_0) \le_{\alpha+1} (y, W_0^2)$.


For $\alpha = 0$ this is immediate, so assume $\alpha > 0$ and assume the claim is true below $\alpha$.
Suppose $x \preceq^{\alpha}_{W_0} y$.
With the intention of showing $(x, W_0) \le_{\alpha+1} (y, W_0^2)$, we fix an arbitrary basic open $U \subseteq W_0$.
Choose any $u \in U$ and basic open neighborhood $U_0$ of $1_G$ such that $U_0^2u \subseteq U$.
We have $u \cdot x \preceq^{\alpha}_{uW_0u^{-1}} u \cdot y$ by Proposition \ref{prop:sim_basic_facts}.(5).
Find $w \in W_0$ such that 
\[ \forall \beta < \alpha, \; uw \cdot y \preceq^{\alpha}_{U_0} u \cdot x.\]
By the induction hypothesis we have
\[\forall \beta < \alpha, \; (uw \cdot y, U_0) \le_{\beta+1} (u \cdot x, U_0^2),\]
or equivalently,
\[ \forall \beta < \alpha, \; (y, U_0uw) \le_{\beta+1} (x, U_0^2u).\]
And because $U_0^2u \subseteq U$, we have
\[\forall \beta < \alpha, \;  (y, U_0uw) \le_{\beta+1} (x, U).\]
A case analysis on whether $\alpha$ is successor or limit applied to the recursive definition of $\le$ gives us
\[(y, U_0uw) \le_{\alpha} (x, U).\]

Finally, we observe that because $U_0uw \subseteq UW \subseteq W_0^2$ and $U$ was arbitrary we have proved
\[(x, W_0) \le_{\alpha+1} (y, W_0^2)\]
as desired.
\end{proof}

\section{A criterion for strong generic ergodicity}\label{S:II:Criterion}

The following is the critical fact for our generic ergodicity result.
This should be compared to the theorem in Hjorth's theory of turbulence that an equivalence relation being generically ergodic with respect to $=^+$ implies being generically ergodic with respect to any orbit equivalence relation induced by a continuous action of $S_\infty$ on any Polish space, see e.g., \cite[Corollary 10.3.7]{Gao2008}.

\begin{theorem}\label{Theorem:SimGenErg}
Let $E$ be an equivalence relation on a Polish space $X$ and let $Y$ be a Polish $G$-space.
Suppose $E$ is generically ergodic with respect to $\sim^1_G$ as computed in $(Y, \tau)$ for every compatible $G$-space Polish topology $\tau$ on $Y$.
Then $E$ is generically ergodic with respect to $E^G_Y$.
\end{theorem}

\begin{proof}
We first claim that it is enough to show that $E$ is generically ergodic with respect to $\sim^\alpha_G$ on $Y$ for every countable $\alpha$.
Indeed, we can find a countable ordinal $\alpha$ and a comeager set $C \subseteq X$ such that $G \cdot f(x)$ is $\BPi^0_\alpha$ for every $x \in C$, see \cite[Claim 5.4]{AP2021}.
In particular, by Propositions \ref{prop:hjorth} and \ref{prop:sim_to_hjorth}, there is a countable ordinal $\alpha$ such that for every $x,y \in C$, $f(x) \sim^\alpha_G f(y)$ if and only if $f(x) \mathrel{E^G_X} f(y)$.
Thus if comeager $D \subseteq X$ satisfies that $f(x) \sim^\alpha_G f(y)$ for every $x,y \in D$, then we have $f(x) \mathrel{E^G_Y} f(y)$ for every $x,y \in C \cap D$.

We proceed to show by induction on $\alpha$ that $E$ is generically ergodic with respect to $\sim^\alpha_G$ on $Y$ for every countable $\alpha$.
If $\alpha = 1$ this is immediately true, so assume $\alpha > 1$.
Let $f : X \rightarrow Y$ be a Baire-measurable homomorphism from $E$ to $E^G_X$.
Suppose for the induction hypothesis we have a comeager set $C \subseteq X$ such that for every $x, y \in C$ and for every $\beta < \alpha$, we have $f(x) \sim_\beta f(y)$.

Fix some $x_0 \in C$.
Let $G_0$ be a countable dense subgroup of $G$.
For any basic open $W_0 \subseteq G$, $g_0 \in G_0$, and $\beta < \alpha$, define
\[ A_{W_0, g_0, \beta} := \{y \in Y \mid \forall \gamma < \beta, \; y \sim^\gamma_{W_0} g_0 \cdot f(x_0)\}.\]

Let $\sigma$ be a compatible Polish topology on $Y$ such that $A_{W_0, g_0, \beta}^{\Delta W_1}$ is open for every basic open neighborhood $W_0$ of $1_G$, basic open $W_1 \subseteq G$, $g_0 \in G_0$, and $\beta < \alpha$.
Such topologies exist by a result of Hjorth, see \cite[Theorem 4.3.3]{Gao2008}.
Let $D \subseteq C$ be comeager such that $f(x) \sim^{1, \sigma}_G f(x')$ for every $x, x' \in D$, where $\sim^{1, \sigma}_G$ is the relation $\sim^1_G$ as computed in $(Y, \sigma)$.
Now fix any $x, x' \in D$ and our goal is to show that 
\begin{equation}\label{eq:claim0}
   f(x) \preceq^\alpha_G f(x'). 
\end{equation}

To that end, fix any basic open neighborhood $V$ of $1_G$.
Choose $V_0$ to be a basic open symmetric neighborhood of $1_G$ such that $V_0^2 \subseteq V$.
Applying the definition of $f(x) \sim^{1, \sigma}_G f(x')$, find some $g \in G$ such that
\begin{equation}\label{eq:part0}
    g \cdot f(x') \sim^{0, \sigma}_{V_0} f(x).
\end{equation}
Fix any $\beta < \alpha$.
We argue that 
\[g \cdot f(x') \sim^\beta_{V} f(x),\]
which would be enough to prove Equation (\ref{eq:claim0}).

To that end, fix any open neighborhood $W$ of $1_G$.
Choose $W_0$ to be a basic open neighborhood of $1_G$ such that $W_0^2 \subseteq W$.
Because $f(x) \sim^\beta_G f(x_0)$ holds, we can find some $g_0 \in G$ such that
\begin{equation}\label{eq:part1}
\forall \gamma < \beta, f(x) \sim^{\gamma}_{W_0} g_0 \cdot f(x_0).
\end{equation}
By Proposition \ref{prop:sim_basic_facts}.(6), we can choose $g_0$ to be in $G_0$.
In particular,
\[f(x) \in A_{W_0, g_0, \beta}^{\Delta V_0}.\]

Thus by Equation (\ref{eq:part0}), there is some $v_0 \in V_0$ such that
\[v_0g \cdot f(x') \in  A_{W_0, g_0, \beta}^{\Delta V_0}.\]

In particular, there is some $v_1 \in V_0$ such that
\begin{equation}\label{eq:part2}
\forall \gamma < \beta, \; v_1v_0g \cdot f(x') \sim^{\gamma}_{W_0} g_0 \cdot f(x_0).
\end{equation}
By Proposition \ref{prop:sim_basic_facts}.(4) and Equation (\ref{eq:part1}) and Equation (\ref{eq:part2}), we have
\[ \forall \gamma < \beta, \; v_1v_0g \cdot f(x') \sim^{\gamma}_{W_0^2} f(x).\]
By Proposition \ref{prop:sim_basic_facts}.(2), we have
\[ \forall \gamma < \beta, \; v_1v_0g \cdot f(x') \sim^\gamma_{W} f(x)\]
and then by observing $v_1v_0 \in V_0^2 \subseteq V$, we are done.
\end{proof}

\section{The proof of Theorem \ref{thm:main2}}\label{S:II:Proof1st}

Here we show how Theorem \ref{thm:main2} follows from  Theorem \ref{Th:SguigglePushesForward} and Theorem \ref{Theorem:SimGenErg}.  We start by proving a lemma that  will be needed in  putting everything together.

\begin{lemma}\label{L:FinalAris}
Let $G$ be a Polish group with  $\rk(G)\leq \alpha$ and   $G\curvearrowright X$ be a Polish $G$-space. For every $x,y\in X$, if  $x \leftrightsquigarrow^{\beta}_G y$ holds for all $\beta<\alpha$, then so does  $ x \sim^{1}_G y$.
\end{lemma}
\begin{proof} 
If $\rk(G)=1$,  then $G=\{1_G\}$ is the trivial group. In this case   $x \leftrightsquigarrow^{0}_G y$ trivially implies that $[x]_G=[y]_G$, and therefore    $x \sim^{1}_G y$ holds. So we may assume that $\alpha>1$.

Let $V\subseteq_1 G$. We will find some $h\in G$ so that  $x\sim^{0}_V hy$. A symmetric argument  will then establish $x \sim^{1}_G y$. Set  $\beta:=\rk(V,G)$.  Since $\rk(G)\leq \alpha$ we have that $\beta<\alpha$. Let now $W\subseteq_1 G$ so that for every $g\in G$ we have that  $\rk(V,gWg^{-1})<\beta$.  But by assumption we have that  $x \leftrightsquigarrow^{\beta}_G y $ holds. Hence there exist   $g^x,g^y\in G$ so that  for all $\gamma<\beta$ we have:
\[g^x x   \leftrightsquigarrow^{\gamma}_W  g^y y. \]
Set  $h_0:=(g^x)^{-1}g^y$ and $g:=(g^x)^{-1}$. By Lemma \ref{Aris:LSimpleSquiggleProperties}(4),  for all $\gamma<\beta$ we have that: 
\[  x   \leftrightsquigarrow^{\gamma}_{gWg^{-1}}  h_0 y.\]
Set $\gamma:=\rk(V,gWg^{-1})$ and notice that by the choice of $W$ we have that $\gamma<\beta$.
 Hence, by the next claim,  we get $u \in \big(gWg^{-1}\big)^3\subseteq G$ so that, setting $h:=u\cdot h_0$, we have that:
\[x\sim^{0}_V hy\]
\begin{claim}
Let $V,U\subseteq_1 G$  with   $\rk(V,U)\leq \gamma$ so that $U=U^{-1}$. Assume that  $a\leftrightsquigarrow^{\gamma}_U b$ holds for some  $a,b\in X$. Then we have that $a \sim^{0}_V u b$  holds for some  $u\in U^3$.
\end{claim}
\begin{proof}[Proof of Claim]
For $\gamma=0$,    $\rk(V,U)\leq 0$ implies that $U\subseteq V$. Hence, by  Lemma \ref{Aris:LSimpleSquiggleProperties}(3),  $a\leftrightsquigarrow^{0}_U b$ implies $a\leftrightsquigarrow^{0}_V u b$ for $u=1_G\in U^{3}$. Hence, we may assume that $\gamma>0$.

Since $\rk(V,U)\leq \gamma$, there exists some $W\subseteq_1 G$ so that for every $g\in U$ we have that  \[\rk(V,gWg^{-1})<\gamma.\]
By Proposition \ref{prop:basic_properties_of_rk}(1) we may assume  that $W^3\subseteq U$. Since $a\leftrightsquigarrow^{\gamma}_U b$ holds, we get   $g^a,g^b\in U$ so that for all $\delta<\gamma$ we have that  $g^a a   \leftrightsquigarrow^{\delta}_W  g^b b$. By Lemma \ref{Aris:LSimpleSquiggleProperties}(4), for all $\delta<\gamma$, we have 
\[a   \leftrightsquigarrow^{\delta}_{(g^a)^{-1}Wg^a}  (g^{a})^{-1}g^b b\]
But since $\rk(V,(g^a)^{-1}Wg^a)\leq \delta$ for some $\delta<\gamma$, by inductive hypothesis, and since 
\[\big((g^a)^{-1}Wg^a\big)^3=(g^a)^{-1}W^3g^a,\]
 we get some $\widetilde{w}\in W^3$ so that
\[a \sim^{0}_V   (g^{a})^{-1} \widetilde{w} g^{a}  (g^{a})^{-1}g^b  b.\]
But then, for $u:= (g^{a})^{-1} \widetilde{w} g^b \in U^{-1} W^3 U\subseteq U^3$  we have that
\[a \sim^{0}_V  u  b\]
as desired.
\end{proof}
The proof of the claim concludes the proof of Lemma \ref{L:FinalAris}.
\end{proof}

We can now put everything together and conclude with the proof of Theorem \ref{thm:main2}. 

\begin{proof}[Proof of Theorem \ref{thm:main2}]
Let $G\curvearrowright X$ and $H\curvearrowright Y$ be continuous actions of Polish groups on Polish spaces,  with $G\curvearrowright X$  generically $\alpha$-unbalanced and  $H$ being $\alpha$-balanced.   By Theorem \ref{Theorem:SimGenErg} it suffices to show that  $E^G_X$ is generically ergodic with respect to  $\sim^1_H$ on $Y$.

Indeed, let $f\colon X\to Y$ be a Baire-measurable $(E^G_X,E^H_Y)$-homomorphism. Let  $C\subseteq X$ be the comeager set provided by Theorem \ref{Th:SguigglePushesForward} and consequently let  $D\subseteq C$ be the comeager set provided by  Definition \ref{D:GenericUnbalancedness2}.  
\begin{claim}\label{Claim100}
For every $x,y\in D$ we have that   $f(x) \sim^{1}_H f(y)$.
\end{claim}
\begin{proof}[Proof of Claim \ref{Claim100}] Fix $x,y\in D$ and let  $x_0,\ldots, x_n\in C$   with  $x_0=x$,  $x_n =y$  and
\[x_0 \leftrightsquigarrow^{\beta}_G x_1 \leftrightsquigarrow^{\beta}_G \cdots \leftrightsquigarrow^{\beta}_G x_{n-1} \leftrightsquigarrow^{\beta}_G x_n, \]
for all $\beta<\alpha$.
 Having chosen  $C$  according to Theorem \ref{Th:SguigglePushesForward},  for all $\beta<\alpha$ we have that: 
\[f(x_0) \leftrightsquigarrow^{\beta}_H f(x_1) \leftrightsquigarrow^{\beta}_H \cdots \leftrightsquigarrow^{\beta}_H f(x_{n-1}) \leftrightsquigarrow^{\beta}_H f(x_n) \]
But since $\rk(H)\leq \alpha$  by Lemma \ref{L:FinalAris} we have that:
\[f(x_0) \sim^{1}_H f(x_1)  \sim^{1}_H  \cdots  \sim^{1}_H  f(x_{n-1})  \sim^{1}_H f(x_n) \]
By  Proposition \ref{prop:sim_basic_facts},  since $x=x_0, y=x_n$ we have that   $f(x) \sim^{1}_H f(y)$.
\end{proof}
The claim shows that $E^G_X$ is generically ergodic with respect to  $\sim^1_H$ and hence,  by Theorem \ref{Theorem:SimGenErg}, we have that  $E^G_X$ is generically ergodic against actions of $H$.
\end{proof}


\section{Examples of generically $\alpha$-unbalanced Bernoulli shifts}\label{S:II:Examples}

Let $N$ be a countably-infinite set and let $P$ be a closed subgroup of the group $\mathrm{Sym}(N)$, of  permutations of $N$, endowed with the pointwise convergence topology. The action $P\curvearrowright N$ induces continuous action  $P\curvearrowright \{0,1\}^N$ on the space of all maps $x\colon N\to\{0,1\}$ given by
\[(p,x)\mapsto p\cdot x \text{ where } (p\cdot x)(n)=x(p^{-1}(n))\]
We call  $P\curvearrowright \{0,1\}^N$  the {\bf Bernoulli shift of $P$}.  In the remainder of this paper we will be interested in the Bernoulli shifts of the automorphism groups  of the $\alpha$-scattered linear orders $\mathbb{Z}[\alpha]$ of countable ordinals $\alpha$, where $\alpha\mapsto \mathbb{Z}[\alpha]$ is defined in Section \ref{SS:Scattered_LOs}. More precisely, for each  ordinal $0<\alpha<\omega_1$, let $\mathbb{Z}[\alpha]$ be the collection of all maps $a\colon \alpha\to \mathbb{Z}$ with $a(\beta)=0$ for all but finitely many $\beta<\alpha$. We view $\mathbb{Z}[\alpha]$ as a linear ordering, by setting:
\[a<^* b \iff ``a(\beta)<b(\beta), \text{ for the largest } \beta<\alpha \text{ for which } a(\beta)\neq b(\beta) "\]

In the remainder of this paper we provide the proof of Theorem \ref{T:mainBernoulli}.
 Set $G_{\alpha}:= \mathrm{Aut}(\mathbb{Z}[\alpha])$ and $X_{\alpha}:=\{0,1\}^{\mathbb{Z}[\alpha]}$, so  that $G_{\alpha}\curvearrowright X_{\alpha}$ denotes the Bernoulli shift of $\mathrm{Aut}(\mathbb{Z}[\alpha])$. Recall  that from Lemma \ref{lem:scattered_cli} we have that $G_{\alpha}$ is $(\alpha+1)$-balanced. 
 We  prove:

\begin{theorem}\label{T:mainBernoulli_2}
 $G_{\alpha}\curvearrowright X_{\alpha}$ is generically $\alpha$-unbalanced and has meager orbits. 
\end{theorem}

One interesting aspect of Theorem \ref{T:mainBernoulli_2} is that its proof seems to require something more sophisticated than a straightforward induction. Before we proceed to the actual proof, we first discuss where the usual base/successor/limit--case induction falls short.

When it comes to the base--case $\alpha=1$, Theorem \ref{T:mainBernoulli_2} follows directly from the fact that the Bernoulli shift $\mathbb{Z}\curvearrowright \{0,1\}^{\mathbb{Z}}$ of $\mathbb{Z}$ is generically ergodic and has countable orbits. Successor stages $\alpha:=\beta+1$ can also be dealt with, just  by  appropriating  ideas and methods from \cite{AP2021}. Indeed, if $\alpha:=\beta+1$, then $G_{\alpha}\curvearrowright X_{\alpha}$ turns out to just be the ``$\mathbb{Z}$--jump"  \[\mathbb{Z} \; \mathrel{\mathrm{Wr}} \; G_{\beta}\curvearrowright (X_{\beta
})^{\mathbb{Z}}.\]
Assuming now  that the paths $x_0,\ldots,x_n$  witnessesing the  generic $\beta$-unbalancedness of   $G_{\beta}\curvearrowright X_{\beta}$ according to Definition \ref{D:GenericUnbalancedness2} are ``nice enough", a technical elaboration on the ideas from  \cite{AP2021} can be used to leverage these dynamics to generic $\alpha$-unbalancedness of  $G_{\alpha}\curvearrowright X_{\alpha}$. 

The difficulty lies in the limit stages. For example, already for finite  $\alpha=1,2,3,\ldots$, the length $n:=n(\alpha)$ of the paths $x_0,\ldots,x_n$ witnessing the $\alpha$-unbalancedness of  $G_{\alpha}\curvearrowright X_{\alpha}$ goes to infinity $n(\alpha)\to \infty$ as $\alpha\to\omega$.  As a consequence, there seems to be no simple way to combine these paths to some finite ``limiting" path witnessing the  generic $\omega$-unbalancedness of $G_{\omega}\curvearrowright X_{\omega}$. But even when the ``right" argument for $\alpha=\omega$ has been established, a similar ``phase transition" takes place at  the limit ordinal $\alpha=\omega^2$, requiring yet another new argument that was not needed for the earlier limit ordinals $\omega, 2\omega,3\omega,\ldots$.

As it turns out one needs to analyze separately the ``atomic" case $\alpha=\omega^{\lambda}$ for each $\lambda<\omega_1$ and then use the Cantor normal form in order to reduce the general $\alpha$  to a finite combination of atomic cases. In the remainder of this section we make this strategy precise. More specifically, we  start by introducing a jump operator  $\big(G\curvearrowright X\big)\mapsto J\big(\omega^{\lambda},G\curvearrowright X\big)$ for Polish  $G$-spaces and we establish that, in a certain technical sense, it amplifies generic unbalancedness  from $\alpha$ to $\alpha+\omega^{\lambda}$; see Lemma \ref{L:Ind}. We then  show how to reduce Theorem \ref{T:mainBernoulli_2} to   Lemma \ref{L:Ind}. The proof of Lemma \ref{L:Ind} is given in the  subsequent Sections \ref{S:lambda=0}, \ref{S:lambda>0}.

\subsection{The fusion lemma} For every Polish group $G$ and any Polish permutation group  $P\leq \mathrm{Sym}(L)$ on a countable set $L$, 
let $P \mathrel{\mathrm{Wr}} G:= P\ltimes\prod_{a\in L}  G$ be the Polish group of all pairs $(p,(g_a)_a)$ with $p\in P$ and $(g_a)_a\in \prod_{a\in L}  G$, endowed with the product topology and 
\[ (p,(g_a)_a)\cdot (q,(h_a)_a) =(pq, (g_{q(a)} h_a)_a)\]
For any Polish $G$-space $X$ we have a Polish  $P \mathrel{\mathrm{Wr}} G$-space $X^L$ given by:
\begin{equation}\label{EQ:Jump}
(p,(g_a)_a)\cdot (x_a)_a := (g_{p^{-1}(a)}\cdot x_{p^{-1}(a)} )_a
\end{equation}


\begin{definition}
For any countable ordinal $\mu>0$ and any Polish $G$-space $X$, 
the {\bf $\mu$-jump} 
\[J(\mu,G\curvearrowright X)\]
of $G\curvearrowright X$ is the Polish  $P \mathrel{\mathrm{Wr}} G$-space  $X^L$, where $L$ is the linear order $\mathbb{Z}[\mu]$ and $P=\mathrm{Aut}(L)$.
\end{definition}

Notice that $J(1,G\curvearrowright X)$ is just the  $\mathbb{Z}$--jump, as defined in \cite{AP2021}; see also \cite{ClemensCoskey}. In this case, $X^\mathbb{Z}$ comes together with a continuous, surjective, and  open map 
$\pi\colon X^{\mathbb{Z}}\times X^{\mathbb{Z}}\to X^{\mathbb{Z}}$ which combines any  pair $(x_a)_a,(y_a)_a\in X^{\mathbb{Z}}$ to the ``fused"  $(z_a)_a=\pi((x_a)_a,(y_a)_a)$ so that if $x_a$ and $y_a$ have dense orbits in $G\curvearrowright X$, then both 
\[(x_a)_a \leftrightsquigarrow^{1}_{\mathbb{Z}\,\mathrel{\mathrm{Wr}}\, G}  (z_a)_a \quad \text{ and } \quad   (z_a)_a \leftrightsquigarrow^{1}_{\mathbb{Z}\,\mathrel{\mathrm{Wr}}\, G}  (y_a)_a\]
hold in $\mathbb{Z}\mathrel{\mathrm{Wr}} G \curvearrowright X^{\mathbb{Z}}$; see proof of  \cite[Theorem 1.5]{AP2021}.
Central to the proof of Theorem \ref{T:mainBernoulli_2} is that the jumps $J(\mu,G\curvearrowright X)$ admit a \emph{fusion map} with similar properties, if $\mu$ is of the form $\omega^{\lambda}$ for some countable ordinal $\lambda$. This is made precise in the following lemma.

\begin{lemma}\label{L:Ind}
Let  $X$ be a Polish $G$-space and  $X^L$ be the $\mathrm{Aut}(L) \mathrm{Wr}\, G$-space $J(\mu,G\curvearrowright X)$ for some ordinal of the form $\mu=\omega^{\lambda}$.  Then, there exists a continuous, surjective, and open map $\pi\colon X^{L}\times X^{L}\to X^{L}$ so that for  any $x=(x_a)_a$, $y=(y_a)_a \in X^{L}$ and $\nu<\omega_1$, if 
\begin{enumerate}
\item  $x_a \leftrightsquigarrow^{\gamma}_G y_a$ holds  for all $a\in L$ and all $\gamma<\nu$; and if
\item the orbits of $x_a$ and of $y_a$ are dense in $X$ for all $a\in L$;
\end{enumerate}
then, for every $\beta<\mu $ we have that:
\[ x  \; \leftrightsquigarrow^{\nu+\beta}_{\mathrm{Aut}(L) \mathrel{\mathrm{Wr}} G}  \; \pi(x,y) \quad  \text{ and } \quad \pi(x,y)   \;  \leftrightsquigarrow^{\nu+\beta}_{\mathrm{Aut}(L) \mathrel{\mathrm{Wr}} G}  \; y.\]
\end{lemma}

The proof of Lemma \ref{L:Ind} will be given in Sections   \ref{S:lambda=0} and \ref{S:lambda>0}. Section \ref{S:lambda=0} deals with the case $\mu=1$, by elaborating on  a construction  covered in \cite{AP2021}.  Section   \ref{S:lambda>0} deals with the remaining cases: $\mu=\omega^{\lambda}$ with $\lambda>0$. The special  property of ordinals of the form  $\omega^{\lambda}$ that will be used is  additive idecomposability: if $\alpha,\beta<\omega^{\lambda}$, then $\alpha+\beta<\omega^{\lambda}$; see Lemma \ref{L:additivelyIndecompo}.
It is not very difficult to find a  uniform argument covering both cases. However, dealing separately with the $\lambda=0$ case  provides not only a nice warm-up, but it also allows us to argue the   $\lambda>0$ case more efficiently, as we can can specialize our notation to the case when $\omega^{\lambda}$ is a limit ordinal.
We may now    reduce Theorem \ref{T:mainBernoulli_2} to Lemma \ref{L:Ind}.

\subsection{Proof of  Theorem \ref{T:mainBernoulli_2} from Lemma \ref{L:Ind}}
Fix $\alpha$ with $0<\alpha<\omega_1$, and let
\begin{equation}\label{EQ:CantorNormalForm}
\alpha = \omega^{\lambda_m} + \omega^{\lambda_{m-1}}+\cdots  +\omega^{\lambda_1}, \quad \quad \lambda_m\geq \lambda_{m-1}\geq \ldots\geq \lambda_1, \quad \quad m\geq 1
\end{equation} 
be the  Cantor normal form  of $\alpha$. We record the length $m$ of the above expression by setting 
\[\mathrm{cnf}(\alpha):=m.\]

\begin{claim}
There is a sequence $\{\pi_i\colon 0\leq i \leq 2^m\}$ of  continuous, open,   surjective maps
\[\pi_i\colon X_{\alpha}\times X_{\alpha}\to X_{\alpha},\]
and a comeager $Y\subseteq X_{\alpha}$ so that for all $x,y\in Y$  and  every $i< 2^m$ we have that 
\[\pi_0(x,y)=x, \quad  \pi_{2^m}(x,y)=y, \; \text{ and } \; \; \pi_i(x,y)\leftrightsquigarrow^{\beta}_{G_\alpha} \pi_{i+1}(x,y) \text{ for all } \beta<\alpha.\]
\end{claim}
\begin{proof}[Proof of Claim]
We will run an induction on $m$.

\medskip{}

\noindent\underline{Assume that  $m=1$}. Then $\alpha=\omega^{\lambda}$ for some countable ordinal $\lambda$.  

If $\lambda=0$, then $\alpha=1$ and $G_1\curvearrowright X_1$ is  just the Bernoulli shift $\mathbb{Z}\curvearrowright \{0,1\}^{\mathbb{Z}}$ of the discrete group $\mathbb{Z}=G_1$. In this case, simply let $Y\subseteq X_1$  be the set of all  elements with dense orbit in $\{0,1\}^{\mathbb{Z}}$ and let $\pi_0,\pi_1,\pi_2$ simply be the maps $X_{\kappa}\times X_{\kappa}\to X_{\kappa}$ with $\pi_0(x,y)=x$, $\pi_1(x,y)=x$, $\pi_2(x,y)=y$. Since  all $x,y\in Y$ have dense orbits, we have $x\leftrightsquigarrow^{1}_{G_1} x \leftrightsquigarrow^{1}_{G_1} y$.

If    $\lambda>0$, then $\alpha=1+\alpha$, and hence   $G_{\alpha}\curvearrowright X_{\alpha}$ is isomorphic to 
$J(\omega^{\lambda},G_1\curvearrowright X_1)$. Let $Y_1\subseteq X_1$ be  set of all  elements with dense orbit in the $G_1\curvearrowright X_1$. Applying  Lemma \ref{L:Ind} to $J(\omega^{\lambda},G_1\curvearrowright X_1)$ we get a map $\pi\colon X^L\times
X^L\to X^L$ as in the conclusion of the lemma, where $L=\mathbb{Z}[\alpha]=\mathbb{Z}[(\omega^{\lambda})]$. Let  $Y:=Y_1^L$ and $\pi_0,\pi_1,\pi_2$ be the maps $X^{L}\times X^{L}\to X^{L}$ with $\pi_0(x,y)=x$, $\pi_1(x,y)=\pi(x,y)$, $\pi_2(x,y)=y$. Since $x=(x_a)_{a\in L}$,  $y=(y_a)_{a\in L}\in Y$, we have that  $x_a \leftrightsquigarrow^{0}_{G}y_a$ for all $a\in L$. Hence, by  Lemma \ref{L:Ind}   for all $\beta<\alpha$ we have that:
\[ \pi_0(x,y) \; \leftrightsquigarrow^{\beta}_{G_{\alpha}}  \; \pi_1(x,y) \quad  \text{ and } \quad \pi_1(x,y)   \;  \leftrightsquigarrow^{\beta}_{G_{\alpha}}  \; \pi_2(x,y).\]

\smallskip
\noindent\underline{Assume that  $m>1$}. Then,  we can rewrite $\alpha$ as $\nu + \omega^{\lambda}$ where $\nu:=\omega^{\lambda_m} +\cdots  +\omega^{\lambda_2}\geq 1$ and  $\omega^{\lambda}:=\omega^{\lambda_{1}}\geq 1$ as in the Cantor normal form (\ref{EQ:CantorNormalForm}) of $\alpha$.
Since  $\mathrm{cnf}(\nu)=m-1<m$, the inductive hypothesis applies and we get a comeager set $Y_{\nu}\subseteq X_{\nu}$ and a sequence $\{\rho_{\ell}\colon 0\leq \ell \leq 2^{m-1}\}$ of maps $X_{\nu}\times X_{\nu}\to X_{\nu}$ so that for  all  $x,y\in Y_{\nu}$ and  all $\gamma<\nu$ we have: 
\[\rho_0(x,y)=x, \quad  \rho_{2^{m-1}}(x,y)=y, \; \text{ and } \; \; \rho_{\ell}(x,y)\leftrightsquigarrow^{\gamma}_{G_\nu} \rho_{\ell+1}(x,y).\]
But since $\alpha=\nu+\omega^{\lambda}$, we have that 
  $G_{\alpha}\curvearrowright X_{\alpha}$ is isomorphic to $J(\omega^{\lambda}, G_{\nu}\curvearrowright X_{\nu})$.    Set   $L:=\mathbb{Z}[(\omega^{\lambda})]$ and let  $\pi\colon X_{\nu}^L\times
X_{\nu}^L\to X_{\nu}^L$  be the map provided by Lemma \ref{L:Ind}.

We define the desired sequence  $\{\pi_i\colon 0\leq i \leq 2^m\}$ of maps
$X_{\alpha}\times X_{\alpha}\to X_{\alpha}$ using the identification of $X_{\alpha}$ with $X_{\nu}^L$. If $i=2\ell$ for some $\ell\leq 2^{m-1}$ then set $\pi_i=\otimes_{a\in L} \,\rho_\ell$. That is,
\[\pi_i((x_a)_a,(y_a)_a):=(\rho_\ell(x_a,y_a))_a, \text{ when  }  i=2\ell \text{ for } \ell\leq 2^{m-1}.\]
Otherwise, we have that $i=2\ell+1$ for some $\ell< 2^{m-1}$. In which case we let $\pi_{i}$ to be the ``fusion of $\rho_{\ell}$ and $\rho_{\ell+1}$ via $\pi$". More precisely for all $(x_a)_a, (y_a)_a \in X_{\nu}^{L}$ we have that 
\[\pi_i((x_a)_a,(y_a)_a):=\pi\big((\rho_\ell(x_a,y_a))_a,(\rho_{\ell+1}(x_a,y_a))_a \big), \; \text{ if  }  i=2\ell+1 \text{ for } \ell< 2^{m-1}.\]
Set  finally $Y:=Y_{\nu}^{L}\subseteq  X_{\nu}^L = X_{\alpha}$. Clearly $Y$ is comeager in $X_{\alpha}$. Moreover, by the choice of  $Y_{\nu}$ and $\rho_{\ell},\rho_{\ell-1}$, for every $(x_a)_a, (y_a)_a\in Y$, every $\ell\leq 2^{m-1}$ and all $a\in L$ we  have that 
\[ \rho_\ell(x_a,y_a) \;    \leftrightsquigarrow^{\gamma}_{G_{\nu}}    \;\rho_{\ell+1}(x_a,y_a),  \; \text{ for all   }  \gamma<\nu.  \]
But then, by the conclusion  of  Lemma \ref{L:Ind}, for $i=2\ell+1$  
we have that:
\[\pi_i((x_a)_a,(y_a)_a)   \;  \leftrightsquigarrow^{\nu+\beta}_{G_{\nu+\omega^{\lambda}}}  \; \pi_{i+1}((x_a)_a,(y_a)_a) \; \leftrightsquigarrow^{\nu+\beta}_{G_{\nu+\omega^{\lambda}}}  \; \pi_{i+2}((x_a)_a,(y_a)_a) \]
for all $x=(x_a)_a, y=(y_a)_a \in Y$ and every $\beta<\omega^{\lambda}$. Hence by \ref{Aris:LSimpleSquiggleProperties}(2) we have that 
\[\pi_i(x,y)\leftrightsquigarrow^{\beta}_{G_\alpha} \pi_{i+1}(x,y)\]
for all $\beta<\alpha$, $i<2^m$    , and   $x=(x_a)_a, y=(y_a)_a \in Y$.
 This concludes the induction.
\end{proof}

\begin{proof}[Proof of  Theorem \ref{T:mainBernoulli_2} from Lemma \ref{L:Ind}]
Let now $\{\pi_i \colon i\leq 2^m\}$ and $Y\subseteq X_{\alpha}$ as in the above claim. We may conclude with the proof of Theorem \ref{T:mainBernoulli} as follows.

Let $C\subseteq X_{\alpha}$ be comeager and consider the subset $\widehat{D}$ of $X_{\alpha}\times X_{\alpha}$ given by
\[\widehat{D}:=\bigg(\bigcap_{i\leq  2^m} \pi^{-1}_i(C) \bigg) \cap \bigg(Y\times Y\bigg). \]
Since each $\pi_i$ is continuous open and surjective, there exists some $z_*\in C$ and some comeager  $D\subseteq C$ so that for all $z\in D$ we have that $(z_*,z)\in \widehat{D}$, i.e.,  $\pi_{i}(z,z_*)\in C$ for all $i\leq 2^m$. But then, for every $x,y\in D$, the concatenation of the paths 
\[x=\pi_0(x,z_*),  \ldots, \pi_{2^m}(x,z_*) = z_* \quad\text{ and } \quad  z_* = \pi_{2^m}(y,z_*),  \ldots,  \pi_0(y,z_*)=y \]
provides a finite path from $x$ to $y$  in $C$ which witnesses the generic $\alpha$-unbalancedness of the Bernoulli shift  $G_{\alpha}\curvearrowright X_{\alpha}$, according to Definition \ref{D:GenericUnbalancedness2}.
Finally, a simple induction on $\alpha$ establishes that the action $G_{\alpha}\curvearrowright X_{\alpha}$ has meager orbits for all $\alpha<\omega_1$. 
\end{proof}

\subsection{The product lemma} We close this section with the following general lemma that is going to be used in both Sections \ref{S:lambda=0} and \ref{S:lambda>0} for the proof of Lemma \ref{L:Ind}.

\begin{lemma}\label{Aris:LProd}
Let  $G_n\curvearrowright X_n$ be a Polish $G_n$-space, $x_n,y_n\in X_n$, and $V_n\subseteq_1 G_n$ for all  $n\in\mathbb{N}$. 
  Set $G:=\prod G_n$,  $V:=\prod V_n$, $X:=\prod X_n$, $x:=(x_n)_n$, $y:=(y_n)_n$. If  there is some  $n_0\in\mathbb{N}$ so that $V_{n}=G_{n}$ for $n<n_0$ and  $x_n \leftrightsquigarrow^\alpha_{V_n} y_n$ holds for all $n\in \mathbb{N}$, then so does 
\[ x \leftrightsquigarrow^\alpha_{V} y.\]
\end{lemma}

\begin{proof}
This is clearly true if $\alpha=0$, as in the product topology on $X$ we have that  ``$y_n \in \overline{V_n \cdot x_n}$ for all $n\in\mathbb{N}$" implies that ``$y\in \overline{V \cdot x }$", and similarly for $x$ in place of $y$. 

Assume now inductively that the lemma holds for all ordinals less than $\alpha>0$ and let  $U\subseteq X$,  $W\subseteq_1 G$ be open, say with $y\in U$ (the case $x\in U$ is similar). We may assume without loss of generality that there exists some $m_0\in\mathbb{N}$ and  open $W_n\subseteq_1 G_n$ and   $U_n\subseteq X_n$ for every $n\in \mathbb{N}$, with $W_n=G_n$  and $U_n=X_n$ for all $n<m_0$, so that  $W:=\prod W_n$ and $U:=\prod U_n$.
 Since $x_n \leftrightsquigarrow^\alpha_{V_n} y_n$,  we can find $g^x_n,g^y_n\in V_n$ so that $g^x_n\cdot x_n, g^y_n\cdot  y_n \in U_n$ and
 \[g^x_n\cdot x_n \leftrightsquigarrow^{\beta}_{W_n} g^y_n\cdot y_n\]
for all $\beta<\alpha$ and  $n\in\mathbb{N}$.
Set $g^x:=(g^x_n)_n$ and $g^y:=(g^y_n)_n$ and notice that
 $g^x x, g^y y \in U$. Moreover, by inductive assumption, for all $\beta<\alpha$ we have that 
\[g^x\cdot x \leftrightsquigarrow^{\beta}_W g^y\cdot y.\]
 \end{proof}

\section{Proof of Lemma \ref{L:Ind}, when $\lambda=0$}\label{S:lambda=0}

If $\lambda=0$ then $L$ is isomorphic to $(\mathbb{Z},\leq)$  and $\mathrm{Aut}(L)$ is isomorphic to the discrete group $\mathbb{Z}$. Hence, $J(1,G\curvearrowright X )$  is  induced simply by taking the wreath product with $\mathbb{Z}$:
\[\mathbb{Z} \; \mathrm{Wr} \; G \curvearrowright X^{\mathbb{Z}}\]
We will define a  continuous, surjective, and open map $\pi\colon X^{\mathbb{Z}}\times X^{\mathbb{Z}}\to X^{\mathbb{Z}}$ satisfying the conclusion of   Lemma \ref{L:Ind}.  First, let $p\colon \mathbb{Z}\to  \{0,1\}$ be any map with the property:

\begin{enumerate}
\item[($\star$)]  if $B\subseteq \mathbb{Z}$ is finite  and  $j\in\{0,1\}$, there is  $\ell\in\mathbb{Z}$ with  $p(k-\ell)=j$ for all $k\in B$.
\end{enumerate}

For example, the generic map in $\{0,1\}^{\mathbb{Z}}$ has this property. Alternatively one can simply take the map sending all negative
$k\in\mathbb{Z}$ to $0$ and strictly positive $k\in\mathbb{Z}$ to $1$. 
Having fixed any  $p\colon \mathbb{Z}\to  \times\{0,1\}$ which satisfies  $(\star)$ as above, we consider the induced map 
\[\pi\colon X^{\mathbb{Z}}\times X^{\mathbb{Z}}\to X^{\mathbb{Z}} \quad  \text{ with } \quad \pi(x,y)=z,  \text{ where } z_k=\begin{cases}  
x_k & \text{ if } p(k)=0\\
y_k & \text{ if } p(k)=1\\
\end{cases} \]

\begin{claim}\label{Aris:Ind0Claim1}
 The map $\pi\colon X^{\mathbb{Z}}\times X^{\mathbb{Z}}\to X^{\mathbb{Z}}$ is a continuous, open, and surjective.
\end{claim}
\begin{proof}[Proof of Claim.]
The map $\pi$ is the composition of the projection  $X^{\mathbb{Z}}\times X^{\mathbb{Z}}\to X^{M}\times X^N$ where $M=p^{-1}(0)$ and $N=p^{-1}(1)$ with the homemorphism $X^{M}\times X^N\to X^{\mathbb{Z}}$ induced by $M\sqcup N= \mathbb{Z}$. These two maps are continuous, open, and surjective.
\end{proof}

\begin{claim}\label{Aris:Ind0Claim2}
Let  $x=(x_k)_k$, $y=(y_k)_k$ in $X^{\mathbb{Z}}$ and  a countable ordinal $\nu$ so that for all  $k\in \mathbb{Z}$ the $G$--orbits of $x_k$ and $y_k$ are dense in $X$ and  $x_k \leftrightsquigarrow^{\gamma}_G y_k$ holds for all $\gamma<\nu$. Then,
\[ x  \; \leftrightsquigarrow^{\beta}_{\mathbb{Z} \, \mathrm{Wr} \, G}  \; \pi(x,y) \quad  \text{ and } \quad \pi(x,y)   \;  \leftrightsquigarrow^{\beta}_{\mathbb{Z} \, \mathrm{Wr} \, G}   \; y  \; \text{ hold for all } \beta<\nu+1.\]
\end{claim}
\begin{proof}[Proof of Claim.]
Set $z:=\pi(x,y)$. By Lemma \ref{Aris:LSimpleSquiggleProperties}(2) it suffices to show that  $x   \leftrightsquigarrow^{\nu}_{\mathbb{Z} \, \mathrm{Wr} \, G} z$ and
$z   \leftrightsquigarrow^{\nu}_{\mathbb{Z} \, \mathrm{Wr} \, G} y$ hold. Since the argument is symmetric we only establish  the relation:
 \[x   \leftrightsquigarrow^{\nu}_{\mathbb{Z} \, \mathrm{Wr} \, G} z.\]

 Let $U\subseteq X^{\mathbb{Z}}$ be open with $z\in U$ (the case $x\in U$ is similar) and let $V\subseteq_1 \mathbb{Z}\, \mathrm{Wr}\, G$. After shrinking $U,V$ if necessary we may assume that there exists some finite set $B\subseteq \mathbb{Z}$ and, for all $k\in B$,   non-empty open  sets $U_{k}\subseteq X$ and
$V_{k}\subseteq_1 G$,  so that
\begin{eqnarray*}
U=\{x \in X^{\mathbb{Z}}\colon x_k \in U_k \text{ for all } k\in B \},&\\
V=\{h \in G^{\mathbb{Z}} \colon h_k \in V_k \text{ for all } k\in B\},&
\end{eqnarray*}
where  $G^{\mathbb{Z}}$ is identified with its natural copy as a clopen subgroup of $\mathbb{Z}\, \mathrm{Wr}\, G$.

By property $(\star)$ of the map $p$ in the definition of $\pi$, we may choose some $g\in \mathbb{Z}\, \mathrm{Wr}\, G$ which implements an  outer $\mathbb{Z}$-shift  $(gx)_k=x_{k-\ell}$ by an appropriate integer $\ell$ so that for all $k\in B$ we have that $(gx)_k=(gz)_k$.
By the density of the orbits of $x_n$ and $y_n$ in $X$ we can further find $h_k\in G$ for every $k\in B$ so that $h_k\cdot (gx)_k \in U_k$, and hence $h_k\cdot (gz)_k \in U_k$, for all $k\in B$. Let $h:=\oplus_{k\in B} h_k$ be the associated element of $G^{\mathbb{Z}}\leq \mathbb{Z}\, \mathrm{Wr} \, G$. Set now $g^x=g^z=h g$. It is immediate that both  $g^x x\in U$ and $g^z z\in U$  hold. But by hypothesis
\[g^x x \leftrightsquigarrow_{V}^{\gamma}  g^z z\]
holds for all $\gamma<\nu$, as desired. Indeed, the latter follows from Lemma \ref{Aris:LProd} since 
\[(g^x x)_k \leftrightsquigarrow_{V_k}^{\gamma}  (g^z z)_k\]
holds for all $\gamma<\nu$ and all
 $k\in\mathbb{Z}$. To see this, notice that: when $k\in B$,  we  have the even stronger property $(g^x x)_k=(g^z z)_k$; and when $k\not\in B$, then $V_k=G$, 
 $(g^x x)_k=x_{k-\ell}$ and $(g^z z)_k=z_{k-\ell}$. But since  $z_{k-\ell}$ is either equal to $x_{k-\ell}$ or to $y_{k-\ell}$, by the hypothesis in the statement of the claim  we have that  $x_{k-\ell} \leftrightsquigarrow^{\gamma}_G y_{k-\ell}$ for all $\gamma<\nu$.
\end{proof}

\section{Proof of Lemma \ref{L:Ind}, when $\lambda>0$} \label{S:lambda>0}
Throughout this section we fix a countable ordinal $\lambda\neq 0 $ and  set  $\mu:=\omega^{\lambda}$ and $L:=\mathbb{Z}[\mu]$. 
We also fix some Polish $G$-space $X$ and consider the jump $J(\mu, G\curvearrowright X)$ of $G\curvearrowright X$:
\begin{equation}
\mathrm{Aut}(L)\mathrm{Wr}\, G \curvearrowright X^L
\end{equation}
We will define a ``fusion" map $\pi \colon X^L\times X^L \to X^L$  satisfying the conclusion of Lemma \ref{L:Ind}.

 The fact that $\mu$ is of the form $\omega^{\lambda}$ implies that  $\mu$  {\bf additively indecomposable}:

\begin{lemma}\label{L:additivelyIndecompo}
If $\mu=\omega^{\lambda}$ for some ordinal $\lambda$ and $\alpha,\beta<\mu$, then  $\alpha+\beta<\mu$.
\end{lemma}
\begin{proof}
For $\lambda=0$ we have $\alpha+\beta=0+0=0<1=\mu$. When $\lambda=\nu+1$ then there exist $k,\ell<\omega$ with $\alpha<\omega^{\nu}\cdot k$ and  $\beta<\omega^{\nu}\cdot \ell$. Hence   $\alpha+\beta< \omega^{\nu}\cdot (k + \ell)<\omega^{\nu}\cdot \omega$. Finally, if $\lambda=\sup_{\xi<\lambda} \xi$ then pick $\xi<\lambda$ with $\alpha,\beta<\omega^{\xi}$ and note $\alpha+\beta<\omega^{\xi}\cdot 2< \omega^{\xi}\cdot \omega=\omega^{\xi+1}<\omega^{\lambda}$.
\end{proof}

Additive indecomposability will be at the heart of  the construction of the  map $\pi$. Indeed the ``saturation"   of $\mu$, derived  by its indecomposability
\[\forall \alpha,\beta<\mu \; \;\exists \gamma<\mu  \;\;(\alpha+\beta<\gamma),\]
will equip the ``generic choice" of a fusion map $\pi$ with the desired properties.

There is just one problem. In trying to saturate  $\pi$  with the desired structure, 
the use of ordinal addition falls short in that it does not preserve order from the right. Indeed, $\beta<\alpha$ does not imply $\beta+\gamma<\alpha+\gamma$ (take for example $\beta=0$, $\alpha=1$, $\gamma=\omega$). For this reason,  we will occasionally need to  make use of what is known as ``natural" or ``Hessenberg" addition.

\smallskip{}

\subsection{Hessenberg addition of ordinals}

Let $\alpha$, $\beta$ be two ordinals and rewrite them as
\begin{equation}\label{EQ:CNFforSUM}
\alpha=\sum_{1\leq i \leq n} \omega^{\xi_i}k_i \quad \text{ and } \quad \beta=\sum_{1\leq i \leq n} \omega^{\xi_i}\ell_i
\end{equation}
by first expressing each of them it its Cantor normal form (\ref{EQ:CantorNormalForm});  then combining all elements of the same power $\omega^{\xi}$ to get  summands as above so that $\xi_n>\xi_{n-1}>\cdots>\xi_1$ and $k_i,\ell_i\geq 0$. By letting $k_i,\ell_i$ take the value $0$ we have expressed both $\alpha$ and $\beta$ as ``polynomials" in $\omega$, where every power $\omega^\xi$ of $\omega$ which shows up in the  above expression  of $\alpha$ shows up also in the above expression of $\beta$, and vice versa. The {\bf Hessenberg} or {\bf natural} addition $\alpha\oplus \beta$ of $\alpha$ and $\beta$ is the ordinal attained if we add the expressions (\ref{EQ:CNFforSUM})  as polynomials of $\omega$:
\[\alpha\oplus \beta := \sum_{1\leq i \leq n} \omega^{\xi_i}(k_i+\ell_i) \]
An important feature of $\oplus$ is that it is commutative, associative, and ``order preserving":
\begin{lemma}\label{L:PropertiesOfHessenberg}
For every $\alpha,\beta,\gamma$ the following properties hold:
\begin{enumerate}
\item $\alpha\oplus \beta = \beta\oplus \alpha$;
\item $(\alpha\oplus \beta)\oplus \gamma =\alpha\oplus (\beta\oplus \gamma)$;
\item $\beta<\alpha$ if and only if $\gamma\oplus \beta < \gamma \oplus \alpha$
\end{enumerate}
\end{lemma}

It will be important that $\mu$ above is $\oplus$--additively indecomposable.

\begin{lemma}
If $\alpha,\beta<\omega^\lambda$ for some $\lambda$, then $\alpha\oplus \beta \in \omega^{\lambda}$. 
\end{lemma}
\begin{proof}
This follows directly from Lemma \ref{L:additivelyIndecompo} and the definition of $\oplus$.
\end{proof}

\smallskip{}

\subsection{Generic antichain pairs}
We fix the following  notation. Elements  $a,b,c,\ldots$ of $L$ are functions $a\colon \mu\to \mathbb{Z}$ with $a(\xi)=0$ for all but finitely many $\xi<\mu$. We view $L$ as a linear order with $a< b$ holds iff  $a(\xi)<b(\xi)$ holds in $\mathbb{Z}$ for the  largest $\xi\in\mu$ for which $a(\xi)\neq b(\xi)$ holds.

Let $T$ be the collection of all possible restrictions $s:=a\upharpoonright [\beta,\mu)$ of elements $a\in L$ to some final interval $[\beta,\mu):=\{\xi\in \mathrm{ORD}   \colon \beta \leq  \xi<  \mu\}$ of $\mu=[0,\mu)$ with $\beta<\mu$. Equivalently, 
\[T=\bigcup_{\beta< \mu} \mathbb{Z}\big[[\beta,\mu)\big],\]
is the union of all sets of the form  $\mathbb{Z}\big[[\beta,\mu)\big]$ with $\beta< \mu$, where for every  $\Xi\subseteq [0,\mu)$ we let $\mathbb{Z}[\Xi]$   be   the set of all maps $s\colon \Xi\to \mathbb{Z}$ so that  $s(\xi)=0$ for all but finitely many $\xi \in \Xi$.

 If $s\in \mathbb{Z}\big[[\beta,\mu)\big]$, then   we say that the  {\bf height} of  $s$ is  $\beta$ and write $\mathrm{ht}(s)=\beta$ . Given $r\in  \mathbb{Z}\big[[\alpha,\beta)\big]$ and $s\in \mathbb{Z}\big[[\beta,\mu)\big]$, we write $r^{\frown}s$ for the unique element of $\mathbb{Z}\big[[\alpha,\mu)\big]$ with
 $r^{\frown}s \upharpoonright [\alpha,\beta)=r$ and  $r^{\frown}s \upharpoonright [\beta,\mu)=s$.
We view $T$ as a partial order under the relation $\sqsubseteq$, where
 $s\sqsubseteq t$  holds for $s,t\in T$ if and only if $\mathrm{ht}(t)\leq  \mathrm{ht}(s)$ and there is  $r\in \mathbb{Z}\big[[\mathrm{ht}(t),\mathrm{ht}(s))\big]$ so that $t=r^{\frown}s$. In this case, we say that $s$ is an {\bf initial segment} of $t$ or that $t$ {\bf extends} $s$.  
Notice that  $L$ is a subset of $T$,  consisting of all the $\sqsubseteq$--maximal elements of $T$. 

When $\mu=\omega$ then $T$ is a tree (with no root).  More generally, for any $\mu$ of the form  $\omega^{\lambda}$ with $\lambda>0$,  $T$ can be thought of as a ``piecewise tree". Every pair  $s,t\in T$ admits a  {\bf meet} $s\wedge t$, which is the $r\in T$ of least $\mathrm{ht}(r)$ with $r\sqsubseteq s$ and $r \sqsubseteq t$. We say that $s,t\in T$ are {\bf comparable} if  $s\wedge t\in \{s,t\}$ and we write  $s\perp t$ if they are not comparable.   
A subset $J\subseteq T$ of $T$ is an {\bf antichain} if $s\perp t$ for every $s,t\in J$.  An antichain $J$  is  {\bf maximal} if for all $t\in T$ there is  $s\in J$ so that $s$ and $t$ are compatible. 
For every  $J\subseteq  T$ consider the sets 
\begin{eqnarray*}
L(J):=\{a\in L \colon s\sqsubseteq a \text{ for some } s \in J\}\subseteq L \\
T(J):=\{t\in T \colon s\sqsubseteq t \text{ for some } s \in J\} \subseteq T
\end{eqnarray*}
If $a\in L(J)$ then we say that  $a$ is {\bf covered} by $J$. If $J=\{s\}$ then we simply write $L(s)$ and  $T(s)$ in place of $L(\{s\})$ and  $T(\{s\})$ above, and we  say that $a$ is {\bf covered} by $s$.

Notice that if $s\in T$ with $\mathrm{ht}(s)=\alpha>0$, then $L(s)$ is a suborder of $L$ isomorphic to $\mathbb{Z}[\alpha]$, as it is the image of $\mathbb{Z}[\alpha]:=\mathbb{Z}\big[[0,\alpha)\big]$ under the embedding $i\colon \mathbb{Z}[\alpha]\to \mathbb{Z}[\mu]$ with $i(a)=a^{\frown}s$, for all $a\in \mathbb{Z}[\alpha]$.
 Notice, moreover, that every $\varphi\in\mathrm{Aut}(L(s))$ extends to the automorphism  $\ltimes_{s,\varphi}\in\mathrm{Aut}(L)$, where: $\ltimes_{s,\varphi}(a)=\varphi(a)$, if $a\in L(s)$; and  $\ltimes_{s,\varphi}(a)=a$, if $a\in L \,\setminus \, L(s)$. 
 
  More generally, let $J\subseteq T$ be an antichain and let $\varphi_s\in\mathrm{Aut}(L(s))$ for every $s\in J$. Then the system $(\varphi_s\colon s\in J)$ of  ``local" automorphisms   induces a ``global" automorphism $\varphi\in \mathrm{Aut}(L)$, where $\varphi(a)=\varphi_s(a)$, if is covered by some $s\in J$; and $\varphi(a)=a$, otherwise.  We  say that $\varphi$ is the automorphism of $L$ {\bf induced by the system}  $(\varphi_s\colon s\in J)$.

 \begin{definition}
 An  {\bf antichain pair} is any pair  $(P_0,P_1)$, with $P_0,P_1\subseteq T$ so that:
\begin{enumerate}
\item[(P1)]  $P_0\cup P_1$ is a maximal antichain of $T$ with $P_0\cap P_1=\emptyset$;
\item[(P2)] if $s\in P_0$ and $t\in P_1$ then $\mathrm{ht}(s)\oplus\mathrm{ht}(t)<\mathrm{ht}(s\wedge t)$.
\end{enumerate}
We denote by   $\mathcal{P}$ the collection of  all antichain pairs.
 \end{definition}

Viewed as a closed subset of the Cantor space $2^T\times 2^T$,  $\mathcal{P}$ is a Polish space.   Moreover, since any partition of $L\subseteq T$ in $P_0,P_1\subseteq L$ satisfies (P1) and (P2), we have that $\mathcal{P}\neq\emptyset$.

We will define the pertinent fusion $\pi$ in Lemma \ref{L:Ind} by means of the generic  element of  $\mathcal{P}$. Hence, in establishing the various properties of $\pi$, it will be convenient to  rely on a handy basis for the topology of $\mathcal{P}$. For any  pair $(F_0,F_1)$ of finite subsets of $T$  let
\[\mathcal{P}(F_0, F_1):=\{(P_0,P_1)\in\mathcal{P} \colon F_0\subseteq P_0 \text{ and } F_1\subseteq P_1\}\]
Let also $\mathcal{F}$ be the collection of all pairs  $(F_0,F_1)$  of finite subsets of $T$ so that moreover: $F_0\cup F_1$ is an antichain with $F_0\cap F_1=\emptyset$ and $(F_0,F_1)$  satisfies (P2), in place of $(P_0,P_1)$.   

\begin{lemma}\label{Lemma:F_implies_not_empty}
The collection of all sets of the form $\mathcal{P}(F_0,F_1)$ forms a basis for the topology on $\mathcal{P}$. Moreover, for all finite $F_0,F_1\subseteq T$ we have that  $\mathcal{P}(F_0, F_1)\neq\emptyset  \iff (F_0, F_1)\in\mathcal{F}$.
\end{lemma}
\begin{proof}
For the first part notice that if $(P_0,P_1)\in\mathcal{P}$ satisfies $s\not\in P_0$, then by (P1) we can pick some $t\in P_0\cup P_1$ with $L(s)\cap L(t)\neq\emptyset$. Hence
any  condition of the form $s\not\in P_0$ for pairs $(P_0,P_1)$ in $\mathcal{P}$ can be written as a union of positive conditions $t\in P_0$ or $t\in P_1$.

For the second part, if $(F_0, F_1)\in\mathcal{F}$, then let $(Q_0,Q_1)$ be any partition of $L \; \setminus \; L( F_0\cup F_1)$ and notice that $(P_0,P_1):=(F_0\cup Q_0, F_1\cup Q_1)$ satisfies both (P1) and (P2) above. Hence, we have that $(P_0,P_1)\in \mathcal{P}(F_0, F_1)$. The other direction is straightforward.
\end{proof}

It will be important for the definition of the fusion map $\pi$ to find  $(P_0,P_1)\in\mathcal{P}$  so that both $P_0$ and $P_1$ satisfy the following property ($\star_{\alpha}$), 
for all $\alpha<\mu$. Recall that $\mathrm{Aut}(L)_A$ denotes the pointwise stabilizer $\{g\in \mathrm{Aut}(L)\colon g(a)=a \text{ for all } a \in A \}$ of $A$ in $\mathrm{Aut}(L)$.

\begin{definition}\label{Def_Star_alpha}
Let $J\subseteq T$ be an antichain and $\alpha<\mu$.  We say that $J$ satisfies ($\star_{\alpha}$) if:
\begin{enumerate}
\item[($\star_{\alpha}$)]  for all finite  $A,B\subseteq L$ and every  $\beta< \alpha$, if   each  $a\in A$  is covered by some $s_a\in J$ with $\mathrm{ht}(s_a)\geq \alpha$, then   there exists some $g\in  \mathrm{Aut}(L)_A$ so that for every $b\in B$ we have that $g(b)$ is covered by some $t_b\in J$ with $\mathrm{ht}(t_b)\geq \beta$.
\end{enumerate}
\end{definition}

\begin{remark}  Property ($\star$) from Section \ref{S:lambda=0} is essentially  the requirement that both $\pi^{-1}(0)$ and $\pi^{-1}(1)$ therein satisfy  ($\star_{\alpha}$) for $\alpha=1$. Indeed, in the context of Section \ref{S:lambda=0}, where $\mu=1$, there is no $s\in \mathbb{Z}\big[[0,\mu)\big]$ with $\mathrm{ht}(s)\geq 1$. As a consequence, in order to check  ($\star_{1}$) in the context of  Section \ref{S:lambda=0}, it would suffice to take $A=\emptyset$ in  Definition \ref{Def_Star_alpha}.
\end{remark}

Next we find some  $(P_0,P_1)\in\mathcal{P}$ so that both $P_0$ and $P_1$ satisfy  property ($\star_{\alpha}$):

\begin{lemma}\label{L:GenericPairHasStar}
Let $\mathcal{R}$ be the collection of all pairs  $(P_0,P_1)$ in $\mathcal{P}$ so that,  for every $i\in \{0,1\}$,
$P_i$ satisfies  ($\star_{\alpha}$) for  all $\alpha$ with   $\alpha<\mu$. Then $\mathcal{R}$
 forms a comeager subset of $\mathcal{P}$.
\end{lemma}
\begin{proof} Fix some $\alpha$ with $0<\alpha<\mu$ and let $\mathcal{R}_{\alpha}$ be the set   of all  of all  $(P_0,P_1)\in \mathcal{P}$ so that $P_i$ satisfies  ($\star_{\alpha}$) for both $i\in \{0,1\}$. We will show that  $\mathcal{R}_{\alpha}$ is comeager in $\mathcal{P}$.

Let $(P_0,P_1)\in \mathcal{P}$ be chosen generically. Let $i\in\{0,1\}$ and assume that we are trying to confirm that   ($\star_{\alpha}$) holds for some fixed  $A,B,\beta$. We will do that by partitioning $B$ into subsets $C\subseteq B$ with the property that each such $C$ is covered by some $r\in T$, of maximum possible $\mathrm{ht}(r)$, so that $L(r)$ does not intersect $A$. By the maximality of $\mathrm{ht}(r)$, there will be enough room for the genericity of the choice of $(P_0,P_1)$ to kick in and give us a ``local" automorphism $g_r\in \mathrm{Aut}(L(r))$ which satisfies the conclusion of  ($\star_{\alpha}$) only for the $C$ piece. Then we will glue all these ``local" automorphisms to the desired global one.

Fix some $i\in\{0,1\}$,  ordinals $\beta,\rho$ with $\beta<\alpha<\rho<\mu$, some $s,r\in T$ with  $\mathrm{ht}(r)=\rho$, $\mathrm{ht}(s)\geq\alpha$, $\mathrm{ht}(r\wedge s)=\rho+1$ and any finite  $C\subseteq L(r)$. 
Consider the set $\mathcal{R}_{\alpha}(s,r,i,\beta,C)$ which consists of all $(P_0,P_1)$ in $\mathcal{P}$ which satisfy the following property:
\begin{center}
if $s\in P_i$, then there is $g\in \mathrm{Aut}(L(r))$  so that  for every $c\in C$, \\ $g(c)$ is covered by some $t\in P_i$ with $\mathrm{ht}(t)\geq\beta$.
\end{center}

\begin{figure}[ht!]
\definecolor{zzttqq}{rgb}{0.6,0.2,0}
\definecolor{sqsqsq}{rgb}{0.12549019607843137,0.12549019607843137,0.12549019607843137}
\begin{tikzpicture}[line cap=round,line join=round,x=1cm,y=1cm,scale=0.5]

\clip(-10,-10) rectangle (10,1.5);

\fill[line width=0.2pt,color=zzttqq,fill=zzttqq,fill opacity=0.1] (6,-2) -- (2,-8) -- (10,-8) -- cycle;
\draw [line width=0.2pt,color=zzttqq] (6,-2)-- (2,-8);
\draw [line width=0.2pt,color=zzttqq] (2,-8)-- (10,-8);
\draw [line width=0.2pt,color=zzttqq] (10,-8)-- (6,-2);
\draw [line width=0.4pt] (0,0)-- (6,-2);
\draw [line width=0.4pt] (0,0)-- (-6,-2);
\draw [line width=0.2pt,dash pattern=on 1pt off 1pt] (-8,-2)-- (10,-2);
\draw (-9,-2) node {$\rho$};
\draw [line width=0.2pt,dash pattern=on 1pt off 1pt] (10,-4.75)-- (-8,-4.75);
\draw (-9,-4.75) node {$\alpha$};
\draw [line width=0.2pt,dash pattern=on 1pt off 1pt] (-8,-6.5)-- (10,-6.5);
\draw (-9,-6.5) node {$\beta$};
\draw [line width=0.4pt] (-6,-2)-- (-4.003577841967387,-3.1429967732102373);
\draw [line width=0.4pt] (-4.003577841967387,-3.1429967732102373)-- (-6,-4);
\draw [line width=0.4pt,dotted] (-6,-4)-- (-4.002122756704283,-7.147303707909575);
\draw [line width=0.4pt] (-4.002122756704283,-7.147303707909575)-- (-6,-8);
\draw [rotate around={0:(4,-8)},line width=0.4pt,fill=black,fill opacity=0.45] (4,-8) ellipse (1.25cm and 0.265cm);
\draw (4,-7.2) node {$C$};
\draw [line width=0.6pt,dotted] (6,-2)-- (5.7,-8);
\draw [line width=0.6pt,dotted] (6,-2)-- (8.7,-8);
\draw [line width=1.5pt] (5.8,-6)  .. controls (6.5,-5)  and     (6.8,-7.4)     .. (7.8,-6);
\draw (6.8,-5.5) node {$P_i$};
\draw [->,double]  (4.5,-7.5)  to [bend left] (7.2,-7.7)  ;
\draw [fill=sqsqsq] (0,0) circle (2.5pt);
\draw[color=sqsqsq] (0,0.5) node {$r\wedge s$};
\draw [fill=sqsqsq] (6,-2) circle (2.5pt);
\draw[color=sqsqsq] (6,-1.5) node {$r$};
\draw [fill=sqsqsq] (-6,-4) circle (2.5pt);
\draw[color=sqsqsq] (-6,-3.5) node {$s$};
\end{tikzpicture}
\caption{The data needed for defining  $\mathcal{R}_{\alpha}(s,r,i,\beta,C)$}
\end{figure}
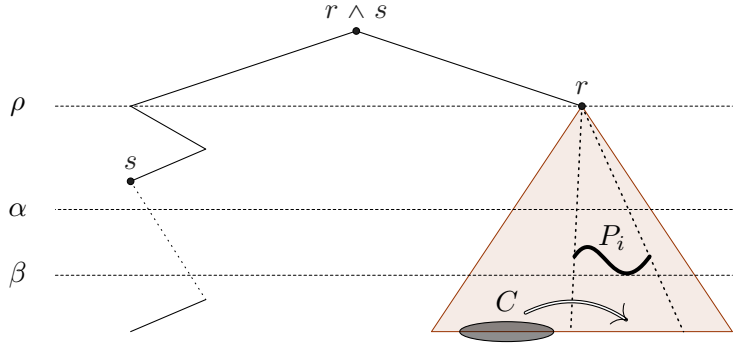

\begin{claim} 
For all $s,r,i,\beta,C$ as above, the set $\mathcal{R}_{\alpha}(s,r,i,\beta,C)$ is comeager in $\mathcal{P}$.
\end{claim}
 \begin{proof}[Proof of Claim]
 Clearly  $\mathcal{R}_{\alpha}(s,r,i,\beta,C)$ is open, so its suffices to show that it intersects any basic open set $\mathcal{P}(F_0,F_1)$, with  $(F_0,F_1)\in\mathcal{F}$; see Lemma \ref{Lemma:F_implies_not_empty}. 
 We may assume without loss of generality that $s\in F_i$, as otherwise we can find some $(P_0,P_1)\in \mathcal{P}(F_0,F_1)$ with $s\not\in P_i$, which would imply that  $(P_0,P_1)\in\mathcal{R}_{\alpha}(s,r,i,\beta,C)$. Moreover, notice that  
if $r\in F_i$, then $\mathcal{P}(F_0,F_1)\subseteq \mathcal{R}_{\alpha}(s,r,i,\beta,C)$. Hence we may also assume that $r\not\in F_i$.

Since $F_0$ and $F_1$ are finite, we may find some ordinal $\gamma$, with $\beta\leq \gamma<\rho$, so that for all $q\in F_0\cup F_1$ with $r \sqsubsetneq q$ we have that $\mathrm{ht}(q)\leq \gamma$. In particular, $q(\gamma)$ is defined for all $q$ as above 
 ---recall  that $q$ is a  map $[\xi,\mu) \to \mathbb{Z}$ for some $\xi<\mu$. Similarly, let 
 \[\delta:=\max\{\mathrm{ht}(t)\colon  t\in F_{i-1} \text{ with }  r \sqsubsetneq t \}.\]

Choose any $g\in \mathrm{Aut}(L(r))$ so that  the following two properties hold:
\begin{enumerate}
\item if $q\in F_0\cup F_1$ with $r \sqsubsetneq q$ and $c\in C$ then $\big(g(c)\big)(\gamma)\neq q(\gamma)$;
\item if $q\in F_{i-1}$ with $r \sqsubsetneq q$ and $c\in C$, then $\mathrm{ht}(g(c)\wedge q)>\beta\oplus \delta$
\end{enumerate}
To find such $g$ notice first that $\beta\oplus \delta<\rho$. Indeed, since $s\in F_i$, by property (P2) for $(F_0,F_1)$, if $t \in F_{i-1}$ with $ r \sqsubsetneq t $ we have that $\alpha \oplus \mathrm{ht}(t)<\rho+1$. It follows that $\alpha \oplus \delta<\rho+1$. On the other hand we know that $\beta<\alpha$. By property (3) of Lemma \ref{L:PropertiesOfHessenberg} we have that
\begin{equation}
\beta\oplus \delta \oplus 1\leq \alpha\oplus \delta <\alpha\oplus \delta\oplus 1 \leq \rho \oplus 1, \quad \text{ which implies that } \quad \beta\oplus \delta < \rho.
\end{equation} 
Notice now that since both $\beta\oplus \delta < \rho$ and  $\gamma<\rho$ hold, for any $n\in \mathbb{Z}$ we have a well-defined automorphism  $\widehat{n}\in \mathrm{Aut}(L(r))$ given by sending every $c\in L(r)$ to $\widehat{n}(c) \in L(r)$ with:
\[\big(\widehat{n}(c)\big)(\xi)
\begin{cases}
c(\xi)-n & \text{ if } \xi \in \{\gamma,  \beta\oplus \delta\},\\
c(\xi) & \text{ otherwise}. \end{cases}\]
Since $F_0,F_1$ and $C$ are all finite, to find $g\in\mathrm{Aut}(L(r))$ for which both properties (1), (2) above hold, we can just set  $g:=\widehat{n}$ for some  large enough $n\in\mathbb{Z}$.

We may now conclude with the proof that  $\mathcal{R}_{\alpha}(s,r,i,\beta,C)$ intersects the basic open set $\mathcal{P}(F_0,F_1)$ as follows.  Choose any $g\in\mathrm{Aut}(L(r))$ which satisfies (1), (2) above and let $(E_0,E_1)$ be the pair given by  $E_i:=F_i\cup\{\big(g(c)\big)(\beta)\colon c\in C\}$ and $E_{i-1}:=F_{i-1}$. Clearly $\mathcal{P}(E_0,E_1)\subseteq \mathcal{P}(F_0,F_1)$ and if $(P_0,P_1)\in \mathcal{P}(E_0,E_1)$ we have that $(P_0,P_1)\in\mathcal{R}_{\alpha}(s,r,i,\beta,C)$. So it suffices to show that $\mathcal{P}(E_0,E_1)\neq \emptyset$, or equivalently by Lemma \ref{Lemma:F_implies_not_empty}, that $(E_0,E_1)\in\mathcal{F}$
 
But $E_0\cup E_1$ is an antichain by (1) above, and since all elements of $E_0\setminus F_0$ have same height. Again by (1) above we have that  $E_0\cap E_1 =\emptyset$. Let now $q\in E_i\setminus F_i$ and $t\in E_{i-1}$. The fact that  $\mathrm{ht}(q) \oplus \mathrm{ht}(t)<\mathrm{ht}(q\wedge t)$ holds follows from (2) above if $r \sqsubseteq t$. It also holds from simple computations when  $r \not\sqsubseteq t$. To see the latter, break  $r \not\sqsubseteq t$ into the cases $q\wedge t= r\wedge s$ and $q\wedge t\sqsubsetneq r\wedge s$ and then use, in each case, that   $\mathrm{ht}(s) \oplus \mathrm{ht}(t)<\mathrm{ht}(s\wedge t)$. We leave the details of confirming these cases to the reader.
\end{proof}
 
Hence, in order to show that $\mathcal{R}_\alpha$ is comeager, we are left to show that:
\[\mathcal{R}_{\alpha}\supseteq \bigcap_{s,r,i,\beta,C}\mathcal{R}_{\alpha}(s,r,i,\beta,C),\]
where the intersection is taken over
 all possible parameters; see paragraph before the  claim. So fix some $(P_0,P_1)$ in the above intersection. We will show $(P_0,P_1)\in \mathcal{R}_\alpha $.

Let $i\in \{0,1\}$,  $\beta<\alpha$, and  $A,B\subseteq L$ finite, so that for every $a\in A$  is covered by some $s\in P_i$ with $\mathrm{ht}(s)\geq \alpha$.
Let $S$ be the set of all $s\in P_i$ for which there is some $a\in A$ with  $s \sqsubseteq a$. Fix any $b\in B\setminus L(S)$ and notice that the set $\{s\wedge b \colon s\in S \}$ is linearly ordered by $\sqsubseteq$.  Let $\widetilde{r}_b$ be the $\sqsubseteq$--maximal element of this set, i.e. the element of this set of least height.
Let also $r_b$  be the  immediate extension of $\widetilde{r}_b$ which covers $b$. That is, $r_b$ is the unique element of $T$ with $\widetilde{r}_b\sqsubseteq r_b\sqsubseteq b$ and $\mathrm{ht}(\widetilde{r}_b)=\mathrm{ht}(r_b)+1$. This is well defined, as $b\not\in A$ implies   $\mathrm{ht}(\widetilde{r}_b)>0$.

Set now $R:=\{r_b\colon b\in B\}$ and notice that $R$ is an antichain. Indeed, if $b,c\in B \setminus L(S)$, say with $r_b \sqsubsetneq r_c$, then let $s\in S$ so that $r_c$ is the immediate extension of $s\wedge c$ and notice that $r_b$ covers both $s$ and $b$, contradicting that $r_b$ is an extension of  the element of least height in $\{ t\wedge b  \colon t\in S\}$.  The desired automorphism will be defined as the automorphism induced by a system $(g_r\colon r\in R)$ of automorphisms $g_r\in \mathrm{Aut}(L(r))$.

To define the system, fix some $r\in R$ and  let $\widetilde{r}$  be the immediate predecessor of $r$.  Set  $B_r:=B\cap L(r)$ and let $s_r$ be any element of $S$ so that for all $b\in B_r$ we have that $b\wedge s_r=\widetilde{r}$.  
Set $\rho_r:=\mathrm{ht}(r)$. Since $(P_0,P_1) \in  \mathcal{R}_{\alpha}(s_r,r,i,\beta,B_r)$ we get some $g_r\in\mathrm{Aut}(L(r))$ so that for all $b\in B_r$, $g_r(b)$ is covered by some $t\in P_i$ with $\mathrm{ht}(t)\geq \beta$. 

Let $g$ be the automorphism of $L$ induced by the system $(g_r\colon r\in R)$. Since $r\perp s$ for every $r\in R$ and $s\in S$, $g$ pointwise fixes every element of $T(S)$. It follows that $g\in \mathrm{Aut}(L)_A$. Moreover, for every $b\in B$, $g(b)$ is covered by some $t_b\in P_i$ with $\mathrm{ht}(t_b)\geq \beta$.  Indeed, if $b\in B\setminus L(S)$,  then this follows from the definition of $g_r$. If $b\in L(S)$, then we can set $t$ to be the unique $s\in S$ which covers $b$ and observe that, in this case, $g(s)=s$ and $g(b)=b$.
\end{proof}

\subsection{The fusion map $\pi$} Fix now  some $(P_0,P_1)\in\mathcal{P}$ so that both $P_0$, $P_1$  satisfy property $(\star_{\alpha})$ for all $\alpha<\mu$. By Lemma \ref{L:GenericPairHasStar} such pair exists. Let also $G\curvearrowright X$ be a Polish $G$-space. Define the ``fusion" map  $\pi\colon X^{L}\times X^{L}\to X^{L}$ as follows:
\[\pi(x_0,x_1)=x ,  \text{ where for all } a\in L \text{ we have that } \; x(a)=
\begin{cases}
 x_0(a),  & \text{ if }  a \in L(P_0)  \\
 x_1(a),  &  \text{ if }   a \in L(P_1) 
\end{cases}\]

As in the case of Claim \ref{Aris:Ind0Claim1} it straightforward to see that the map $\pi$ is continuous open and surjective.
We claim that $\pi$  also satisfies the conclusion of Lemma \ref{L:Ind}.

Fix some $x=(x_a)_a$, $y=(y_a)_a$  in $X^{L}$ and assume that for a countable ordinal  $\nu$ we have:
\begin{enumerate}
\item  $x_a \leftrightsquigarrow^{\gamma}_G y_a$ holds  for all $a\in L$ and all $\gamma<\nu$; and 
\item the orbits of $x_a$ and of $y_a$ are dense in $X$ for all $a\in L$.
\end{enumerate}
We will show that for  every $\beta<\mu $ we have that:
\[ x  \; \leftrightsquigarrow^{\nu+\beta}_{\mathrm{Aut}(L) \mathrm{Wr} \, G}  \; \pi(x,y) \quad  \text{ and } \quad \pi(x,y)   \;  \leftrightsquigarrow^{\nu+\beta}_{\mathrm{Aut}(L) \mathrm{Wr}\, G}  \; y.\]

Set  $z:=\pi(x,y)$. Since the argument below is symmetric, we will only prove that:

\begin{claim}\label{C:Prefinal}
For every $\beta<\mu$ we have that  $x   \leftrightsquigarrow^{\nu+\beta}_{\mathrm{Aut}(L) \mathrm{Wr} \, G}   z$.
\end{claim}

This  follows from a  more general claim which is needed in order to  run the necessary induction on $\alpha$. First recall that elements of  $\mathrm{Aut}(L) \mathrm{Wr} \, G$  are pairs $(\varphi,(g_a)_a)$ where $\varphi\in \mathrm{Aut}(L)$ and $g_a\in G$ for all $a\in L$. In particular,  $\mathrm{Aut}(L) \mathrm{Wr} \, G$ admits a basis of open neighborhoods of the identity of the form:
\[V:=\{(\varphi,(h_a)_a)\colon \varphi\in  \mathrm{Aut}(L) _A \text{ and } h_a\in V_a \text{ for all } a\in A\},\]
where $A$ ranges over all finite subsets of $L$ and $V_a$ range over all  open neighborhoods of $1$ in $G$. Let $\mathcal{V}$ denote the collection of all such neighborhoods. For each $V$ as above we set $\mathrm{supp}(V):=A$ and we say that $(V_a\colon a\in A)$ are the associated  fibers of $V$.  Claim \ref{C:Prefinal}   follows  from the next  by setting $V:=\mathrm{Aut}(L) \mathrm{Wr}\, G$,  $h:=1_{\mathrm{Aut}(L) \mathrm{Wr}\, G}$ and say $\alpha:=\beta+1$.

\begin{claim}\label{C:Final}
Let $\beta<\alpha<\mu$, let $V\in\mathcal{V}$ and $h=(\varphi,(h_a)_a)\in \mathrm{Aut}(L) \mathrm{Wr}\, G$. 
If for every $a\in\mathrm{supp}(V)$ there exists  $s_a\in P_0$ so that $\mathrm{ht}(s_a)\geq \alpha$ and $\varphi^{-1}(a)\sqsupseteq s_a$,
then  $h x   \leftrightsquigarrow^{\nu+\beta}_{V}   h z$.
\end{claim}
\begin{proof}
We prove the claim by induction on $\alpha$. For $\alpha=1$ it follows from Lemma \ref{Aris:LProd}. Assume now that it holds for all ordinals  below $\alpha$ and let $V,h$ as in the statement.   
We will show that  $h x   \leftrightsquigarrow^{\nu+\beta}_{V}   h z$ holds.

Let  $W\subseteq_1  \mathrm{Aut}(L) \mathrm{Wr}\, G$ and $U\subseteq X$ be  open with $hz\in U$ (the case $hx\in U$ is similar). After shrinking $U,W$ if necessary, we may assume that $W\in\mathcal{V}$ and that for all $c\in\mathrm{supp}(W)$ there exists an  open  $U_{c}\subseteq X$ so that $U=\{x \in X^{L}\colon x_c \in U_c \text{ for all } c\in \mathrm{supp}(W) \}$. Set 
\[A:=\{ \varphi^{-1}(c) \colon c\in \mathrm{supp}(V) \} \quad \text{ and } \quad   B:=\{ \varphi^{-1}(c) \colon c\in \mathrm{supp}(W) \}.\]

Since the map $\pi$ is induced by some $(P_0,P_1)\in\mathcal{P}$ which satisfies $(\star_{\alpha})$, we may find some $\psi\in\mathrm{Aut}(L)_A$ so that for every $b\in B$ there is $t_b\in P_0$  with $\mathrm{ht}(t_b)\geq \beta$ so that $\psi^{-1}(b)$ is covered by  $t_b$. Since
 the orbits of $x_a$ and $z_a$ are dense in $X$ for all $a\in L$,  we may choose some sequence $(g_a)_{a}$ in $G^L$ so that,  setting $g:=(\varphi \psi \varphi^{-1}, (g_a)_a)$, we have that 
\begin{equation}
ghx,ghz\in U 
\end{equation}
In fact, since $(hx)_c=(hz)_c$ for all $c\in\mathrm{supp}(V)$ and  $\varphi \psi \varphi^{-1}\in\mathrm{Aut}(L)_{\mathrm{supp}(V)}$, we may assume without loss of generality that $g_c=1_G$ for all $c\in \mathrm{supp}(V)$. It follows that $g\in V$.

We are left with checking that  $gh x  \leftrightsquigarrow^{\nu+\gamma}_{W}  gh z $ holds for all $\gamma<\beta$. But this follows by induction hypothesis. Indeed, notice that $gh$ is of the form $( \varphi \psi, (f_a)_a)$ for some choice of  $(f_a)_{a\in L} \in G^L$ and that for all $c\in \mathrm{supp}(W)$ we have chosen  $\psi$ so that there exists $t_c\in P_0$  with $\mathrm{ht}(t_c)\geq \beta$ so that $\psi^{-1}(\varphi^{-1}(c))$ is covered by  $t_c$. 
\end{proof}

The proof of Lemma \ref{L:Ind} and hence of Theorem \ref{T:mainBernoulli} is now complete.


\printbibliography

\end{document}